\documentclass[a4paper,11pt,twoside]{article}
\usepackage{amsmath}
\usepackage[utf8]{inputenc}
\usepackage[T1]{fontenc}
\usepackage{natbib}
\usepackage{hyperref}
 \hypersetup{colorlinks, citecolor=blue, filecolor=blue, linkcolor=blue, urlcolor=blue }
\usepackage{titlesec}
\titleformat{\section}{\bfseries\sffamily}{\thesection.}{10pt}{\bfseries\sffamily}
\titleformat{\subsection}{\sffamily}{\thesubsection.}{10pt}{\sffamily}
\titleformat{\subsubsection}{\it}{\thesubsubsection.}{10pt}{\it}

\usepackage{amsmath}
\usepackage{amssymb}
\usepackage{xspace}
\usepackage{newtxtext}
\usepackage{newtxmath}
\usepackage{amssymb}
\usepackage{amsfonts}
\usepackage{graphicx}
\usepackage[font+=scriptsize,skip=0ex]{subcaption}

\usepackage{amsthm}
\usepackage{mathtools}
\usepackage{siunitx}
\usepackage{dsfont}
\usepackage{etoolbox,refcount}
\usepackage{multicol}
\usepackage{booktabs,multirow}
\usepackage{enumerate}
\usepackage[inline]{enumitem}
\usepackage{setspace}
\usepackage{centernot}
\usepackage{xfrac}
\usepackage{multirow}
\usepackage{caption}
\usepackage{algorithm}
\usepackage{algpseudocode}
\usepackage{amsopn}
\usepackage{longtable}
\usepackage{supertabular}
\usepackage[titletoc,title]{appendix}
\usepackage{abstract}
\usepackage{extarrows}
\usepackage[all]{xy}
\input xy
\xyoption{all}
\usepackage{tikz}
\usepackage{pgfplots}
\usepackage{environ}
\usepackage{tikz-3dplot}
\usetikzlibrary{arrows}
\usetikzlibrary{3d}
\tikzset{>=latex}
\usepackage{fancyvrb}
\usetikzlibrary{hobby,decorations.pathreplacing}
\usepackage{listings}
\makeatletter
\usepackage[bottom]{footmisc}

\makeatother
\pgfplotsset{grid style={gray}}
\pgfplotsset{minor grid style={ultra thin, loosely dashed , gray}}
\pgfplotsset{major grid style={ultra thin, gray}}
\makeatletter
\newsavebox{\measure@tikzpicture}
\NewEnviron{scaletikzpicturetowidth}[1]{%
  \def\tikz@width{#1}%
  \def\tikzscale{1}\begin{lrbox}{\measure@tikzpicture}%
  \BODY
  \end{lrbox}%
  \pgfmathparse{#1/\wd\measure@tikzpicture}%
  \edef\tikzscale{\pgfmathresult}%
  \BODY
}
\makeatother
\makeatletter
\renewcommand{\maketitle}{\bgroup\setlength{\parindent}{0pt}
\begin{flushleft}
  \hspace{27.5pt}\textbf{\@title}

  \hspace{27.5pt}\@author
\end{flushleft}\egroup
}
\makeatother
\usepackage{cleveref}
\crefname{section}{section}{Sections}
\crefname{subsection}{subsection}{Subsections}
\crefname{subsubsection}{subsubsection}{Subsubsections}
\crefname{theorem}{theorem}{Theorems}
\crefname{corollary}{corollary}{Corollaries}
\crefname{lemma}{lemma}{Lemmas}
\crefname{proposition}{proposition}{Propositions}
\crefname{definition}{definition}{Definitions}
\crefname{remark}{remark}{Remarks}
\crefname{table}{table}{Table}
\crefname{figure}{figure}{Figures}
\crefname{equation}{equation}{Equations}
\crefname{appsec}{Appendix}{Appendices}
\DeclareMathAlphabet{\mathpzc}{OT1}{pzc}{m}{it}

\newcommand{\Ss}{\mathpzc{S}}
\newcommand{\Hh}{\mathpzc{H}}
\newcommand{\X}{\mathpzc{X}}
\numberwithin{equation}{section}
\theoremstyle{plain}
\newtheorem{theorem}{Theorem}[section]

\newtheorem{proposition}{Proposition}[section]
\theoremstyle{definition}
\newtheorem{definition}{\emph{Definition}}[section]
\theoremstyle{remark}
\newtheorem{remark}{Remark}[section]
\newtheorem{example}{Example}[section]

\usepackage{chngcntr}
 \counterwithin{figure}{section}
 \counterwithin{table}{section}
 \counterwithin{algorithm}{section}
\usepackage{marginnote}
\usepackage[marginparwidth=2.8cm, marginparsep=.2cm]{geometry}
\renewcommand{\Pr}{\operatorname{\mathds{P}}}

\newcommand{\RR}{\mathds{R}}

\newcommand{\bbG}{\mathbb{G}}
\newcommand{\bbF}{\mathbb{F}}
\newcommand{\bbK}{\mathbb{K}}       \newcommand{\wbbK}{\w{\mathbb{K}}}

\newcommand{\bbP}{\mathbb{P}}
\newcommand{\SSS}{\mathds{S}}
\newcommand{\one}[1]{\mathds{1}_{\{#1\}}}

\newcommand{\E}{\operatorname{\mathds{E}}}

\newcommand{\diag}{\operatorname{diag}}

\newcommand{\CA}{\mathcal{A}}
\newcommand{\CCC}{\mathcal{C}}

\newcommand{\CX}{\mathcal{X}}

\newcommand{\CD}{\mathcal{D}}
\newcommand{\CN}{\mathcal{N}}
\newcommand{\CP}{\mathcal{P}}
\newcommand{\CR}{\mathcal{R}}

\newcommand{\ROC}{ROC\xspace}
\newcommand{\dHSIC}{dHSIC\xspace}

\newcommand{\AUK}{\mathrm{AUK}}
\newcommand{\hAUK}{\h{\AUK}}
\newcommand{\wAUK}{\w{\AUK}}
\newcommand{\Frank}{\mathrm{Frank}}
\newcommand{\Clayton}{\mathrm{Clayton}}
\newcommand{\Joe}{\mathrm{Joe}}
\newcommand{\FGM}{\textrm{F-G-M}}
\newcommand{\Morgenstern}{\textrm{Morgenstern}}
\newcommand{\Plackett}{\textrm{Plackett}}
\newcommand{\ALIHAQ}{\textrm{ALI-HAQ's}}
\newcommand{\Gumbel}{\textrm{Gumbel}}

\newcommand{\hT}{\h{T}}

\newcommand{\hI}{\h{I}}
\newcommand{\hxi}{\h{\xi}}
\newcommand{\wC}{\w{C}}
\newcommand{\wT}{\w{T}}               \newcommand{\cT}{\check{T}}
\newcommand{\bx}{\bm{x}}              \newcommand{\bX}{\bm{X}}              \newcommand{\wbX}{\w{\bm{X}}}              \newcommand{\cbX}{\check{\bm{X}}}
\newcommand{\by}{\bm{y}}              \newcommand{\bY}{\bm{Y}}
\newcommand{\bc}{\bm{c}}              \newcommand{\bC}{\bm{C}}
              \newcommand{\bD}{\bm{D}}
\newcommand{\bs}{\bm{s}}              
\newcommand{\bu}{\bm{u}}              \newcommand{\bU}{\bm{U}}
     \newcommand{\bDelta}{\bm{\varDelta}}
\newcommand{\brho}{\bm{\rho}}

\newcommand{\bZ}{\bm{Z}}

\newcommand{\BX}{\mathbf{X}}

\newcommand{\bzero}{\mathbf{0}}
\newcommand{\Erlang}{\mathrm{Erlang}}
\newcommand{\Fa}{F_{\mathbf{1}}}
\newcommand{\Fb}{F_{\mathbf{2}}}

\newcommand{\vP}{\varPi}
\newcommand{\as}{\xlongrightarrow{\textrm{\begin{picture}(0,0)\put(0,0){\makebox(-1.75,-.5){$\mathrm{as}$}}\end{picture}}}}

\newcommand{\eqd}{\xlongequal{\textrm{\begin{picture}(2,0)\put(1,0){\makebox(0,1.5){$\mathrm{d}$}}\end{picture}}}}
\newcommand{\SIG}{\mathrel{\mathpalette\SIS\relax}}
  \newcommand{\SIS}[2]{%
    \ooalign{$#1\mathbf{\Sigma}$\cr\hfil$#1|$\hfil\cr}}
\newcommand{\Sig}{{\SIG}}%\newcommand{\Sig}{\operatorname{\SIG}}
\newcommand{\h}[1]{\hat{#1}}
\newcommand{\w}[1]{\tilde{#1}}
\newcommand{\bm}[1]{\boldsymbol{#1}}
\newcommand{\ud}[1]{\operatorname{d\textrm{$#1$}}}
\newcommand{\FIG}[2]{#2}
\definecolor{rlogo}{RGB}{126,142,176}
\definecolor{Blue}{rgb}{0,0,.6}
\definecolor{Red}{rgb}{0.8,0,0}
\definecolor{Green}{rgb}{0,.4,.4}
\let\originalleft\left
\let\originalright\right
\renewcommand{\left}{\mathopen{}\mathclose\bgroup\originalleft}
\renewcommand{\right}{\aftergroup\egroup\originalright}
\newcommand\blfootnote[1]{%
  \begingroup
  \renewcommand\thefootnote{}\footnote{#1}%
  \addtocounter{footnote}{-1}%
  \endgroup}
\renewcommand*{\thefootnote}{\fnsymbol{footnote}}
 \title{\begin{center} 
 		\begin{minipage}[a]{.9\linewidth}
        \begin{center}
        	\Large\bf An AUK-based index for measuring and testing the joint dependence of a random vector	
        \end{center}
        \end{minipage}
        \end{center}}
 \author{
 \begin{itemize}[labelindent=0pt, topsep=5pt, parsep=3pt, partopsep=0pt]
   \item[]\large Georgios Afendras\footnote{Corresponding author}
                                  \blfootnote{E-mail addresses: \href{mailto:gafendra@math.auth.gr}{\tt gafendra@math.auth.gr} (Georgios Afendras);
                                                                \href{mailto:markatou@buffalo.edu}{\tt markatou@buffalo.edu} (Marianthi Markatou);
                                                                \href{mailto:avexler@buffalo.edu}{\tt avexler@buffalo.edu} (Albert Vexler)} \\
          \small\it Department of Mathematics, Aristotle University of Thessaloniki, Thessaloniki, Greece
   \item[]\large\rm Marianthi Markatou\\
          \small\it Department of Biostatistics and Jacobs School of Medicine and Biomedical Sciences, The \linebreak State University of New York at Buffalo, Buffalo, NY, USA\\
   \item[]\large\rm Albert Vexler\\
          \small\it Department of Biostatistics, The State University of New York at Buffalo, Buffalo, NY, USA
 \end{itemize}
 }
 \date{}

\newcommand{\runtitle}{An AUK-based index for measuring and testing the joint dependence of a random vector}
\newcommand{\authors}{G.~Afendras, M.~Markatou and A.~Vexler}

\usepackage{fancyhdr}
\fancypagestyle{firstpage}{%
  \lhead{}%\large\color{blue} {\sc Final Manuscript}, \emph{\today}}%\footnotesize{\it Submitted to} \sffamily J.\ R.\ Statist.\ Soc.\ B}
  \rhead{}
}

\begin{document}
\maketitle
\thispagestyle{firstpage}
 \begin{abstract}
\noindent\makebox[\linewidth]{\rule{\linewidth}{0.4pt}}
 \noindent
 {\bf Abstract.}
We present an index of dependence that allows one to measure the joint or mutual dependence of a $d$-dimensional random vector with $d>2$. The index is based on a $d$-dimensional Kendall process. We further propose a standardized version of our index of dependence that is easy to interpret, and provide an algorithm for its computation. We discuss tests of total independence based on consistent estimates of the area under the Kendall curve. We evaluate the performance of our procedures via simulation, and apply our methods to a real data set.
\vspace*{-2ex}

\noindent\makebox[\linewidth]{\rule{\linewidth}{0.4pt}}
\end{abstract}

\section{Introduction and Literature Review}
\label{sec.intr}

Advances in technology allow the generation and collection of very large data sets, where the dependencies among the different variables that comprise these data sets are often of considerable interest. In this paper, we consider ways to measure joint or mutual dependence among the components of a $d$-dimensional random vector and to construct test procedures for mutual or total independence.

There is substantial literature on the construction of indices of dependence between two random variables. These include the relevant concept of correlation, for example \citet{Pearson1895} and \citeauthor{Spearman1904}'s rho (\citeyear{Spearman1904}). More recently, \citet{SzRB2007} and \citet{SzR2009,SzR2014} proposed the concept of distance correlation, that represents one of the most important breakthroughs in measuring dependence between two random vectors. Furthermore, a generalization of the concept of distance correlation was recently proposed by \citet*{SPV2020}. Other proposals include a projection correlation \citep{ZXLZ2017,BD2014,GNRM2019}.

We propose an index of dependence that is based on the area under the Kendall curve, generalizing work by \citet{VAM2019} from the two-dimensional to the $d$-dimensional case. The proposed index measures the dependence between the $d$ components of a random vector. We then construct a test of total independence based on the estimated area under the Kendall curve. This task is different from the task of testing pairwise independence in a collection of random variables.

The literature in statistics includes a few nonparametric procedures to test mutual or joint dependence. These include \citet{Simon1977} who proposed a nonparametric test based on Kendall's tau. Other tests include \citet{Kankainen1995} and \citet{Karvanen2005}. The later author developed a restamping test for testing total independence that is based on the test proposed by \citeauthor{Kankainen1995}, for the case of stationary time series.

\citet{PBSP2018} investigated also the problem of testing mutual or joint independence and proposed tests based on the $d$-variable Hilbert-Schmidt independence criterion, \dHSIC, thereby extending \citet{Kankainen1995} to the flexible framework of kernel methods. These authors developed their test based on the idea of probability embeddings of marginals and the joint distribution of the random vector $\bX$.

Our contribution is two-fold. First, we propose n algorithm to construct an index of dependence that is based on the Euclidean distance between the vector of appropriately defined areas under the Kendall curve and the $d$-dimensional vector with components equal to $1/2$. We also present the standardized version of our index of dependence that is easily interpretable, and we note that these indices are copula-based.

Next, we propose a practical, fast to compute test of mutual independence that is also based on the estimate of the area under the Kendall curve. We discuss the asymptotic distribution of our test statistic under the null hypothesis of mutual independence and investigate the performance characteristics of our test via simulation. Furthermore, we provide comparisons with the \dHSIC test. The results of our study show that the proposed testing procedure is practical and easily computable --- the computation time required to carry out our test is considerably less than the time required to compute the corresponding \dHSIC test. Furthermore, the newly proposed test outperforms in terms of power the \dHSIC test in smaller samples and performs as the \dHSIC test in large samples. The power results apply to all distributions from where samples were generated with the exception of the distributions $U\{C(0,1)\}$ and spiral 1 and 3 distributions.

The paper is organized as follows. Section \ref{sec.motiv} offers motivation for constructing the index of dependence and presents discussion on appropriate axioms that a measure of dependence must satisfy. Section \ref{sec.bivariate} briefly presents the bivariate case, while Section \ref{sec.extension} extends the index to the case where the dimension $d$ is greater than 2. This section also offers an algorithm for the computation of a standardized version of our index. Section \ref{sec.axioms} proves that the proposed index satisfies the axioms a function must satisfy in order to be a measure of dependence, while Section \ref{sec.empirical} discuss the estimation of the area under the Kendal curve and the two indices of dependence. Section \ref{sec.test} introduces the test of mutual independence, while Sections \ref{sec.Simulation} and \ref{sec.RealData} present simulation results and a real data example. Section \ref{sec.Discussion} concludes with a short summary and discussion of the results. Proofs of all relevant theorems are given in the appendix.

\section{Motivation, Notation and Dependence Axioms}
\label{sec.motiv}

Let $X_1,\ldots,X_d$ be a set of $d$ unidimensional random variables (rvs). Let $d_1$ and $d_2$ be two positive integers such that $d_1+d_2=d$; consider a partitioning of $X$s into two sets of dimensions $d_1$ and $d_2$, say $X_1,\ldots,X_{d_1}$ and $X_{d_1+1},\ldots,X_{d}$.  Consider the random vectors $\bX_1=(X_1,\ldots,X_{d_1})^T$, $\bX_2=(X_{d_1+1},\ldots,X_{d})^T$ and $\bX=(\bX_1^T,\bX_2^T)^T=(X_1,\ldots,X_d)^T$ of dimension $d_1$, $d_2$ and $d$ respectively. The cumulative distribution functions (cdfs) of $X_j$, $\bX_1$, $\bX_2$ and $\bX$ are denoted by $F_j$, $\Fa$, $\Fb$ and $F$ respectively. Hereafter, we denote $\bx_1=(x_1,\ldots,x_{d_1})^T\in\RR^{d_1}$, $\bx_2=(x_{d_1+1},\ldots,x_{d})^T\in\RR^{d_2}$ and $\bx=(x_1,\ldots,x_d)^T\in\RR^{d}$ being the corresponding realizations of $\bX_1$, $\bX_2$ and $\bX$. Also, whenever we say that $\bX$ (or $F$) is a continuous random vector (or cdf) on $\RR^d$ we mean that it has continuous marginal distributions --- that do not necessarily have a density function. Hence, we define the next class of random vectors
\[
\CX^d_0\doteq\{\bX\colon \bX\textrm{ is a continuous random vector on } \RR^d\}.
\]
Suppose $\bX\in\CX^d_0$ and assume that we are interested in measuring the following two kinds of dependence, measured as distances between appropriate cdfs:
\begin{enumerate}[topsep=0ex, itemsep=0ex, labelindent=0pt, leftmargin=6ex, label=$\CD_{\arabic*}\colon$, ref=\textrm{\textcolor{black}{$\CD_{\arabic*}$}}]
 \item
 \label{D1}
 measuring the distance between $F(\bx)$ and $\Fa(\bx_1)\Fb(\bx_2)$;
 \item
 \label{D2}
 measuring the distance between $F(\bx)$ and $\prod_{j=1}^{d}F_j(x_j)$.
\end{enumerate}
In other words, in \ref{D1} we are interested in measuring the pairwise-dependence between $\bX_1$ and $\bX_2$; while in \ref{D2} we are interested in measuring the dependence structure of the joint distribution function of $\bX$. Thus, in \ref{D2} we aim to measure joint or mutual dependence. In the bivariate case, i.e.\ $d=2$ and so $d_1=d_2=1$, the statements \ref{D1} and \ref{D2} are identical; while these differ from one another when $d>2$.

We focus in the multivariate case where $d>2$. Consider the rvs $U_j$, $j=1,\ldots,d$, independent and identically distributed from $U(0,1)$ such that $\bX$ and $U_j$s are independent, and define the rvs
\begin{equation}\label{eq.T,Pi,E}
T\doteq F(\bX);
\quad
\vP\doteq U_1\cdots U_d;
\quad
E\doteq-\ln\vP.
\end{equation}

Let the components of a random vector $\bX$ be standardized (each marginal distribution has zero mean and variance one), and denote its variance/covariance matrix by $\Sig_d(\brho)$ with $\brho=(\rho_{ij})\in\RR^{n(n-1)/2}$, $1\le i<j\le d$, that is
\begin{subequations}\label{eq.Sig}
\begin{equation}\label{eq.Sig(a)}
\Sig_d(\brho)=(\sigma_{ij})_{d\times d} \ \textrm{ with } \sigma_{ii}=1 \textrm{ and } \sigma_{ij}=\sigma_{ji}=\rho_{ij} \textrm{ for } i<j.
\end{equation}
Specifically, for the cases where $\rho_{ij}=\rho$ for all $i<j$, we denote the variance/covariance matrix by $\Sig_d(\rho)$, that is,
\begin{equation}\label{eq.Sig(b)}
\Sig_d(\rho)=(\sigma_{ij})_{d\times d} \ \textrm{ with } \sigma_{ii}=1 \textrm{ and } \sigma_{ij}=\rho \textrm{ for } i\ne j.
\end{equation}
\end{subequations}

In \citeyear{Renyi1959}, \citeauthor{Renyi1959}'s fundamental article defined a set of axioms that a measure of dependence for a pair of random variables must satisfy; these axioms have been studied and suitably modified by \citet{SW1981}, see Axioms \ref{R1} to \ref{R8} in Appendix \ref{app:axioms}. In a recent article, \citet{MSz2019} propose four natural axioms for dependence measures, see Axioms \ref{MS1} to \ref{MS4} in Appendix \ref{app:axioms}.  \citeauthor{Renyi1959}'s axioms refer only to the 2-dimensional case; while those of \citeauthor{MSz2019} only refer to the case of the statement in \ref{D1}. Similar to \citet{MSz2019}, \citet*{HOPZ2019} establish axioms for dependence measures, that also, in general, refer to the statement in \ref{D1}. We discuss now a modification of \citeauthor{Renyi1959}'s axioms that is suitable for \ref{D2}.

Let $\X^d\subset\CX^d_0$. According to \citeauthor{Renyi1959}, a mapping $\delta\colon\X^d\to[0,1]$ is called a mutual dependence measure on $\X^d$ if
\begin{enumerate}[topsep=0ex, itemsep=0ex, labelindent=0pt, leftmargin=6ex, label=$A_{\arabic*}\colon$, ref=\textrm{\textcolor{black}{$A_{\arabic*}$}}]
 \item
 \label{A1}
 $\delta(\bX)$ is defined for any $\bX\in\X^d$.

 \item
 \label{A2}
 $\delta(\bX)=\delta(\bX^*)$ for any permutation $\bX^*=(X_{i_1},\ldots,X_{i_d})^T$ of $\bX\in\X^d$.

 \item
 \label{A3}
 $0\le\delta(\bX)\le1$ for any $\bX\in\X^d$.

 \item
 \label{A4}
 $\delta(\bX)=0$ if and only if $X_1,\ldots,X_d$ are independent rvs.

 \item
 \label{A5}
 For $\bX\in\X^d$, $\delta(\bX)=1$ if and only if each of the $X_i$s is a.s.\ a strictly monotone function of any other.

 \item
 \label{A6}
 For $\bX\in\X^d$, if $f_i$ is strictly monotone function a.s.\ on the range of $X_i$ for all $i=1,\ldots,d$, then $\delta\{f(\bX)\}=\delta(\bX)$, where $f(\bX)=(f_1(X_1),\ldots,f_d(X_d))^T\in\X^d$.

 \item
 \label{A7}
  If $d=2$ and the joint distribution of $\bX$ is bivariate normal, with correlation coefficient $\rho$, then $\delta(\bX)$ is a strictly increasing function of $|\rho|$.

 \item
 \label{A8}
 If $\bX$ and $\bX_n$, $n = 1, 2,\ldots$, belong to $\X^d$, and if $\bX_n\rightsquigarrow\bX$, then $\delta(\bX_n)\to\delta(\bX)$.
 \end{enumerate}

\begin{remark}
\label{rem.A7}
\begin{enumerate}[topsep=0ex, itemsep=0ex, wide, labelwidth=!, labelindent=0pt, label=\rm(\alph*), ref=\textrm{\textcolor{black}{\alph*}}]
\item\label{rem.A7(a)}
For the case $d=3$, consider the matrix given by \eqref{eq.Sig(b)}, i.e.\ $\Sig_3(\rho)=\left(\textrm{\tiny\begin{tabular}{@{}c@{~~}c@{~~}c@{}}1&$\rho$&$\rho$\\[-.3ex]$\rho$&1&$\rho$\\[-.3ex]$\rho$&$\rho$&1\end{tabular}}\right)$. The determinant of this matrix is $\det\left\{\Sig_3(\rho)\right\}=(1-\rho)^2(1+2\rho)$. Since a variance/covariance matrix is a positive definite matrix and so its determinant is non-negative, it follows that $\rho\ge-1/2$. Consequently, $\Sig_3(\rho)$ is a variance/covariance matrix if and only if $-1/2\le\rho\le1$. Based on the above analysis, there are positive values of $\rho$ for which the fact that $\Sig_3(\rho)$ can be assumed as a variance/covariance matrix of a 3-dimensional random vector does not imply that $\Sig_3(-\rho)$ can also be assumed as a variance/covariance matrix. For example, set $\rho=0.7$; then, $\Sig_3(0.7)$ is a positive definete $3\times3$ matrix and hence we can consider the normal distribution $\CN_3(\bzero,\Sig_d(0.7))$, while $\Sig_3(-0.7)$ is not a positive definete matrix and so the notation $\CN_3(\bzero,\Sig_d(-0.7))$ makes no sense.

In general, the determinant of $\Sig_d(\rho)$, given by \eqref{eq.Sig(b)}, is $\det\left\{\Sig_d(\rho)\right\}=(1-\rho)^{d-1}\{1+\rho(d-1)\}$. Therefore, this matrix is a variance/covariance matrix if and only if $(1-d)^{-1}\le\rho\le1$. Let $\bX\sim \CN_d(\bzero,\Sig_d(\rho))$; then, the acceptable values of $\rho$ are $(1-d)^{-1}\le\rho\le1$. If Axiom \ref{R7} can be modified for general $d$, then a measure of dependence will take the same value for both $\CN_d(\bzero,\Sig_d(\rho))$ and $\CN_d(\bzero,\Sig_d(-\rho))$ distributions for all $\rho$-values; on the other hand, for $d>2$ and each $(d-1)^{-1}<\rho\le1$ the notation $\CN_d(\bzero,\Sig_d(\rho))$ denotes the $d$-dimensional normal distribution with mean vector $\bzero$ and variance/covariance matrix $\Sig_d(\rho)$, while the notation $\CN_d(\bzero,\Sig_d(-\rho))$ does not correspond to any distribution.

\item\label{rem.A7(b)}
One may incorrectly suggest an extension of \ref{R7} as follows. For a $d$-dimensional normal distribution for fixed $\rho_{ij}$ with $(i,j)\ne(i_0,j_0)$, a measure of mutual dependence must be an increasing function of $|\rho_{i_0j_0}|$. To explain why the aforementioned statement is incorrect, we consider the 3-dimensional normal distributions $F\sim \CN_3(\bzero,\Sig_F)$ and $G\sim \CN_3(\bzero,\Sig_G)$, where the variance/covariance matrices are
$\Sig_F= \left(\textrm{\tiny\begin{tabular}{@{}c@{~~}c@{~~}c@{}}1&$-0.5$&$-0.5$\\[-.3ex]$-0.5$&1&$\phantom{-}0.5$\\[-.3ex]$-0.5$&$\phantom{-}0.5$&1\end{tabular}}\right)$ and $\Sig_G=\left(\textrm{\tiny\begin{tabular}{@{}c@{~~}c@{~~}c@{}}1&$-0.5$&$-0.5$\\[-.3ex]$-0.5$&1&$-0.5$\\[-.3ex]$-0.5$&$-0.5$&1\end{tabular}}\right)$. Suppose $\bX\sim G$. Because of $\det(\Sig_G)=0$, there is constant vector $\bc\in\RR^3$ such that $\bc^T\bX$ is a constant with probability one, that is, $\bX$ lies on a subset of $\RR^3$ of dimension 2. The support of $F$ is $S_{F}=\RR^3$ while the support of $G$ is $S_{G}\subsetneq\RR^3$ with $\dim(S_{G})=2$. Figures \ref{sfig.normal-5-55} and \ref{sfig.normal-5-5-5} present scatterplots based on random samples of size $1000$ from $F$ and $G$ respectively. It is clear that $G$ has a stronger dependence structure than $F$, and so it is an abnormal situation a measure of dependence to have the same value for $F$ and $G$.

\item\label{rem.A7(c)}
In view of \eqref{rem.A7(a)} and \eqref{rem.A7(b)}, such modification of Axiom \ref{R7} makes no sense for dimension $d>2$.
\end{enumerate}
\end{remark}

On the other hand, according to the definition by \citeauthor{MSz2019}, a mapping $\delta\colon\X^d\to[0,1]$ is called a structure dependence measure on $\X^d$ if
\begin{enumerate}[topsep=0ex, itemsep=0ex, labelindent=0pt, leftmargin=6ex, label=$A'_{\arabic*}\colon$, ref=\textrm{\textcolor{black}{$A'_{\arabic*}$}}]
 \item
 \label{B1}
 $\delta(\bX)=0$ if and only if $X_1,\ldots,X_n$ are independent rvs.

 \item
 \label{B2}
 $\delta(\bX)$ is invariant with respect to all similarity transformations of $\RR^d$; that is, $\delta(L\bX)=\delta(\bX)$ where $L$ is similarity transformation of $\RR^d$.

 \item
 \label{B3}
 $\delta(\bX)=1$ if and only if there is a similarity transformation $L$ such that $L\bX=(X_L,\ldots,X_L)^T\in\X^d$ with probability 1.

 \item
 \label{B4}
 $\delta(\cdot)$ is continuous; that is, if $\bX$ and $\bX_n$, $n = 1, 2,\ldots$, belong to $\X^d$ such that $\bX_n$ is uniform integrable and $\bX_n\rightsquigarrow\bX$, then $\delta(\bX_n)\to\delta(\bX)$.
\end{enumerate}

We note that \citet{MSz2019} as well as \citet{HOPZ2019} avoid to establish an axiom analogous to \ref{R7}. In Section \ref{sec.axioms} we prove that our proposed index satisfies both sets of axioms, the modified \citeauthor{Renyi1959} (\ref{A1}-\ref{A6}, \ref{A8}) and \citeauthor{MSz2019} (\ref{B1}-\ref{B4}).

\section{The bivariate case}
\label{sec.bivariate}

We aim to construct a measure of mutual dependence that satisfies axioms \ref{A1}-\ref{A8} discussed in the previous section. For this purpose, in this section we briefly discuss the bivariate case that has been studied by \citet{VAM2019}.

Consider the rvs $T$ and $\vP$ defined by \eqref{eq.T,Pi,E}. The Kendall distribution function is defined as $K(t)=\Pr(T\le t)$. Since the distribution function of $\vP$ that corresponds to the independent case is $K_{\vP}(t)=t-t\ln t$, $t\in[0,1]$, the Kendall-plot is the visualization of the curve $(K(t),t-t\ln t)$,  $t\in[0,1]$, and compared to the diagonal line $(t-t\ln t,t-t\ln t)$, $t\in[0,1]$, it provides a rich source of information reference the association between the components of $\bX$. In the light of an \ROC analysis, \citet{VAM2019} show that the area under the Kendall curve is $\AUK=-\int_0^1 K(t)\ln t\ud{t}$, and then that
\[
\AUK=\E(1-T+T\ln T),
\]
which is useful both, for direct evaluation and for simulation studies based on the $\AUK$; it is also used for nonparametric estimation of $\AUK$ based on a random sample of pairs from $F$.

\citet{VAM2019} offer motivation for the use of an extension of the Kendall plot to a multi-panel Kendall plot that consists of the ordinary Kendall plots of the four orthogonal rotations of $\bX$; in particular, these are the Kendall plots that correspond to the Kendall functions $K_i(t)=\Pr\{H_i(\bX)\le t\}$, $i=1,\ldots,4$, where
\begin{equation}\label{eq.Hs}
\begin{split}
H_1(x_1,x_2)&=\Pr(X_1\le x_1,X_2\le x_2),\quad H_2(x_1,x_2)=\Pr(X_1> x_1,X_2\le x_2),\\
H_3(x_1,x_2)&=\Pr(X_1\le x_1,X_2> x_2),  \quad H_4(x_1,x_2)=\Pr(X_1> x_1,X_2> x_2).
\end{split}
\end{equation}
Using these Kendall functions, they provide the four corresponding $\AUK$s ($\AUK_i$, $i=1,\ldots,4$). Based on $\bD=(\AUK_1,\ldots,\AUK_4)^T$ they establish an index of dependence between $X_1$ and $X_2$ by the relation
\[
I=(8/5)^{1/2}\|\bD-\bDelta\|,
\]
where $\bDelta=(1/2,\ldots,1/2)^T\in\RR^4$ and $\|\bx\|$ denotes the Euclidean norm of the vector $\bx$. They also present a standardized index of dependence based on $I$ that ensures a linear mapping with $|\rho|$ in the case of the bivariate normal distribution, and they give an approximation of this standardized index by
\[
I^*=\varphi_2(I),
\]
where $\varphi_2(t)=2.070t + 0.061t^2 - 2.471t^3 + 1.307t^4 + 0.033t^5$, $t\in[0,1]$.

In what follows, we will extend the index of dependence discussed above to the $d$-dimensional case, $d>2$.

\section{Extension to higher dimensional case}
\label{sec.extension}

Let $\bX$ be a $d$-dimensional random vector which has continuous components. For the dimension $d$ of $\bX$, consider the rvs $T$, $\vP$ and $E$ defined by \eqref{eq.T,Pi,E}. A natural extension from the bivariate case to $d$-dimensional case of the Kendall distribution function is given by $K(t)=\Pr(T\le t)$. When we are interested in measuring the dependence structure of $\bX$, according to \ref{D2}, we can plot this Kendall function against the cdf of $\vP$, which corresponds to the case of independence of the components of $\bX$; and then, we compare this curve with the diagonal line in $[0,1]\times[0,1]$. Namely, we compare the curve $(K(t),K_{\vP}(t))$ with the diagonal line $(K_{\vP}(t),K_{\vP}(t))$, $t\in[0,1]$. Both Kendall cdf, $K_{\vP}(\cdot)$, and probability density function (pdf), $k_{\vP}(\cdot)$, of $\vP$ may by computed in close forms. Specifically, as is evident $E\sim\Erlang(d,1)$,
\begin{equation}\label{eq.K,k}
 K_{\vP}(t)=t\sum_{j=0}^{d-1}\frac{(-1)^k}{k!}\ln^k t
 \ \ \text{and} \ \
 k_{\vP}(t)=\frac{(-1)^{d-1}}{(d-1)!}\ln^{d-1}t,
 \quad 0<t<1.
\end{equation}
Observe now that $2E$ follows the chi-square distribution with $2d$ degrees of freedom, denoted as $E\sim\chi^2_{2d}$. By definition of $E$, see \eqref{eq.T,Pi,E}, the cdf and pdf of $\vP$ can be presented through the cdf and pdf of $\chi^2_{2d}$; specifically,
\[
K_{\vP}(t)=1-F_{\chi^2_{2d}}(-2\ln t) \ \ \textrm{and} \ \ k_{\vP}(t)=\frac{2}{t}f_{\chi^2_{2d}}(-2\ln t), \quad t>0.
\]
The above formulas are appropriate for computation.

The area under the Kendall ($\AUK$) curve is $\int_0^1 K(t)\ud{K_{\vP}(t)}=\int_0^1 K(t)k_{\vP}(t)\ud{t}$, that is
\[
\AUK=\Pr(T\le \vP).
\]
Since $\Pr(T\le \vP)=\E\{\Pr(T\le \vP|\vP)\}=\E\{\Pr(T\le \vP|T)\}=\E\{1-\Pr(\vP<T|T)\}$,
%\begin{equation}\label{eq.AUK}
\[
\AUK=\E\left\{K(\vP)\right\}=\E\left\{1-K_{\vP}(T)\right\}.
\]
%\end{equation}
Note that the formulas of $K_{\vP}$ and $\AUK$ in bivariate case follow from the above corresponding formulas of the general case by setting $d=2$.

The distance between $\AUK$ and $1/2$ may itself be used as a measure of divergence between $F$ and $\prod_{j=1}^dF_j$. But, it cannot be used as a measure of dependence because it does not satisfy the axioms of a measure of dependence that are listed in Section \ref{sec.motiv}, Axioms \ref{A1} to \ref{A8} or \ref{B1} to \ref{B4}.

We now investigate the four Kendall functions that lead to this multi-panel plot in the bivariate case. This investigation offers understanding that is used in the extension of the work for the case of $d>2$. Observe that, since $X_1$ and $X_2$ are continuous, the functions $H_1,\ldots,H_4$ that are defined by \eqref{eq.Hs} are the cdfs of $\bX_1=(X_1,X_2)^T$, $\bX_2=(-X_1,X_2)^T$, $\bX_3=(-X_1,-X_2)^T$ and $\bX_4=(X_1,-X_2)^T$ respectively; hence, the Kendall distribution functions $K_1,\ldots,K_4$ are the corresponding ordinary Kendall distribution functions of $\bX_1,\ldots,\bX_4$. Therefore, the four rotations of a bivariate random vector are leading to the construction of the corresponding Kendall functions, and so to the construction of the $\AUK$-vector, $\bD=(\AUK_1,\ldots,\AUK_4)^T$, that can be presented in an algebraic way, through all possible multiplications of the components $X_1$ and $X_2$ with $\pm$. This representation provides us a natural algebraic manner to extend the $\AUK$-method in highest dimensional random vectors.

In view of the above analysis, we consider the set of the orthogonal ``rotations'' of $\bX$, say $\CR(\bX)$. Analytically, define the subset $\SSS^d$ of $\RR^d$ that contains all possible vectors with elements $\pm1$, that is
\[
\SSS^d\doteq\{-,+\}^d=\{\bs_j\}_{j=1}^{2^d}.
\]
%the cardinality of this set is $\card(\SSS^d)=2^d$. Therefore, it is convenient the use an index enumeration $1,\ldots,2^d$. Suppose an enumeration of the elements of $\SSS^d$,
%\[
%\SSS^d=\{\bs_j\}_{j=1}^{2^d}.
%\]
Then, the set of the orthogonal ``rotations'' of $\bX$ is
\[
\CR(\bX)\doteq\{\bX_j=\diag(\bs_j)\bX\colon \bs\in\SSS^d\},
\]
where $\diag(\bs_j)$ denotes the $d\times d$ diagonal matrix with diagonal elements the elements of $\bs_j$. For $\bX_j\in\CR(\bX)$, $j=1,\ldots,2^d$, define
\[
T_j\doteq F_{\bX_j}(\bX_j)
\quad\text{and}\quad
K_j(t)=\Pr(T_j\le t).
\]
The corresponding, under the Kendall curve, areas are
\[
\AUK_j%=\frac{(-1)^{d-1}}{(d-1)!}\int_0^1 K_j(t)\ln^{d-1}t\ud{t}
      =\Pr(T_j\le \vP)
      =\E\left\{1-K_{\vP}(T_j)\right\}.
\]
Define $\bD=(\AUK_1,\ldots,\AUK_{2^d})^T$ and $\bDelta=(1/2,\ldots,1/2)^T\in\RR^{2^d}$.

Let now $X$ be a continuous rv, and consider the random vectors $\bX_1=(X,\ldots,X)^T$ and $\bX_2=(X,\ldots,X,-X)^T$ in $\CX^d$. Write $\bX_2=f(\bX_1)$, where $f=(f_1,\ldots,f_d)^T$ with $f_i(x)=x$ for $i=1,\ldots,d-1$ and $f_d(x)=-x$. Obviously, $f_i$ is a strictly monotone function for all $i=1,\ldots,d$, and $f$ is a similarity transformation; so, for any measure of dependence $\delta$ it is required $\delta(\bX_1)=\delta(\bX_2)$. On the other hand, the Kendall functions of $\bX_1$ and $\bX_2$ are $K_1(t)\equiv K_{\bX_1}(t)=t$ and $K_2(t)\equiv K_{\bX_2}(t)=1$, $0\le t\le1$. By application of the formula $\AUK=\int_0^1 K(t)k_{\vP}(t)\ud{t}$, we get $\AUK_1=2^{-d}$ and $\AUK_2=1$ (for computational details see in the proof of \eqref{eq.c_d} in Appendix \ref{app:proofs}), and hence, the distances between $\AUK_{\bX_i}$ and $1/2$ (the value that corresponds to the independence case), $i=1,2$, are $1/2-2^{-d}$ and $1/2$ respectively. Because of Axiom \ref{A6} or Axiom \ref{B2}, the $\AUK$ of $\bX\in\CX^d$ cannot be used for measuring the dependence structure of $\bX$.

In general, due to \ref{A6} or \ref{B2}, $\delta(\bX_j)=\delta(\bX)$ for all $\bX_j\in\CR(\bX)$. Hence, we propose an $\AUK$-based index of dependence of the form $c_d\|\bD-\bDelta\|$, where the constant $c_d$ is chosen such that the index takes the value 1 whenever $\bX$ is of the form $(X,\ldots,X)^T$ for some continuous rv $X$. Let $X$ be a continuous rv and set $\bX=(X,\ldots,X)^T\in\CX^d$. Then, the set of random vectors $\CR(\bX)$ contains exactly two random vectors with all components being the same, $(X,\ldots,X)^T$ and $(-X,\ldots,-X)^T$. Each of $2^d-2$ remaining random vectors has components that are rvs $X$ and $-X$. Each of $(X,\ldots,X)^T$ and $(-X,\ldots,-X)^T$ has the same Kendall function $t$, $0\le t\le1$, while each of $2^d-2$ remaining random vectors has the same Kendall function $1$, $0\le t\le1$; therefore, the constant $c_d$ is (for computational details see in the proof of \eqref{eq.c_d} in Appendix \ref{app:proofs})
\begin{equation}\label{eq.c_d}
c_d=\left(2^{d-2}-2^{1-d}+2^{1-2d}\right)^{-1/2}.
\end{equation}
Thus, the proposed index of dependence takes the form
\begin{equation}\label{eq.I}
I=\left(2^{d-2}-2^{1-d}+2^{1-2d}\right)^{-1/2}\|\bD-\bDelta\|.
\end{equation}

Similar to the bivariate case, we are interested in finding the easily interpreted standardized index, $I_{\mathrm{st}}$. This index is a function of $I$ that maps $I\mapsto|\rho|$ whenever $\bX\sim \CN_d(\bzero,\Sig_d(\rho))$, $0\le\rho\le1$; namely, $I_{\mathrm{st}}=\phi_d(I)$. The index does not have a closed form expression even in the bivariate case, but it can be approximated by $I^*=\varphi_d(I)$.

For the trivariate normal distribution, $\bX\sim \CN_3(\bzero,\Sig_3(\rho))$, we compute the true value of $I$ for various values of $\rho$, see Table \ref{table.(r,I(r))} in Supplementary Material. In similar way as in \citet{VAM2019}, i.e.\ by applying the Lagrange interpolation polynomial formula, a polynomial approximation of $\phi_3$ is
\[
\varphi_3(t)=1.62t+4.45t^2-13.48t^3+12.13t^4-3.72t^5, \ t\in[0,1].
\]

The preceding approximation function $\varphi_3\colon[0,1]\to[0,1]$ is a strictly increasing function. Moreover, $\varphi_3$ is a perfect fit to the true $\phi_3$, see Figure \ref{sfig.phi3}; therefore, the approximate standardized index $I^*$ has an almost perfect fit to the standardized index $I_{\mathrm{st}}$, that is, it has an almost perfect fit to the identity function of $\rho$ when the underlying distribution is $\CN_3\left(\bzero,\Sig_3(\rho)\right)$ with $\rho\in[0,1]$, see Figure \ref{sfig.I}.

\begin{figure}[htp]
\begin{subfigure}[a]{.5\textwidth}
\centering
\FIG
{\resizebox{.7\linewidth}{!}{
\begin{tikzpicture}
\begin{axis}[x=75, y=75, xmin=-.1, xmax=1.1, ymin=-.1, ymax=1.1, %axis lines=middle,
             tick label style={fill=none, font=\scriptsize}, %clip=false,
             font=\scriptsize, xlabel={$t$}, xlabel style={yshift=5},
             legend entries={$\phi_3(t)$,$\varphi_3(t)$}, legend pos=south east, legend style={fill=yellow!20, draw=none, font=\scriptsize, row sep=-3pt}]
\addplot[red, hobby, ultra thin] coordinates {
(0	       ,0.00)
(0.02659993,0.05)
(0.05186133,0.10)
(0.07652715,0.15)
(0.10050050,0.20)
(0.12391750,0.25)
(0.14751990,0.30)
(0.17112230,0.35)
(0.19434300,0.40)
(0.21837800,0.45)
(0.24250160,0.50)
(0.26716900,0.55)
(0.29398300,0.60)
(0.32149380,0.65)
(0.35195040,0.70)
(0.38533590,0.75)
(0.42312600,0.80)
(0.46735520,0.85)
(0.52426650,0.90)
(0.60895190,0.95)
(0.66237760,0.97)
(0.70056830,0.98)
(0.75774610,0.99)
(0.8413990,.995)
(1,1)};
\addplot[color=blue, domain=0:1, dashed, samples=100,smooth] {(-3.751*\x*\x*\x*\x*\x+12.235*\x*\x*\x*\x-13.607*\x*\x*\x+4.513*\x*\x+1.61*\x)};
\addplot[black, draw=none, mark=*, mark options={solid, scale=.3}] coordinates {
(0,0)
(1,1)};
\end{axis}
\end{tikzpicture}
}}
{\includegraphics[width=.7\linewidth]{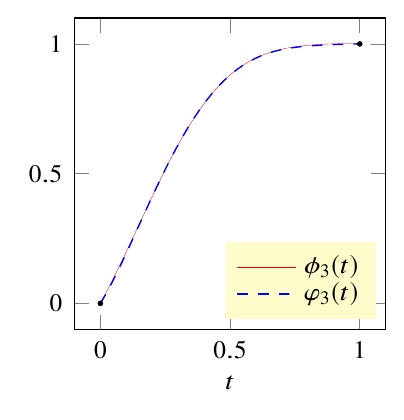}}
\caption{The plots of the functions $\phi_3$ and $\varphi_3$.}
\label{sfig.phi3}
\end{subfigure}
\hfill
\begin{subfigure}[a]{.5\textwidth}
\centering
\FIG
{\resizebox{.7\linewidth}{!}{
\begin{tikzpicture}
\begin{axis}[x=75, y=75, xmin=-.1, xmax=1.1, ymin=-.1, ymax=1.1, %axis lines=middle,
             tick label style={fill=none, font=\scriptsize}, %clip=false,
             font=\scriptsize, xlabel={$\rho$-values}, xlabel style={yshift=5},
             legend entries={$I(\rho)$,$I^*(\rho)$}, legend pos=north west, legend style={fill=yellow!20, draw=none, font=\scriptsize, row sep=-3pt}]
\addplot[red, hobby, mark=star, mark options={solid, scale=.6}] coordinates {
(0.00,0         )
(0.05,0.02659993)
(0.10,0.05186133)
(0.15,0.07652715)
(0.20,0.10050050)
(0.25,0.12391750)
(0.30,0.14751990)
(0.35,0.17112230)
(0.40,0.19434300)
(0.45,0.21837800)
(0.50,0.24250160)
(0.55,0.26716900)
(0.60,0.29398300)
(0.65,0.32149380)
(0.70,0.35195040)
(0.75,0.38533590)
(0.80,0.42312600)
(0.85,0.46735520)
(0.90,0.52426650)
(0.95,0.60895190)
(0.97,0.66237760)
(0.98,0.70056830)
(0.99,0.75774610)
(.995,0.8413990)
(1,1)};
\addplot[blue,hobby, mark=*, mark options={scale=.3}] coordinates {
(0.00,0          )
(0.05,0.045769067)
(0.10,0.093824008)
(0.15,0.143950147)
(0.20,0.194786074)
(0.25,0.245690398)
(0.30,0.297568627)
(0.35,0.349417333)
(0.40,0.399880290)
(0.45,0.451064924)
(0.50,0.500943175)
(0.55,0.550018159)
(0.60,0.600781558)
(0.65,0.649736164)
(0.70,0.699925930)
(0.75,0.749841237)
(0.80,0.799731096)
(0.85,0.849240483)
(0.90,0.899494070)
(0.95,0.949626564)
(0.97,0.969012626)
(0.98,0.978480541)
(0.99,0.987680219)
(.995,0.992269079)
(1,1)};
\addplot[black, draw=none, mark=*, mark options={solid, scale=.3}] coordinates {
(0,0)
(1,1)};
\end{axis}
\end{tikzpicture}
}}
{\includegraphics[width=.7\textwidth]{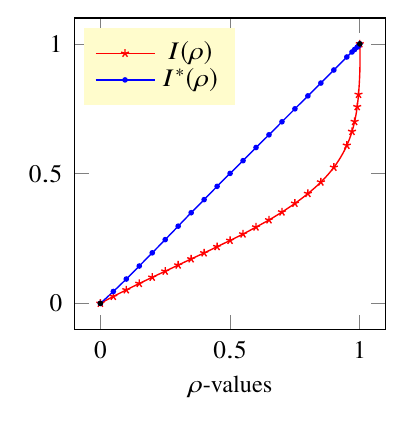}}
\caption{The plots of $I$ and $I^*$ in $\CN_3\left(\bzero,\Sig_3(\rho)\right)$ case.}
\label{sfig.I}
\end{subfigure}
\caption[The functions $\phi_3$ and $\varphi_3$, as well as the indices $I$ and $I^*$ in trivariate normal distribution.]{The functions $\phi_3$ and $\varphi_3$, as well as the indices $I$ and $I^*$ in $\CN_3\left(\bzero,\Sig_3(\rho)\right)$ case.}
\label{fig.N3}
\end{figure}

Both $\varphi_2\colon[0,1]\to[0,1]$ and $\varphi_3\colon[0,1]\to[0,1]$ are strictly increasing functions. The Lagrange interpolation polynomial formula, that has been used for the cases of dimension 2 and 3, often leads to find interpolating polynomials that are not strictly monotone; hence, it is wise to avoided for dimensions greater than 3. For any specific value of $d$, we propose the use the {\tt MonoPoly} R package \citep{TM2019} to find an appropriate, strictly monotone, polynomial for approximating the standardized index. The procedure is presented by Algorithm \ref{algorithm-phi}; in Section \ref{sm:R} of the Supplementary Material, we provide a realization of Step \ref{algorithm-phi.step2} of the algorithm.

\begin{algorithm}
\caption{Computation of the approximate standardized index $I^*$.}
 \label{algorithm-phi}
\small
\begin{algorithmic}[1]
\State\label{algorithm-phi.step1}
Let $0\le\rho\le1$ and consider $\bX_\rho\sim \CN_d(\bzero,\Sig_d(\rho))$ as in Remark \ref{rem.A7}, and consider the vector of $\rho$-values $\bm{\varrho}=(0,0.4,0.8,0.95,0.99,1)$. Then compute $\bm{I}\equiv\bm{I}(\bm{\varrho})\doteq(I(\varrho_1),\ldots,(\varrho_6))$.

\Comment{By definition, $I(0)=0$ and $I(1)=1$.

\State\label{algorithm-phi.step2}
Fit a monotone polynomial on $[0,1]$ of degree $7$, say $\varphi_d$, to the data set $\{(I(\varrho_j),\varrho_j), \ j=1,\ldots,6\}$.

\State\label{algorithm-phi.step3}
Define $I^*\doteq\varphi_d(I)$.
\end{algorithmic}
\end{algorithm}
}
\pagebreak

Based on the values of $I^*$ we define four levels of dependence.
\begin{definition}
\label{def.level}
Let $\bX\in\CX^d$. We say that $\bX$ has a weak, mild, strong or very strong dependence structure if $I^*$ belongs to $[0,0.25)$, $[0.25,0.5)$, $[0.5,0.75)$ or $[0.75,1]$ respectively.
\end{definition}

\begin{remark}
\label{rem.increasing}
Via numerical study, we observe that when the underlying distribution is $\CN_2\left(\bzero,\Sig_2(\rho)\right)$ \citep[see][]{VAM2019} or $\CN_3\left(\bzero,\Sig_3(\rho)\right)$ (see Table \ref{table.(r,I(r))}, Section \ref{sm:numerical} of the Supplementary Material) the index $I$ is an increasing function of $\rho\in[0,1]$. We conjecture that this is true whenever the underlying distribution is $\CN_d\left(\bzero,\Sig_d(\rho)\right)$ for any value of $d$.
\end{remark}

\section{Investigation of the index in view of the dependence axioms}
\label{sec.axioms}

Here we investigate the classes of random vectors in which Axioms \ref{A1} to \ref{A8} and \ref{B1} to \ref{B4} hold; notice that Axiom \ref{A7} is only mentioned in the case $d=2$.

For Axioms \ref{A4} and \ref{B1} we define the class of random vectors $\CX^d_1$, while for Axioms \ref{A3}, \ref{A5} and \ref{B3} we define the class $\CX^d_2$. We first state the conditions that determine these classes of random vectors.
\begin{enumerate}[topsep=0ex, itemsep=0ex, labelindent=0pt, leftmargin=6ex, label=$\mathrm{C}_{\arabic*}\colon$, ref=\textrm{\textcolor{black}{$\mathrm{C}_{\arabic*}$}}]
 \item
 \label{C1}
 One can detect at least a subscript $j=1,\ldots,2^d$ for which the Kendall function of $\bX_j\in\CR(\bX)$ is such that $K_j(t)-K_{\vP}(t)\ge0$ or $\le0$ for all $0\le t \le 1$. That is, we can find at least one Kendall plot that does not cross the diagonal line.

 \item
 \label{C2}
 One can detect at least two subscripts $j=1,\ldots,2^d$ for which the functions of $\bX_j\in\CR(\bX)$ are such that $K_j(t)-K_{\vP}(t)\le0$ for all $0\le t \le 1$. That is, we can find at least two Kendall plots that lie below the diagonal line.
\end{enumerate}
Consider now the following subclasses of $\CX^d_0$:
\[
\CX^d_1\doteq\{\bX\in\CX^d_0\colon \textrm{\ref{C1} holds}\};
\quad
\CX^d_2\doteq\{\bX\in\CX^d_0\colon \textrm{\ref{C2} holds}\}.
\]
Note that $\CX^d_2\subset\CX^d_1\subset\CX^d_0$. Moreover, $\CX^d_2$ is not an empty set; for each continuous rv $X$, the random vector $(X,\ldots,X)^T$ on $\RR^d$ belongs to $\CX^d_2$ because its Kendall function is $K(t)=t$, $0\le t\le1$, see the proof of Equation \eqref{eq.c_d} in Appendix \ref{app:proofs}, which lies below of $K_{\vP}(t)$, $0\le t\le1$ (set $g(t)=K_{\vP}(t)-t$, $0\le t\le1$, and observe that $g$ is a concave function with $g(0)=g(1)=0$). In Section \ref{sec.empirical} we present Algorithm \ref{algorithm-class}, that one can use to check whether the data come from a distribution that belongs in $\CX^d_1$ or $\CX^d_2$.

We are now in a position to investigate Axioms \ref{A1} to \ref{A8} as well as \ref{B1} to \ref{B4}.

\begin{proposition}
\label{prop.A1-A8}
For the index of dependence $I$ defined by \eqref{eq.I} we have:
\begin{enumerate*}[topsep=0ex, itemsep=0ex, wide, labelwidth=!, labelindent=0pt, label=\rm(\alph*), ref=\textrm{\textcolor{black}{\alph*}}]
 \item
 \label{prop.A1-A8(a)}
 Axioms \ref{A1}, \ref{A2}, \ref{A6}, \ref{A7} and \ref{A8} hold in $\CX^d_0$;

 \item
 \label{prop.A1-A8(b)}
 Axiom \ref{A4} holds in $\CX^d_1$;

 \item
 \label{prop.A1-A8(c)}
 Axioms \ref{A3} and \ref{A5} hold in $\CX^d_2$.\hfill\phantom{.}
\end{enumerate*}
\end{proposition}

\begin{proposition}
\label{prop.B1-B4}
For the index of dependence $I$ defined by \eqref{eq.I} we have:
\begin{enumerate*}[topsep=0ex, itemsep=0ex, wide, labelwidth=!, labelindent=0pt, label=\rm(\alph*), ref=\textrm{\textcolor{black}{\alph*}}]
 \item
 \label{prop.B1-B4(a)}
 Axiom \ref{B1} holds in $\CX^d_1$;

 \item
 \label{prop.B1-B4(b)}
 Axioms \ref{B2} and \ref{B4} hold in $\CX^d_0$;

 \item
 \label{prop.B1-B4(c)}
 Axioms \ref{B3} holds in $\CX^d_2$.\hfill\phantom{.}
\end{enumerate*}
\end{proposition}

\begin{remark}
\label{rem.measure}
\begin{enumerate}[topsep=0ex, itemsep=0ex, wide, labelwidth=!, labelindent=0pt, label=\rm(\alph*), ref=\textrm{\textcolor{black}{\alph*}}]
\item\label{rem.measure(a)}
Propositions \ref{prop.A1-A8} and \ref{prop.B1-B4} say that $I$ is a measure of dependence for the structure of random vectors in $\CX^d_2$, in both \citeauthor{Renyi1959}'s as well as \citeauthor{MSz2019}'s definitions.
\item\label{rem.measure(b)}
The standardized index $I^*=\varphi_d(I)$ is also a measure of dependence for the structure of random vectors in $\CX^d_2$, in both \citeauthor{Renyi1959}'s as well as \citeauthor{MSz2019}'s definitions, whenever $\varphi_d$ is a bijective increasing mapping.
\item\label{rem.measure(c)}
$I$ and $I^*$ are indices of the dependence structure for random vectors in $\CX^d_0$.
\item\label{rem.measure(d)}
Consider a random vector $\bX\sim F$ in $\CX^d_0$, and suppose its marginal cdfs are $F_i(t)=\Pr(X_i\le t)$, $i=1,\ldots,d$. Applying the probability integral transform to each $X_i$, the random vector $\bU=(U_1,\ldots,U_d)^T\doteq (F_1(X_1),\ldots,F_d(X_d))^T$ has marginals that are uniformly distributed on $[0,1]$. The copula of $\bX$ is defined as the cdf of $\bU$, $C(\bu)=\Pr(\bU\le\bu)$; notice that the copula $C$ contains all information on the dependence structure between the components of $\bX$. Since Axiom \ref{A6} holds, see Proposition \ref{prop.A1-A8}\ref{prop.A1-A8(a)}, we have that the index $I$, and hence $I^*$, is a copula-based index of dependence; namely, $I_F=I_C$ and $I^*_F=I^*_C$.
\end{enumerate}
\end{remark}

\section{Estimating $\mathbf{AUK}$, $\bm{I}$, $\bm{I^*}$}
\label{sec.empirical}

Suppose that a random sample $\bX_i$, $i=1,\ldots,n$, has been drawn from a $d$-dimensional distribution $F$. For each $\bs_j\in\SSS^d$, $j=1,\ldots,d$, set $\bX_{ij}=\diag(\bs_j)\bX_i$, $i=1,\ldots,n$, which is a random sample from $\bX_j=\diag(\bs_j)\bX$; and let $\bbF_{nj}$ denote the empirical cumulative distribution function (ecdf) of $\bX_{ij}=\diag(\bs_j)\bX_i$, $i=1,\ldots,n$,
\[
\bbF_{nj}(\bx)=\frac{1}{n}\sum_{i=1}^{n}\one{\bX_{ij}\le \bx},\quad \bx\in\RR^d;
\]
consider the corresponding empirical Kendall distribution function
\begin{equation}\label{eq.bbK,bbk}
\bbK_{nj}(t)=\frac{1}{n}\sum_{i=1}^{n}\one{\bbF_{nj}(\bX_{ij})\le t},\quad t\in[0,1].
\end{equation}
Let $\hT_{nj}\sim\bbK_{nj}$, $j=1,\ldots,d$. By virtue of the formula $\AUK=\E\left\{1-K_{\vP}(T)\right\}$, we can estimate the $\AUK_j$s in a nonparametric manner via the statistic
\[
\hAUK_{nj}=\E\left\{1-K_{\vP}(\hT_{nj})\right\}=1-\frac{1}{n}\sum_{i=1}^{n}K_{\vP}(\hT_{nij})
%=\E\left\{F_{\chi^2_{2d}}(-2\ln\hT_{nj})\right\}=\frac{1}{n}\sum_{i=1}^{n}F_{\chi^2_{2d}}(-2\ln\hT_{nij})
,
\]
where $\hT_{nij}=\bbF_{nj}(\bX_{ij})$, $i=1,\ldots,n$. Then, we get nonparametric estimations of $I$ and $I^*$, $\hI_n$ and $\hI^*_n$ respectively, by
\[
\hI_n\doteq\left(2^{d-2}-2^{1-d}+2^{1-2d}\right)^{-1/2}\|\h{\bD}_n-\bDelta\|
\quad\textrm{and}\quad
\hI^*_n\doteq\varphi_d(\hI_n),
\]
where $\h{\bD}_n=(\hAUK_{n1},\ldots,\hAUK_{n2^d})^T$.

We are interested in studying the consistency of the proposed estimators. Theorem \ref{thm.hT->T} presents the consistency of $\bbK_{nj}$s, see in Appendix. Based on this result, in Appendix \ref{app:proofs}, we prove that $\hAUK_{nj}$ is a strongly consistent estimator of $\AUK_{j}$, that is,
\begin{equation}
\label{eq.as->AUK}
\hAUK_{nj}\as\AUK_{j}\textrm{ for all } j=1,\ldots,2^d;
\end{equation}
hence, using the continuous mapping theorem, $\hI_n\as I$ as well as $\hI^*_n\as I^*$, that is, both $\hI_n$ and $\hI^*_n$ are strongly consistent estimators of $I$ and $I^*$ respectively.

In addition to the consistency of $\hI_n$ and $\hI^*_n$, Theorem \ref{thm.hT->T} gives a graphical method for checking if the underlying data distribution belongs in $\CX^d_1$ or $\CX^d_2$; Algorithm \ref{algorithm-class} enables one to carry out such a check.

\begin{algorithm}
\caption{Algorithm for checking if the underlying data distribution belongs in $\CX^d_1$ or $\CX^d_2$.}
 \label{algorithm-class}
\small
\begin{algorithmic}[1]
\State\label{algorithm-class.step1}
Compute $K_\vP$ and $K_{nj}$, $j=1,\dots,2^d$, using the formulas \eqref{eq.K,k} and \eqref{eq.bbK,bbk}.

\State\label{algorithm-class.step2}
Visualize the curves $(K_{nj}(t),K_\vP(t))$, $0\le t\le1$, for all $j=1,\dots,2^d$.

\State\label{algorithm-class.step3}
{\bf Decision Rule:}
\begin{enumerate}[topsep=0ex, itemsep=0ex, labelindent=0pt, leftmargin=2.5ex]
\item[--]
If there is at least one curve that does not cross the diagonal line, then the data set is from $\CX^d_1$.
\item[--]
If there are at least two curves that lie below the diagonal line, then the data set is from $\CX^d_2$.
\end{enumerate}
\end{algorithmic}
\end{algorithm}

\section{Testing joint or mutual dependence}
\label{sec.test}

In this section, we investigate the null hypothesis of total independence. Let $\hAUK_n$ be the estimator of $\AUK$ of the original/non-rotated data distribution; namely, the $\hAUK_{nj}$ that corresponds to $\bs_j=(+,\ldots,+)^T\in\SSS^d$.

We are interested in testing the hypothesis
\[
H_0\colon F(\bx)=F_1(x_1)\cdots F_d(x_d) \quad\textrm{versus}\quad H_1\colon F(\bx)\ne F_1(x_1)\cdots F_d(x_d).
\]
For testing this hypothesis, we propose the statistic
\[
z_n=\surd{n}(\hAUK_n-1/2)/\sigma_{\vP},
\]
where the true value of the standard deviation $\sigma_{\vP}$ can be computed explicitly through the relations \eqref{eq.covergence(b)}, \eqref{eq.cov-function(a)}--\eqref{eq.cov-function(c)} given below. In the case $d=2$, we have computed $\sigma_{\vP}=(19/432)^{1/2}$ --- see Appendix \ref{app:Gn}.

We now prove the asymptotic normality of the $\hAUK_n$ under this null hypothesis. The result is derived through the asymptotic behavior of a modified Kendall process that originally was defined by \citet{GR1993} as well as the results given by \citet{BGGR1996}. For easy reference, we state these results in Appendix \ref{app:Gn}. In Appendix \ref{app:Gn}, we prove that
\begin{subequations}\label{eq.covergence}
\begin{equation}\label{eq.covergence(a)}
\surd{n}(\hAUK_n-1/2)\stackrel{H_0}{\rightsquigarrow}\CN(0,\sigma_{\vP}^2),
\quad\textrm{with}
\end{equation}
\begin{equation}\label{eq.covergence(b)}
\sigma_{\vP}^2=\int_{0}^{1}\int_{0}^{1}k_{\vP}(s)\varGamma_\vP(s,t)k_{\vP}(t)\ud{s}\ud{t}
              =2\iint_{0<s<t<1}k_{\vP}(s)\varGamma_\vP(s,t)k_{\vP}(t)\ud{s}{\ud{t}},
\end{equation}
\end{subequations}
where
\begin{subequations}\label{eq.cov-function}
\begin{equation}\label{eq.cov-function(a)}
\varGamma_\vP(s,t)=K_\vP(s\wedge t)-K_\vP(s)K_\vP(t)+k_\vP(s)k_\vP(t)R_\vP(s,t)-k_\vP(t)Q_\vP(s,t)-k_\vP(s)Q_\vP(t,s),
\end{equation}
where $\bm{u}\wedge\bm{v}$ denotes the componentwise minimum between $\bm{u}$ and $\bm{v}$,
\begin{equation}\label{eq.cov-function(b)}
Q_\vP(s,t)=(s\wedge t)\sum_{i=1}^{d}\frac{\ln^i(t/s\wedge t)}{i!}-tK_\vP(s)
\end{equation}
\begin{equation}\label{eq.cov-function(c)}
R_\vP(s,t)=\E\left(\exp\left[\sum_{i=1}^{d}\min\left\{\ln(s)U_1^{(i)},\ln(t)U_2^{(i)}\right\}\right]\right)-st,
\end{equation}
where $\bU_j=\left(U_j^{(1)},\ldots,U_j^{(d)}\right)^T$, $j=1,2$, are independent random vectors uniformly distributed on the simplex in $[0,1]^d$.
\end{subequations}

Therefore, at significance level $\alpha$, the asymptotic rejection region is $R=\{|z_n|>z_{\alpha/2}\}$, where $z_{\alpha}$ is the upper $\alpha$-point of the standard normal distribution. Hereafter, we call this test as $\AUK$ test for total independence or simply $\AUK$ independence test.

We investigate further the proposed independence test in the bivariate case. The distribution of $z_n$ is asymptotically $\CN(0,1)$. Table \ref{table.zna} contains the 90th, 95th and 99th percentiles of the empirical distribution of $|z_n|$ generated by simulation of $10^5$ samples of various sizes $n$ from the standard bivariate uniform distribution. For a sample of size $n$, the proposed $\alpha$-level independence test is
\[
\textrm{reject the total independence if } |z_n|>p_{d=2,n;1-\alpha},
\]	
where $p_{d=2,n;1-\alpha}$ is given in Table \ref{table.zna} noting that $p_{d=2,\infty;1-\alpha}=z_{\alpha/2}$.

\begin{table}[htp]
 \caption{Empirical percentiles of the distribution of $|z_n|$ for random samples of size $n$ drawn standard bivariate uniform distribution.}
 \label{table.zna}
 \small
 \begin{tabular*}{\textwidth}
 {@{\hspace{0ex}}@{\extracolsep{\fill}}l@{\hspace{0ex}}c@{\hspace{0ex}}c@{\hspace{0ex}}c@{\hspace{0ex}}c@{\hspace{0ex}}c@{\hspace{0ex}}c@{\hspace{0ex}} c@{\hspace{0ex}}c@{\hspace{0ex}}c@{\hspace{0ex}}c@{\hspace{0ex}}c@{\hspace{0ex}}c@{\hspace{0ex}}}
 \toprule
 $n$ & 30 & $50$ & $70$ & $100$ & $150$ & $200$ & $300$ & $400$ & $500$ & $750$ & $1000$ & $\infty$ \\
 \midrule
 $p_{d=2,n;0.90}$  & 2.30 & 2.11 & 2.01 & 1.93 & 1.84 & 1.79 & 1.74 & 1.72 & 1.71 & 1.68 & 1.67 & 1.65 \\
 $p_{d=2,n;0.95}$  & 2.62 & 2.44 & 2.34 & 2.25 & 2.17 & 2.12 & 2.06 & 2.05 & 2.03 & 2.01 & 1.98 & 1.96 \\
 $p_{d=2,n;0.99}$  & 3.19 & 3.05 & 2.95 & 2.87 & 2.78 & 2.75 & 2.68 & 2.67 & 2.65 & 2.63 & 2.60 & 2.57 \\
 \bottomrule
 \end{tabular*}
 \end{table}

Suppose we have a $d$-dimensional continuous cdf $G$ such that the corresponding $\AUK\ne1/2$. Due to consistency, see \eqref{eq.as->AUK}, and using the continuous mapping theorem, we have that $|\hAUK_n-1/2|\as|\AUK-1/2|\ne0$; therefore, the statistic $|z_n|$ diverges almost surely to infinity; thus, if the alternative hypothesis is $F=G$, the power of the test goes to 1. Consequently, the proposed $\hAUK_n$-based test for total independence is consistent against the set of alternatives of $d$-dimensional continuous cdfs with $\AUK\ne1/2$.

We now offer an alternative way to compute the standard deviation $\sigma_\varPi$. Recall that the true standard deviation $\sigma_\varPi$ can be computed exactly through the relations \eqref{eq.covergence(b)} and \eqref{eq.cov-function(a)}--\eqref{eq.cov-function(c)}. Since $\bU_j$s are uniformly distributed on the $d$-dimensional probability simplex, it is rather laborious to compute the exact value of the true standard deviation $\sigma_\varPi$ for general values of $d$; the difficulty is due to the calculations on the simplex. Thus, we provide a Monte Carlo based algorithm for estimating the standard deviation $\sigma_\varPi$.

\begin{algorithm}
\caption{Approximation of the standard deviation $\sigma_\varPi$.}
 \label{algorithm-sigmaPi}
\small
\begin{algorithmic}[1]
\State\label{algorithm-sigma2Pi1}
Let $d$ be the dimension of interest. Consider a $d$-dimensional continuous distribution with independent components; without loss of generality, consider the standard $d$-dimensional uniform distribution, say $\bU$, that is the distribution with density $f_{\bU}(\bu)=u_1\cdots u_d$, $\bu=(u_1,\dots,u_d)^T\in[0,1]^d$.

\State\label{algorithm-sigma2Pi2}
Let $n$ and $r$ be two large positive integer numbers. Generate $r$ independent random samples of size $n$ from $F_{\bU}$, say $\bU_{i,j}$, $i=1,\dots,n$ and $j=1,\dots,r$.

\State\label{algorithm-sigma2Pi3}
Based on the $j$th random sample, $\bU_{i,j}$, $i=1,\dots,n$, compute the $\hAUK_j$ and then $A_i\doteq\surd{n}(\hAUK_j-1/2)$, $j=1,\dots,r$.

\State\label{algorithm-sigma2Pi4}
Define $\h\sigma_{\vP}$ to be the sample standard deviation of $A_1,\dots,A_r$.
\end{algorithmic}
\end{algorithm}

Using Algorithm \ref{algorithm-sigmaPi} with $r=10000$ and $n=50000$, we offer estimations $\h\sigma_{\vP}$ for the cases $d=2,\dots,10$. Observe that in the case $d=2$ the estimated value $\h\sigma_{\vP}=0.20988$ is very closed to the true value $\sigma_{\vP}=(19/432)^{1/2}\approxeq0.20972$ ($\h\sigma_{\vP}/\sigma_{\vP}=1.000779$).

\begin{table}[htp]
 \caption{Mode Carlo estimate of $\sigma_{\varPi}$, using Algorithm \ref{algorithm-sigmaPi} with $r=10000$ and $n=50000$, when $d=2,\dots,10$.}
 \label{table.hsigma2Pi}
 \small
 \begin{tabular*}{\textwidth}
 {@{\hspace{0ex}}@{\extracolsep{\fill}}l@{\hspace{0ex}}c@{\hspace{0ex}}c@{\hspace{0ex}}c@{\hspace{0ex}}c@{\hspace{0ex}}c@{\hspace{0ex}}c@{\hspace{0ex}} c@{\hspace{0ex}}c@{\hspace{0ex}}c@{\hspace{0ex}}}
 \toprule
 $d$                 & 2 & 3 & 4 & 5 & 6 & 7 & 8 & 9 & 10 \\
 $\h\sigma_{\vP}$  & 0.20988 & 0.19383 & 0.16254 & 0.12511 & 0.09407 & 0.06853 & 0.04912 & 0.03395 & 0.02377 \\
 \bottomrule
 \end{tabular*}
 \end{table}

To make our test practical for small sample sizes $n$, we present an algorithm to approximate the percentiles of the distribution of $|z_n|$. Suppose $\bX_1,\ldots,\bX_n$ is a random sample of size $n$ from a $d$-dimensional continuous distribution $F$ and we are interested in testing the null hypothesis of total independence. We use the statistic $z_n$; and the rejection region of the null hypothesis at $\alpha$-significant level is $|z_n|>p_{d,n;1-\alpha}$, where $p_{d,n;1-\alpha}$ denote the $(1-\alpha)$-percentile point of the distribution of $|z_n|$. If the sample size $n$ is large, that is $n>\max\{1000,100d\}$, then $p_{d,n;1-\alpha}=p_{d,\infty;1-\alpha}=z_{\alpha/2}$ else we approximate this percentile point using Algorithm \ref{algorithm-percentiles}.

\begin{algorithm}
\caption{Approximation of the percentiles of the distribution of $|z_n|$ for small sample sizes.}
 \label{algorithm-percentiles}
\small
\begin{algorithmic}[1]
\State\label{algorithm-percentiles1}
Consider the standard $d$-dimensional uniform distribution, say $\bU$, that is the distribution with density $f_{\bU}(\bu)=u_1\cdots u_d$, $\bu=(u_1,\dots,u_d)^T\in[0,1]^d$.

\State\label{algorithm-percentiles2}
For a large positive integer $r$, generate $r$ independent random samples of size $n$ from $F_{\bU}$, say $\bU_{i,j}$, $i=1,\dots,n$ and $j=1,\dots,r$.

\State\label{algorithm-percentiles3}
Based on the $j$th random sample, $\bU_{i,j}$, $i=1,\dots,n$, compute the corresponding $|z_n|$-value, say $|z_{n;j}|$, $j=1,\dots,r$.

\State\label{algorithm-percentiles4}
Define $p_{d,n;1-\alpha}$ to be the $(1-\alpha)$-percentile point of the distribution of the values $|z_{n;1}|,\dots,|z_{n;r}|$.
\end{algorithmic}
\end{algorithm}

Notice that if the dimension $d$ is small, for example $d\le10$, and since $n\le1000$, the number of Monte Carlo repetitions $r=10^4$ is sufficient for an accurate approximation of $p_{d,n;1-\alpha}$ using Algorithm \ref{algorithm-percentiles} with low computational cost.

\begin{remark}
\label{rem.U,V-statistics}
For the estimation of $\AUK$ we used the relationship $\AUK=\E\left\{1-K_{\vP}(T)\right\}$. The study of the asymptotic behavior of this estimator, i.e.\ the estimator $\hAUK_{n}$ that is presented in Section \ref{sec.test}, as well as its limit distribution under the hypothesis of total independence go through the Kendall process and the functional \emph{delta}-method.  Another way for estimating the $\AUK$ is via use of the formula $\AUK=\Pr(T\le \vP)$; then, similar to \citet{PBSP2018}, a kernel-based estimation of $\AUK$ seems to be feasible and then the use of $U$- or $V$-statistics theory allows one to obtain the asymptotic distribution of the statistic.
\end{remark}

We now discuss the possibility of a general $\AUK$-based hypothesis test, i.e.\ $H_0\colon \AUK=a_0$ versus $H_1\colon \AUK\ne a_0$ for a fixed constant $a_0\in[2^{-d},1]$; until now we have only studied the case $a_0=1/2$. We first present the following example.
\begin{example}
\label{exm.C(0,1),normal}
Let $\bX=(X_1,X_2)^T$ be a bivariate rv uniformly distributed on the circumference of the circle centered at the origin with a radius of 1. The ordinary Kendall cdf of $\bX$ is $K_{\bX}(t)=\one{0\le t<1/2}(t+1/4)+\one{1/2\le t\le1}$ \citep[see][p.~426]{VAM2019} and so $\AUK_{\bX}=11/16-\ln(2)/4$. Consider now the bivariate normal rv $\bZ=(Z_1,Z_2)^T\sim \CN_2\left(\bzero,\left({1\atop-0.073622}~{-0.073622\atop1}\right)\right)$. A numerical computation gives $\AUK_{\bZ}=\AUK_{\bX}$. One can verify that hypotheses \ref{H1} and \ref{H2}, see Appendix \ref{app:Gn}, hold and so Theorem \ref{thm.Gn} also holds for the random vector $\bZ$. On the other hand, hypothesis \ref{H1} does not hold for the random vector $\bX$; hence, the statement of Theorem \ref{thm.Gn} is not true for this case.
\end{example}
In view of the above example, we conclude that it is not possible to construct general $\AUK$-based hypotheses tests, as these are described at the beginning  of the paragraph.

\section{Simulation study}
\label{sec.Simulation}

In this section, we present a simulation study designed with two specific goals in mind. First, we are interested in understanding the behavior of the proposed indices of dependence as a function of the sample size, the type and the degree of dependence that a data set exhibits.

Our second goal relates to the evaluation of the performance of the proposed tests of independence, in terms of level of significance and power of the tests. In what follows, we present evidence of this performance by simulating data from a range of distributions and hence, dependence structures.

\subsection{The estimators $\hI_n$ and $\hI^*_n$}
\label{ssec.sim.indices}

We first present a simulation study for the proposed empirical dependence indices.

\subsubsection{Normal case}
\label{sssec.sim.normal}

We simulate $r=1000$ random samples of size $n=100,200,500,1000$ from the trivariate normal $\CN_3(\bzero,\Sig_3(\brho))$ distribution (for the variance/covariance matrix $\Sig_3(\brho)$ see \eqref{eq.Sig(a)}), for various values of $n$ and $\brho=(\rho_{12},\rho_{13},\rho_{23})$. Based on these samples we compute the averages (and the Monte Carlo mean square error) of $\hI$ and $\hI^*$; Table \ref{table.Sim-N3} presents the results. In all cases the empirical mean square error of the indices tends to zero as $n$ goes to infinity, confirming the consistency of the estimation for both $\hI$ and $\hI^*$ estimators. Figure \ref{fig.N3Scatterplots} presents illustrations of graphs of a single sample $\{\bX_1,\ldots,\bX_n\}$ of size $n=1000$ drawn from these distributions. The standardized index $I^*$ indicates that, see Definition \ref{def.level}, the normal distributions with $(\rho_{12},\rho_{13},\rho_{23})\in\{(0,0,0),(-0.1,-0.1,0.2)\}$ have a weak dependence level, that the ones with $(\rho_{12},\rho_{13},\rho_{23})\in\{(-0.3,-0.3,-0.3),(0.2,0.3,0.4)\}$ have a mild dependence level, while the distributions with $(\rho_{12},\rho_{13},\rho_{23})\in\{(0.7,0.5,0),(0.2,-0.8,0)\}$ have a strong dependence level, and that normal distributions with $(\rho_{12},\rho_{13},\rho_{23})\in\{(1,1,1),(-0.5,-0.5,-0.5),(0.1,0.2,-0.9)\}$ have a very strong dependence level. On the other hand, the case $(\rho_{12},\rho_{13},\rho_{23})=(-0.5,-0.5,0.5)$ lies in the frontier between mild and strong dependence level.

\begin{table*}[htp]
 \centering{
 \caption[Monte Carlo averages (and associated mean square error) of $\hI_n$ and $\hI^*_n$; the number of Monte Carlo repetitions is $r=1000$. Random samples of size $n$ are drawn from different trivariate normal distributions, for various values of $n$.]{Monte Carlo averages (and associated mean square error) of $I$ and $I^*$; the number of Monte Carlo repetitions is $r=1000$. Random samples of size $n$ are drawn from different $\CN_3\left(\bzero,\Sig_3(\brho)\right)$ distributions, for various values of $n$ and $\brho=(\rho_{12},\rho_{13},\rho_{23})$.}
 \label{table.Sim-N3}
 \footnotesize
\begin{subtable}{.47\textwidth}
 \caption{$(\rho_{12},\rho_{13},\rho_{23})=(0,0,0)$; true values of indices: $I=0$ and $I^*=0$.}
 \label{subtable.normal000}
 \scriptsize
 \begin{tabular*}{\textwidth}
 {@{\hspace{0ex}}@{\extracolsep{\fill}}l@{\hspace{0ex}}c@{\hspace{0ex}}c@{\hspace{0ex}}c@{\hspace{0ex}}c@{\hspace{0ex}}}
 \addlinespace
 \toprule
 & \multicolumn{4}{c}{$n$}\\
 \cmidrule{2-5}
 & 100 & $200$ & $500$ & $1000$ \\
 \midrule
 $\hI_n$   & .097(.009) & .059(.004) & .032(.001) & .022(.000) \\
 $\hI^*_n$ & .186(.035) & .109(.012) & .057(.003) & .037(.001) \\
 \bottomrule
 \end{tabular*}
\end{subtable}
\hfill
\begin{subtable}{.47\textwidth}
 \caption{$(\rho_{12},\rho_{13},\rho_{23})=(1,1,1)$; true values of indices: $I=1$ and $I^*=1$.}
 \label{subtable.normal111}
 \scriptsize
 \begin{tabular*}{\textwidth}
 {@{\hspace{0ex}}@{\extracolsep{\fill}}l@{\hspace{0ex}}c@{\hspace{0ex}}c@{\hspace{0ex}}c@{\hspace{0ex}}c@{\hspace{0ex}}}
 \addlinespace
 \toprule
 & \multicolumn{4}{c}{$n$}\\
 \cmidrule{2-5}
 & 100 & $200$ & $500$ & $1000$ \\
 \midrule
 $\hI_n$   & .739(.068) & .833(.028) & .912(.008) & .947(.003) \\
 $\hI^*_n$ & .985(.000) & .994(.000) & .998(.000) & .999(.000) \\
 \bottomrule
 \end{tabular*}
\end{subtable}
\bigskip\linebreak
\begin{subtable}{.47\textwidth}
 \caption{$(\rho_{12},\rho_{13},\rho_{23})=(-0.5,-0.5,0.5)$; true values of indices: $I=0.243$ and $I^*=0.5$.}
 \label{subtable.normal-.45-.45.45}
 \scriptsize
 \begin{tabular*}{\textwidth}
 {@{\hspace{0ex}}@{\extracolsep{\fill}}l@{\hspace{0ex}}c@{\hspace{0ex}}c@{\hspace{0ex}}c@{\hspace{0ex}}c@{\hspace{0ex}}}
 \addlinespace
 \toprule
 & \multicolumn{4}{c}{$n$}\\
 \cmidrule{2-5}
 & 100 & $200$ & $500$ & $1000$ \\
 \midrule
 $\hI_n$   & .222(.001) & .229(.000) & .236(.000) & .239(.000) \\
 $\hI^*_n$ & .459(.004) & .472(.002) & .487(.001) & .493(.000) \\
 \bottomrule
 \end{tabular*}
\end{subtable}
\hfill
\begin{subtable}{.47\textwidth}
 \caption{$(\rho_{12},\rho_{13},\rho_{23})=(-0.5,-0.5,-0.5)$; true values of indices: $I=0.554$ and $I^*=0.92$.}
 \label{subtable.normal-.45-.45-.45}
 \scriptsize
 \begin{tabular*}{\textwidth}
 {@{\hspace{0ex}}@{\extracolsep{\fill}}l@{\hspace{0ex}}c@{\hspace{0ex}}c@{\hspace{0ex}}c@{\hspace{0ex}}c@{\hspace{0ex}}}
 \addlinespace
 \toprule
 & \multicolumn{4}{c}{$n$}\\
 \cmidrule{2-5}
 & 100 & $200$ & $500$ & $1000$ \\
 \midrule
 $\hI_n$   & .418(.019) & .464(.008) & 0.505(.002) & 0.525(.001) \\
 $\hI^*_n$ & .794(.016) & .846(.006) & 0.885(.001) & 0.900(.000) \\
 \bottomrule
 \end{tabular*}
\end{subtable}
\bigskip\linebreak
\begin{subtable}{.47\textwidth}
 \caption{$(\rho_{12},\rho_{13},\rho_{23})=(0.10,0.20,-0.90)$; true values of indices: $I=0.461$ and $I^*=0.842$.}
 \label{subtable.normal.1.2-.9}
 \scriptsize
 \begin{tabular*}{\textwidth}
 {@{\hspace{0ex}}@{\extracolsep{\fill}}l@{\hspace{0ex}}c@{\hspace{0ex}}c@{\hspace{0ex}}c@{\hspace{0ex}}c@{\hspace{0ex}}}
 \addlinespace
 \toprule
 & \multicolumn{4}{c}{$n$}\\
 \cmidrule{2-5}
 & 100 & $200$ & $500$ & $1000$ \\
 \midrule
 $\hI_n$   & .382(.006) & .413(.002) & .438(.000) & .449(.000) \\
 $\hI^*_n$ & .745(.010) & .787(.003) & .818(.000) & .829(.000) \\
 \bottomrule
 \end{tabular*}
\end{subtable}
\hfill
\begin{subtable}{.47\textwidth}
 \caption{$(\rho_{12},\rho_{13},\rho_{23})=(0.70,0.50,0)$; true values of indices: $I=0.363$ and $I^*=0.717$.}
 \label{subtable.normal.7.50}
 \scriptsize
 \begin{tabular*}{\textwidth}
 {@{\hspace{0ex}}@{\extracolsep{\fill}}l@{\hspace{0ex}}c@{\hspace{0ex}}c@{\hspace{0ex}}c@{\hspace{0ex}}c@{\hspace{0ex}}}
 \addlinespace
 \toprule
 & \multicolumn{4}{c}{$n$}\\
 \cmidrule{2-5}
 & 100 & $200$ & $500$ & $1000$ \\
 \midrule
 $\hI_n$   & .311(.003) & .330(.001) & .348(.000) & .355(.000) \\
 $\hI^*_n$ & .631(.008) & .664(.003) & .693(.001) & .704(.000) \\
 \bottomrule
 \end{tabular*}
\end{subtable}
\bigskip\linebreak
\begin{subtable}{.47\textwidth}
 \caption{$(\rho_{12},\rho_{13},\rho_{23})=(0.20,-0.80,0)$; true values of indices: $I=0.334$ and $I^*=0.671$.}
 \label{subtable.normal.2-.80}
 \scriptsize
 \begin{tabular*}{\textwidth}
 {@{\hspace{0ex}}@{\extracolsep{\fill}}l@{\hspace{0ex}}c@{\hspace{0ex}}c@{\hspace{0ex}}c@{\hspace{0ex}}c@{\hspace{0ex}}}
 \addlinespace
 \toprule
 & \multicolumn{4}{c}{$n$}\\
 \cmidrule{2-5}
 & 100 & $200$ & $500$ & $1000$ \\
 \midrule
 $\hI_n$   & .291(.002) & .307(.001) & .322(.000) & .327(.000) \\
 $\hI^*_n$ & .594(.007) & .624(.003) & .651(.001) & .659(.000) \\
 \bottomrule
 \end{tabular*}
\end{subtable}
\hfill
\begin{subtable}{.47\textwidth}
 \caption{$(\rho_{12},\rho_{13},\rho_{23})=(-0.30,-0.30,-0.30)$; true values of indices: $I=0.204$ and $I^*=0.420$.}
 \label{subtable.normal-.3-.3-.3}
 \scriptsize
 \begin{tabular*}{\textwidth}
 {@{\hspace{0ex}}@{\extracolsep{\fill}}l@{\hspace{0ex}}c@{\hspace{0ex}}c@{\hspace{0ex}}c@{\hspace{0ex}}c@{\hspace{0ex}}}
 \addlinespace
 \toprule
 & \multicolumn{4}{c}{$n$}\\
 \cmidrule{2-5}
 & 100 & $200$ & $500$ & $1000$ \\
 \midrule
 $\hI_n$   & .192(.001) & .192(.000) & .197(.000) & .200(.000) \\
 $\hI^*_n$ & .395(.003) & .395(.002) & .406(.001) & .412(.000) \\
 \bottomrule
 \end{tabular*}
\end{subtable}
\bigskip\linebreak
\begin{subtable}{.47\textwidth}
 \caption{$(\rho_{12},\rho_{13},\rho_{23})=(0.20,0.30,0.40)$; true values of indices: $I=0.159$ and $I^*=0.324$.}
 \label{subtable.normal.2.3.4}
 \scriptsize
 \begin{tabular*}{\textwidth}
 {@{\hspace{0ex}}@{\extracolsep{\fill}}l@{\hspace{0ex}}c@{\hspace{0ex}}c@{\hspace{0ex}}c@{\hspace{0ex}}c@{\hspace{0ex}}}
 \addlinespace
 \toprule
 & \multicolumn{4}{c}{$n$}\\
 \cmidrule{2-5}
 & 100 & $200$ & $500$ & $1000$ \\
 \midrule
 $\hI_n$   & .165(.001) & .156(.000) & .156(.000) & .157(.000) \\
 $\hI^*_n$ & .336(.003) & .316(.002) & .317(.001) & .319(.000) \\
 \bottomrule
 \end{tabular*}
\end{subtable}
\hfill
\begin{subtable}{.47\textwidth}
 \caption{$(\rho_{12},\rho_{13},\rho_{23})=(-0.10,-0.10,0.20)$; true values of indices: $I=0.074$ and $I^*=0.138$.}
 \label{subtable.normal-.1-.1.2}
 \scriptsize
 \begin{tabular*}{\textwidth}
 {@{\hspace{0ex}}@{\extracolsep{\fill}}l@{\hspace{0ex}}c@{\hspace{0ex}}c@{\hspace{0ex}}c@{\hspace{0ex}}c@{\hspace{0ex}}}
 \addlinespace
 \toprule
 & \multicolumn{4}{c}{$n$}\\
 \cmidrule{2-5}
 & 100 & $200$ & $500$ & $1000$ \\
 \midrule
 $\hI_n$   & .114(.002) & .089(.000) & .077(.000) & .075(.000) \\
 $\hI^*_n$ & .224(.009) & .170(.002) & .146(.001) & .141(.000) \\
 \bottomrule
 \end{tabular*}
\end{subtable}
}
 \end{table*}

\begin{figure}[htp]
\begin{subfigure}[a]{.24\textwidth}
\FIG
{\resizebox{\linewidth}{!}{
\begin{tikzpicture}
\begin{axis}[view={30}{20}]
\addplot3+[draw=none, mark=o,color=blue,mark size=1] table {normal5-5-5.dat};
\end{axis}
\end{tikzpicture}
}}
{\includegraphics[width=\textwidth]{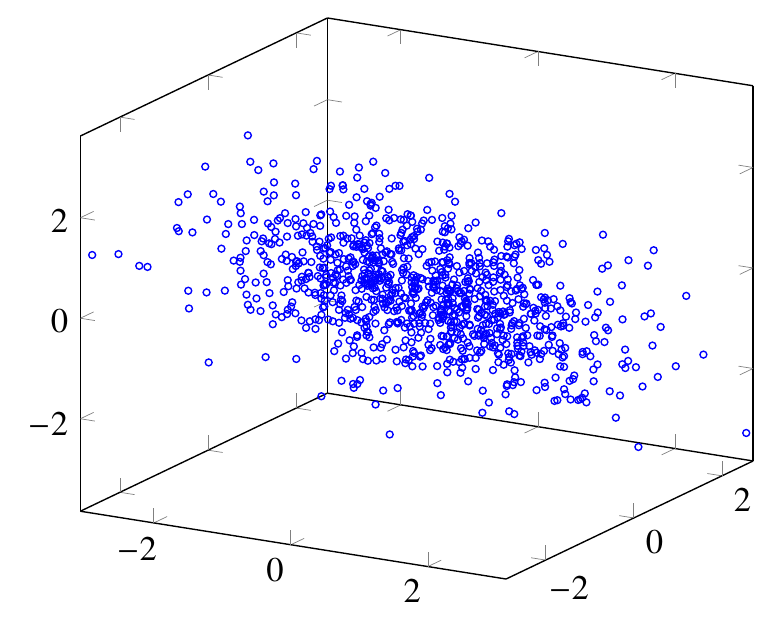}}
\caption{$(-0.5,-0.5,0.5)$.}
\label{sfig.normal-5-55}
\end{subfigure}
\hfill
\begin{subfigure}[a]{.24\textwidth}
\FIG
{\resizebox{\linewidth}{!}{
\begin{tikzpicture}
\begin{axis}[view={30}{20}]
\addplot3+[draw=none, mark=o,color=blue,mark size=1] table {normal-5-5-5.dat};
\end{axis}
\end{tikzpicture}
}}
{\includegraphics[width=\textwidth]{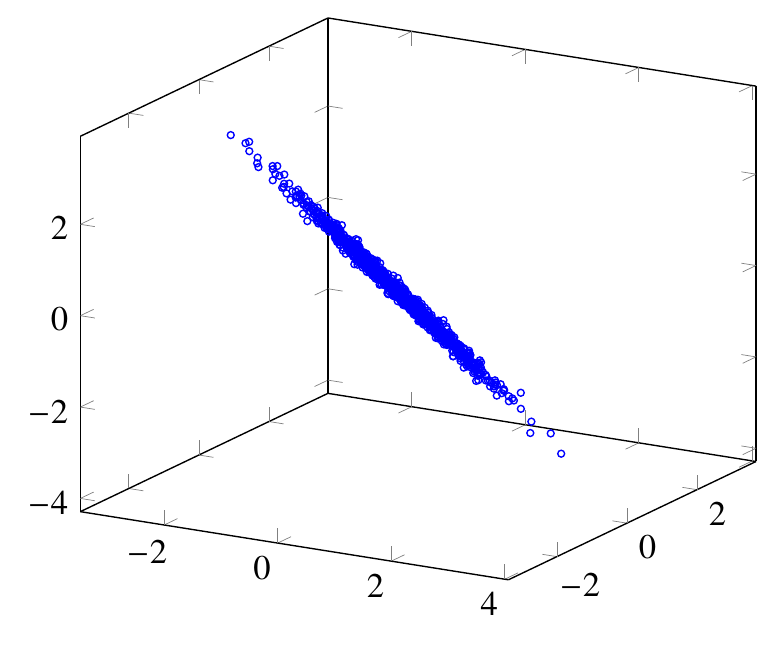}}
\caption{$(-0.5,-0.5,-0.5)$.}
\label{sfig.normal-5-5-5}
\end{subfigure}
\hfill
\begin{subfigure}[a]{.24\textwidth}
\FIG
{\resizebox{\linewidth}{!}{
\begin{tikzpicture}
\begin{axis}[view={30}{20}]
\addplot3+[draw=none, mark=o,color=blue,mark size=1] table {normal12-9.dat};
\end{axis}
\end{tikzpicture}
}}
{\includegraphics[width=\textwidth]{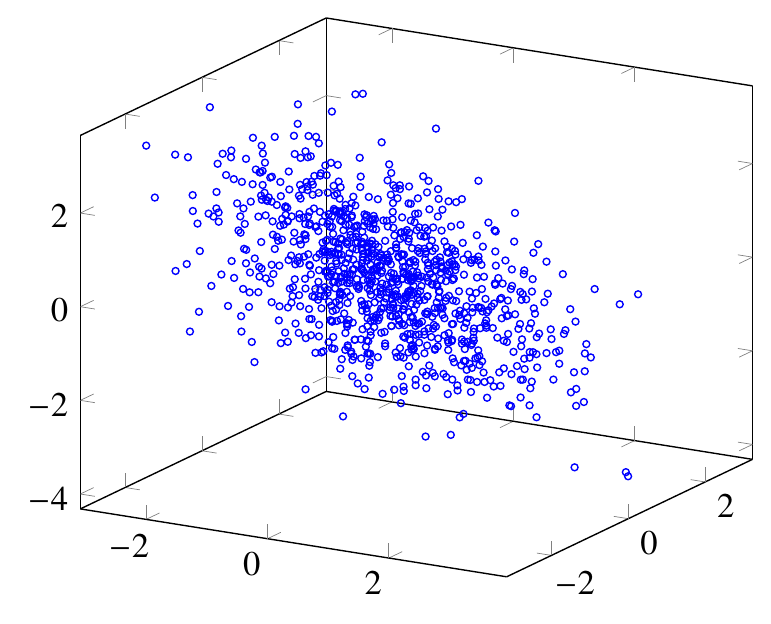}}
\caption{$(0.1,0.1,-0.9)$.}
\label{sfig.normal12-9}
\end{subfigure}
\hfill
\begin{subfigure}[a]{.24\textwidth}
\FIG
{\resizebox{\linewidth}{!}{
\begin{tikzpicture}
\begin{axis}[view={30}{20}]
\addplot3+[draw=none, mark=o,color=blue,mark size=1] table {normal750.dat};
\end{axis}
\end{tikzpicture}
}}
{\includegraphics[width=\textwidth]{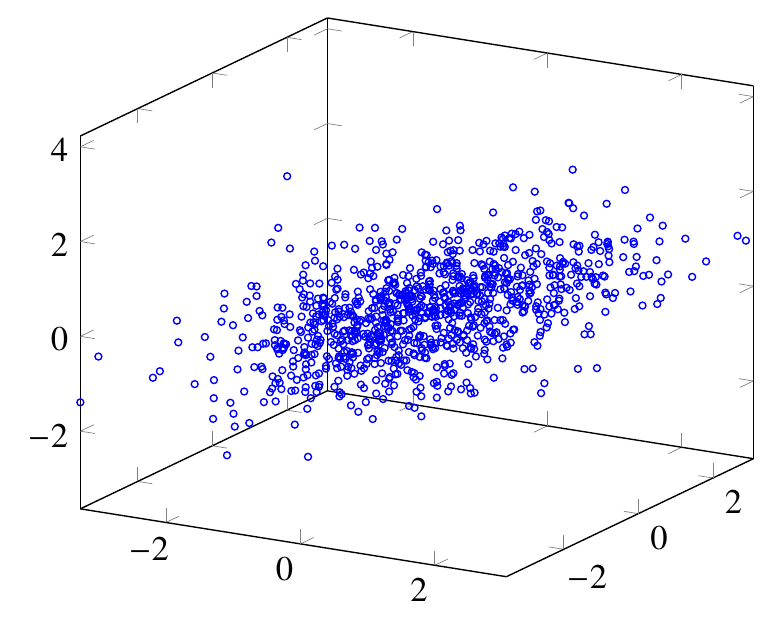}}
\caption{$(0.7,0.5,0)$.}
\label{sfig.normal750}
\end{subfigure}
\medskip\linebreak
\begin{subfigure}[a]{.24\textwidth}
\FIG
{\resizebox{\linewidth}{!}{
\begin{tikzpicture}
\begin{axis}[view={30}{20}]
\addplot3+[draw=none, mark=o,color=blue,mark size=1] table {normal2-80.dat};
\end{axis}
\end{tikzpicture}
}}
{\includegraphics[width=\textwidth]{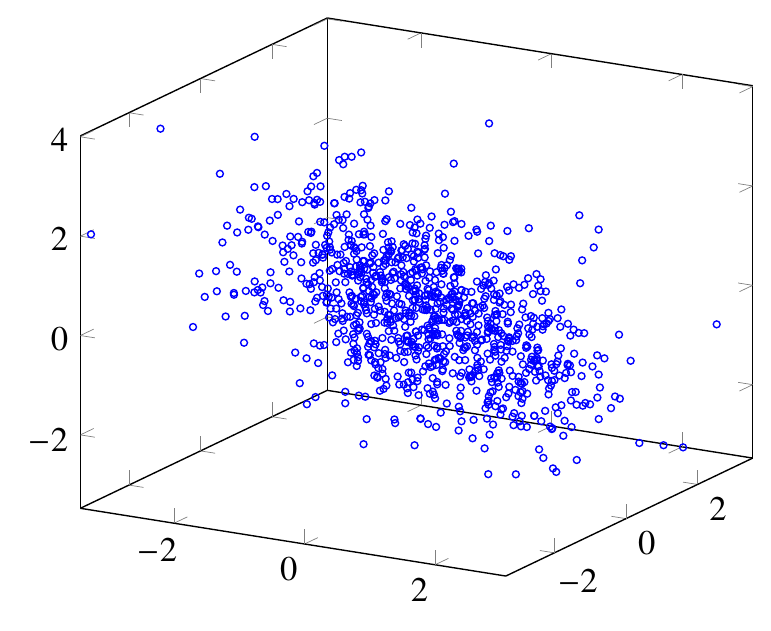}}
\caption{$(0.2,-0.8,0)$.}
\label{sfig.normal2-80}
\end{subfigure}
\hfill
\begin{subfigure}[a]{.24\textwidth}
\FIG
{\resizebox{\linewidth}{!}{
\begin{tikzpicture}
\begin{axis}[view={30}{20}]
\addplot3+[draw=none, mark=o,color=blue,mark size=1] table {normal-3-3-3.dat};
\end{axis}
\end{tikzpicture}
}}
{\includegraphics[width=\textwidth]{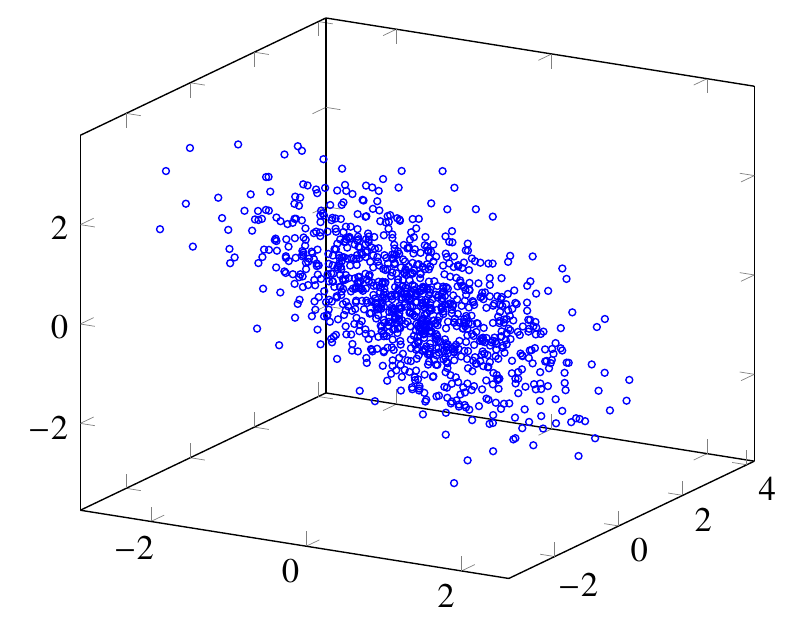}}
\caption{$(-0.3,-0.3,-0.3)$.}
\label{sfig.normal-3-3-3}
\end{subfigure}
\hfill
\begin{subfigure}[a]{.24\textwidth}
\FIG
{\resizebox{\linewidth}{!}{
\begin{tikzpicture}
\begin{axis}[view={30}{20}]
\addplot3+[draw=none, mark=o,color=blue,mark size=1] table {normal234.dat};
\end{axis}
\end{tikzpicture}
}}
{\includegraphics[width=\textwidth]{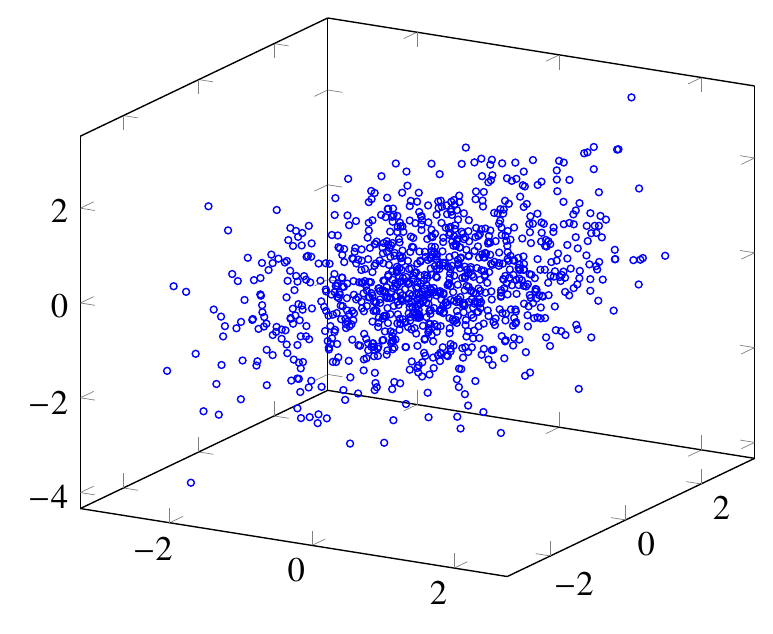}}
\caption{$(0.2,0.3,0.4)$.}
\label{sfig.normal12-9}
\end{subfigure}
\hfill
\begin{subfigure}[a]{.24\textwidth}
\FIG
{\resizebox{\linewidth}{!}{
\begin{tikzpicture}
\begin{axis}[view={30}{20}]
\addplot3+[draw=none, mark=o,color=blue,mark size=1] table {normal-1-12.dat};
\end{axis}
\end{tikzpicture}
}}
{\includegraphics[width=\textwidth]{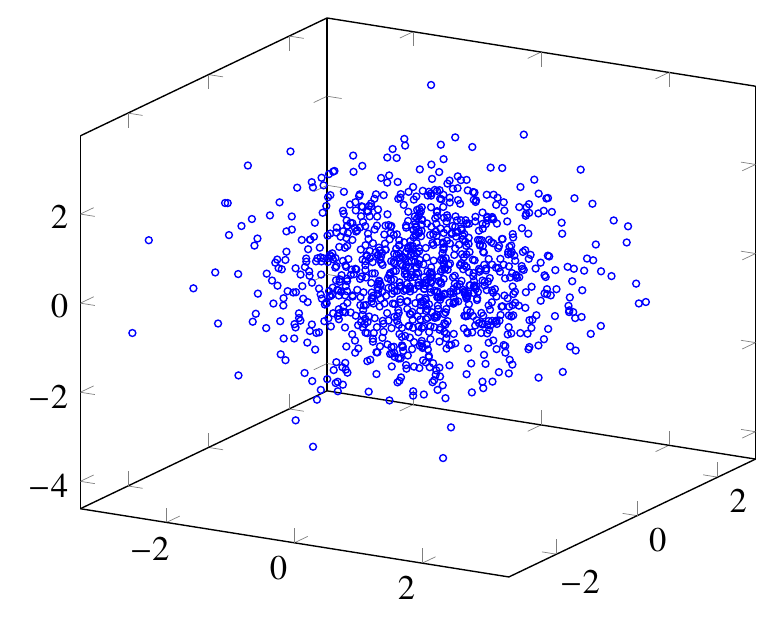}}
\caption{$(-0.1,-0.1,0.2)$.}
\label{sfig.normal-1-12}
\end{subfigure}
\caption[Scatterplots: Random samples of size $n=1000$ are drawn from the the trivariate normal distributions]{Scatterplots: Random samples of size $n=1000$ are drawn from the the trivariate normal distributions $\CN_3(\bzero,\Sig_3(\brho))$ for various correlation values of $\brho=(\rho_{12},\rho_{13},\rho_{23})$.}
\label{fig.N3Scatterplots}
\end{figure}

\subsubsection{Non-normal cases}
\label{sssec.sim.non-normal}

We simulate $r=1000$ random samples of size $n$ from some trivariate Archime\-dean copulas (Clayton, Frank, Gumbel and Joe) as well the non-Archimedean copula Farlie-Gumbel-Morgenstern ($\FGM$). For the parameters of these Archimedean copulas see the \emph{R}-package \href{http://copula.r-forge.r-project.org/}{{`copula'}} \citep[see, also,][]{BS2013}; regarding data generation from the $\FGM$ copula see Section \ref{sm:FGM} of the Supplementary Material. Based on these samples we compute the averages (and the Monte Carlo mean square error) of $\hI$ and $\hI^*$; Tables \ref{table.Sim-ArchmCop} and \ref{table.Sim-FGM} present the results for the Archimedean and F-G-M copula cases respectively; note that the copulas $C_{\theta}$ or $C_{-\theta}$, $\theta\in[0,1]$, see \eqref{eq.F-G-M(a)}, produce the same value of the index $I$, and so the index $I^*$ [the same is true for the copulas $\wC_{\theta}$ or $\wC_{-\theta}$, $\theta\in[0,1]$, see \eqref{eq.F-G-M(a)}]. For all cases the results again empirically verify the consistency of the estimation for both $\hI$ and $\hI^*$ estimators. Figure \ref{fig.N3Scatterplots} presents illustrations of graphs of a single sample $\{\bX_1,\ldots,\bX_n\}$ of size $n=1000$ drawn from these distributions. The standardized index $I^*$ indicates that the Archimedean copulas $\Clayton\{2\}$, $\Frank\{4\}$, $\Gumbel\{2\}$ and $\Joe\{2\}$ have a strong dependence level, while those corresponding to $\Clayton\{5\}$, $\Frank\{8\}$, $\Gumbel\{4\}$ and $\Joe\{5\}$ have a very strong dependence level.

\begin{table*}[htp]
 \centering{
 \caption[Monte Carlo averages (and associated mean square error) of $\hI_n$ and $\hI^*_n$; the number of Monte Carlo repetitions is $r=1000$. Random samples of size $n$ are drawn from different trivariate Archimedean copulas, for various values of $n$.]{Monte Carlo averages (and associated mean square error) of $I$ and $I^*$; the number of Monte Carlo repetitions is $r=1000$. Random samples of size $n$ are drawn from different Clayton, Frank, Gumbel and Joe Archimedean copulas, for various values of $n$. For each case, the parameter is shown in curly brackets, and the true values of $I$ and $I^*$ are denoted in the caption.}
 \label{table.Sim-ArchmCop}
 \footnotesize
\begin{subtable}{.47\textwidth}
 \caption{$\Clayton\{2\}$; $I=0.362$, $I^*=0.715$.}
 \label{subtable.Clayton2}
 \scriptsize
 \begin{tabular*}{\textwidth}
 {@{\hspace{0ex}}@{\extracolsep{\fill}}l@{\hspace{0ex}}c@{\hspace{0ex}}c@{\hspace{0ex}}c@{\hspace{0ex}}c@{\hspace{0ex}}}
 \addlinespace
 \toprule
 & \multicolumn{4}{c}{$n$}\\
 \cmidrule{2-5}
 & 100 & $200$ & $500$ & $1000$ \\
 \midrule
 $\hI_n$   & .315(.003) & .332(.001) & .348(.000) & .354(.000) \\
 $\hI^*_n$ & .637(.008) & .666(.003) & .694(.001) & .703(.000) \\
 \bottomrule
 \end{tabular*}
\end{subtable}
\hfill
\begin{subtable}{.47\textwidth}
 \caption{$\Clayton\{5\}$; $I=0.528$, $I^*=0.903$.}
 \label{subtable.Clayton5}
 \scriptsize
 \begin{tabular*}{\textwidth}
 {@{\hspace{0ex}}@{\extracolsep{\fill}}l@{\hspace{0ex}}c@{\hspace{0ex}}c@{\hspace{0ex}}c@{\hspace{0ex}}c@{\hspace{0ex}}}
 \addlinespace
 \toprule
 & \multicolumn{4}{c}{$n$}\\
 \cmidrule{2-5}
 & 100 & $200$ & $500$ & $1000$ \\
 \midrule
 $\hI_n$   & .442(.008) & .479(.003) & .507(.001) & 0.517(.000) \\
 $\hI^*_n$ & .821(.007) & .860(.002) & .886(.000) & 0.893(.000) \\
 \bottomrule
 \end{tabular*}
\end{subtable}
\bigskip\linebreak
\begin{subtable}{.47\textwidth}
 \caption{$\Frank\{4\}$; $I=0.263$, $I^*=0.543$.}
 \label{subtable.Frank4}
 \scriptsize
 \begin{tabular*}{\textwidth}
 {@{\hspace{0ex}}@{\extracolsep{\fill}}l@{\hspace{0ex}}c@{\hspace{0ex}}c@{\hspace{0ex}}c@{\hspace{0ex}}c@{\hspace{0ex}}}
 \addlinespace
 \toprule
 & \multicolumn{4}{c}{$n$}\\
 \cmidrule{2-5}
 & 100 & $200$ & $500$ & $1000$ \\
 \midrule
 $\hI_n$   & .241(.001) & .247(.001) & .255(.000) & .260(.000) \\
 $\hI^*_n$ & .497(.000) & .510(.002) & .527(.001) & .535(.000) \\
 \bottomrule
 \end{tabular*}
\end{subtable}
\hfill
\begin{subtable}{.47\textwidth}
 \caption{$\Frank\{8\}$; $I=0.410$, $I^*=0.783$.}
 \label{subtable.Frank8}
 \scriptsize
 \begin{tabular*}{\textwidth}
 {@{\hspace{0ex}}@{\extracolsep{\fill}}l@{\hspace{0ex}}c@{\hspace{0ex}}c@{\hspace{0ex}}c@{\hspace{0ex}}c@{\hspace{0ex}}}
 \addlinespace
 \toprule
 & \multicolumn{4}{c}{$n$}\\
 \cmidrule{2-5}
 & 100 & $200$ & $500$ & $1000$ \\
 \midrule
 $\hI_n$   & .354(.004) & .378(.001) & .396(.000) & .403(.000) \\
 $\hI^*_n$ & .702(.007) & .739(.002) & .764(.001) & .774(.000) \\
 \bottomrule
 \end{tabular*}
\end{subtable}
\bigskip\linebreak
\begin{subtable}{.47\textwidth}
 \caption{$\Gumbel\{2\}$; $I=0.361$, $I^*=0.714$.}
 \label{subtable.Gumbel2}
 \scriptsize
 \begin{tabular*}{\textwidth}
 {@{\hspace{0ex}}@{\extracolsep{\fill}}l@{\hspace{0ex}}c@{\hspace{0ex}}c@{\hspace{0ex}}c@{\hspace{0ex}}c@{\hspace{0ex}}}
 \addlinespace
 \toprule
 & \multicolumn{4}{c}{$n$}\\
 \cmidrule{2-5}
 & 100 & $200$ & $500$ & $1000$ \\
 \midrule
 $\hI_n$   & .316(.003) & .333(.001) & .348(.000) & .354(.000) \\
 $\hI^*_n$ & .638(.007) & .668(.003) & .693(.001) & .703(.000) \\
 \bottomrule
 \end{tabular*}
\end{subtable}
\hfill
\begin{subtable}{.47\textwidth}
 \caption{$\Gumbel\{4\}$; $I=0.562$, $I^*=0.925$.}
 \label{subtable.Gumbel4}
 \scriptsize
 \begin{tabular*}{\textwidth}
 {@{\hspace{0ex}}@{\extracolsep{\fill}}l@{\hspace{0ex}}c@{\hspace{0ex}}c@{\hspace{0ex}}c@{\hspace{0ex}}c@{\hspace{0ex}}}
 \addlinespace
 \toprule
 & \multicolumn{4}{c}{$n$}\\
 \cmidrule{2-5}
 & 100 & $200$ & $500$ & $1000$ \\
 \midrule
 $\hI_n$   & .471(.008) & .509(.003) & .539(.001) & .550(.000) \\
 $\hI^*_n$ & .853(.006) & .887(.002) & .910(.000) & .918(.000) \\
 \bottomrule
 \end{tabular*}
\end{subtable}
\bigskip\linebreak
\begin{subtable}{.47\textwidth}
 \caption{$\Joe\{2\}$; $I=0.267$, $I^*=0.549$.}
 \label{subtable.Joe2}
 \scriptsize
 \begin{tabular*}{\textwidth}
 {@{\hspace{0ex}}@{\extracolsep{\fill}}l@{\hspace{0ex}}c@{\hspace{0ex}}c@{\hspace{0ex}}c@{\hspace{0ex}}c@{\hspace{0ex}}}
 \addlinespace
 \toprule
 & \multicolumn{4}{c}{$n$}\\
 \cmidrule{2-5}
 & 100 & $200$ & $500$ & $1000$ \\
 \midrule
 $\hI_n$   & .240(.001) & .248(.001) & .258(.000) & .262(.000) \\
 $\hI^*_n$ & .495(.005) & .512(.003) & .532(.001) & .540(.000) \\
 \bottomrule
 \end{tabular*}
\end{subtable}
\hfill
\begin{subtable}{.47\textwidth}
 \caption{$\Joe\{5\}$; $I=0.496$, $I^*=0.876$.}
 \label{subtable.Joe5}
 \scriptsize
 \begin{tabular*}{\textwidth}
 {@{\hspace{0ex}}@{\extracolsep{\fill}}l@{\hspace{0ex}}c@{\hspace{0ex}}c@{\hspace{0ex}}c@{\hspace{0ex}}c@{\hspace{0ex}}}
 \addlinespace
 \toprule
 & \multicolumn{4}{c}{$n$}\\
 \cmidrule{2-5}
 & 100 & $200$ & $500$ & $1000$ \\
 \midrule
 $\hI_n$   & .416(.007) & .449(.002) & .475(.001) & .484(.000) \\
 $\hI^*_n$ & .790(.008) & .829(.003) & .857(.001) & .866(.000) \\
 \bottomrule
 \end{tabular*}
\end{subtable}
}
 \end{table*}

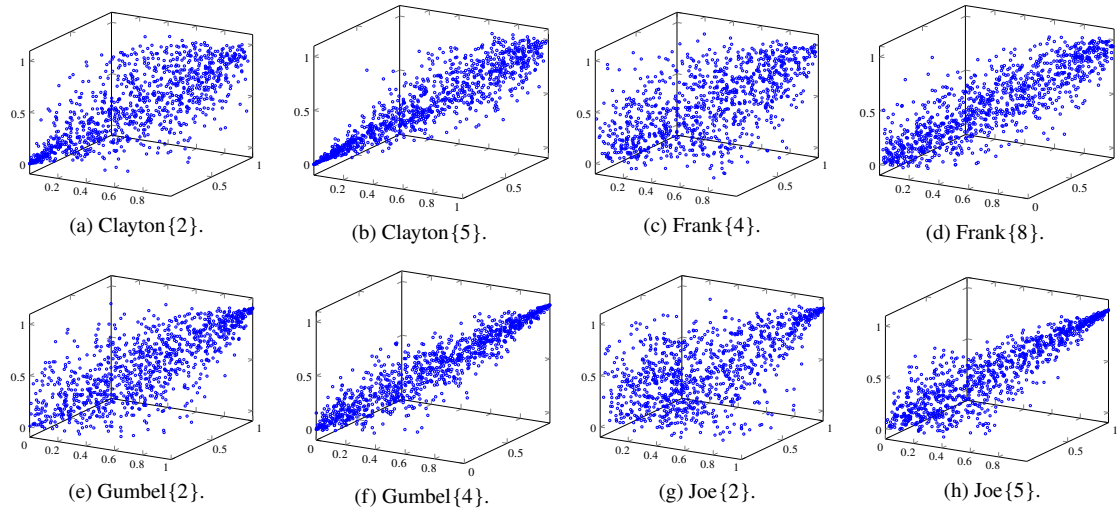
\begin{figure}[htp]
\begin{subfigure}[a]{.24\textwidth}
\FIG
{\resizebox{\linewidth}{!}{
\begin{tikzpicture}
\begin{axis}[view={30}{20}]
\addplot3+[draw=none, mark=o,color=blue,mark size=1] table {clayton2.dat};
\end{axis}
\end{tikzpicture}
}}
{\includegraphics[width=\textwidth]{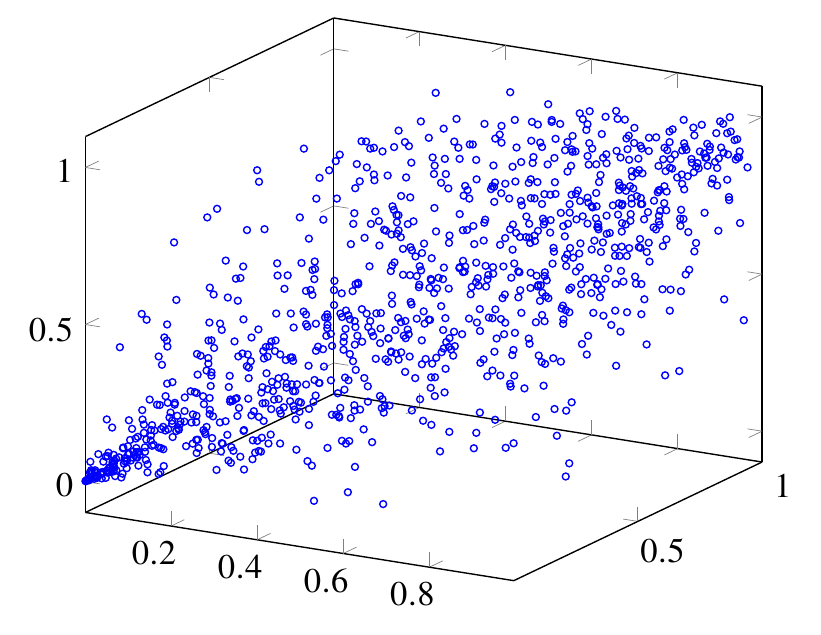}}
\caption{$\Clayton\{2\}$.}
\label{sfig.clayton2}
\end{subfigure}
\hfill
\begin{subfigure}[a]{.24\textwidth}
\FIG
{\resizebox{\linewidth}{!}{
\begin{tikzpicture}
\begin{axis}[view={30}{20}]
\addplot3+[draw=none, mark=o,color=blue,mark size=1] table {clayton5.dat};
\end{axis}
\end{tikzpicture}
}}
{\includegraphics[width=\textwidth]{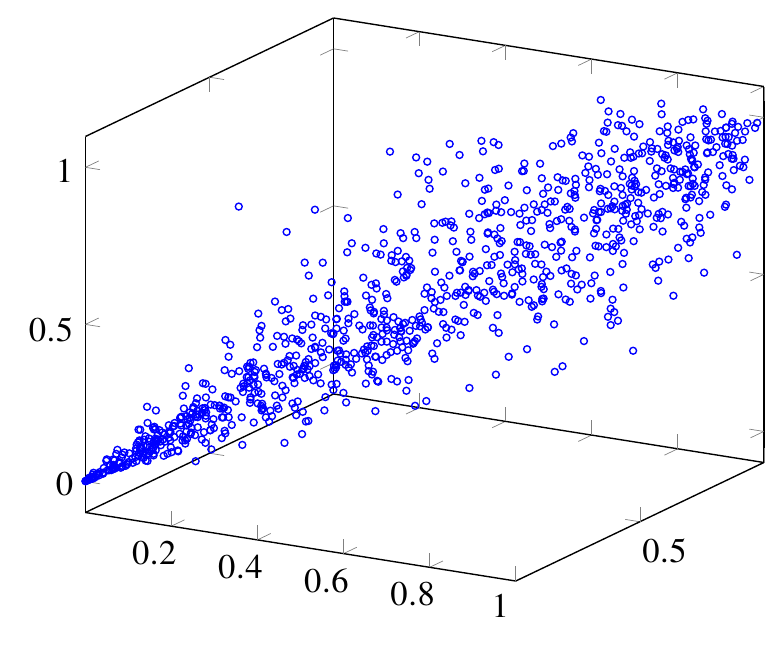}}
\caption{$\Clayton\{5\}$.}
\label{sfig.clayton5}
\end{subfigure}
\hfill
\begin{subfigure}[a]{.24\textwidth}
\FIG
{\resizebox{\linewidth}{!}{
\begin{tikzpicture}
\begin{axis}[view={30}{20}]
\addplot3+[draw=none, mark=o,color=blue,mark size=1] table {frank4.dat};
\end{axis}
\end{tikzpicture}
}}
{\includegraphics[width=\textwidth]{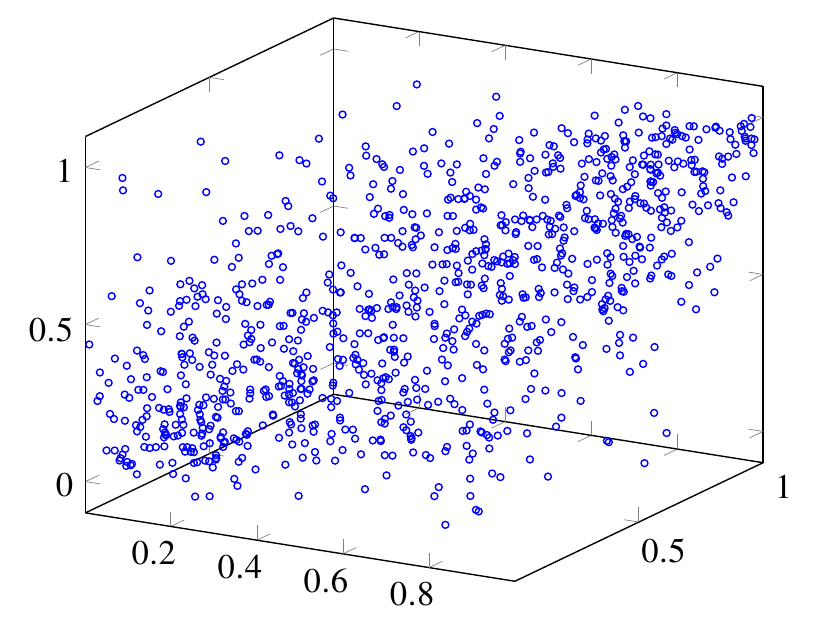}}
\caption{$\Frank\{4\}$.}
\label{sfig.frank4}
\end{subfigure}
\hfill
\begin{subfigure}[a]{.24\textwidth}
\FIG
{\resizebox{\linewidth}{!}{
\begin{tikzpicture}
\begin{axis}[view={30}{20}]
\addplot3+[draw=none, mark=o,color=blue,mark size=1] table {frank8.dat};
\end{axis}
\end{tikzpicture}
}}
{\includegraphics[width=\textwidth]{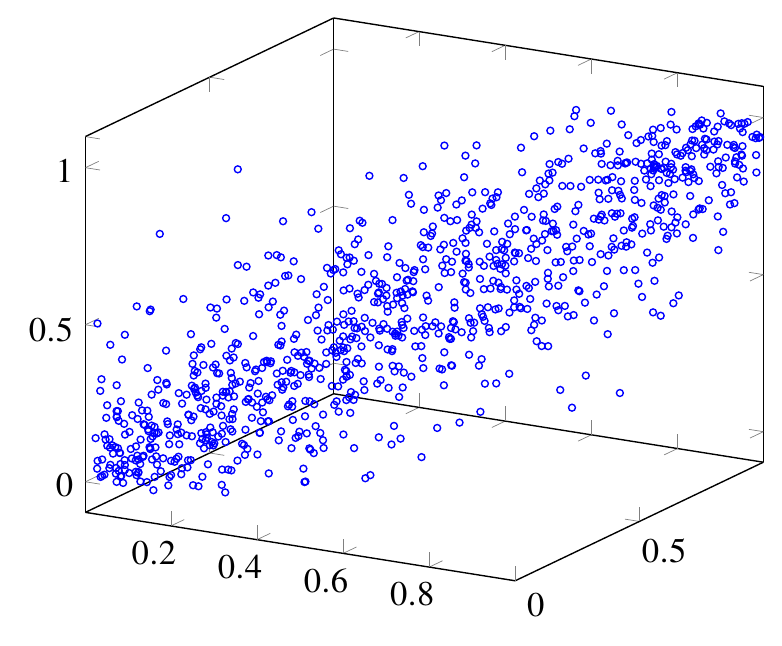}}
\caption{$\Frank\{8\}$.}
\label{sfig.frank8}
\end{subfigure}
\medskip\linebreak
\begin{subfigure}[a]{.24\textwidth}
\FIG
{\resizebox{\linewidth}{!}{
\begin{tikzpicture}
\begin{axis}[view={30}{20}]
\addplot3+[draw=none, mark=o,color=blue,mark size=1] table {gumbel2.dat};
\end{axis}
\end{tikzpicture}
}}
{\includegraphics[width=\textwidth]{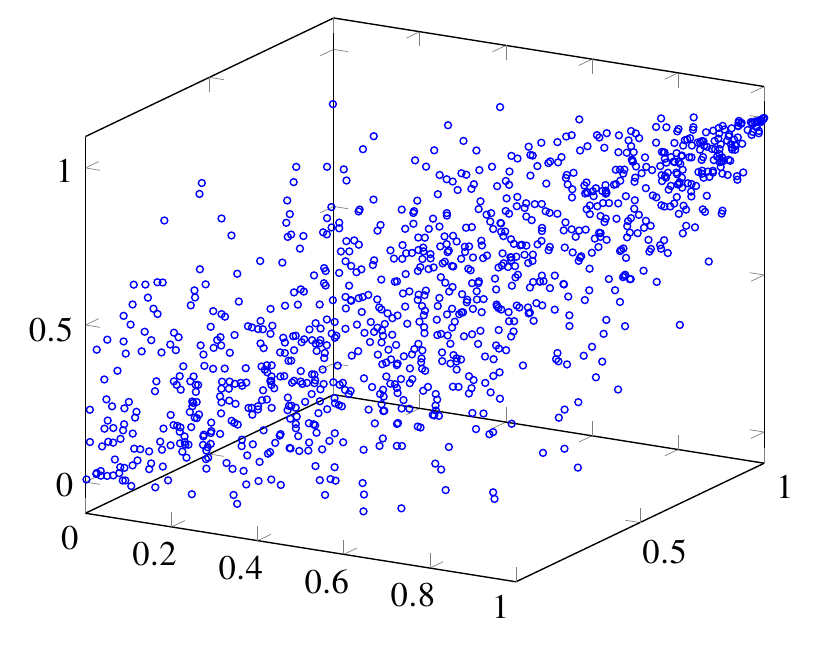}}
\caption{$\Gumbel\{2\}$.}
\label{sfig.gumbel2}
\end{subfigure}
\hfill
\begin{subfigure}[a]{.24\textwidth}
\FIG
{\resizebox{\linewidth}{!}{
\begin{tikzpicture}
\begin{axis}[view={30}{20}]
\addplot3+[draw=none, mark=o,color=blue,mark size=1] table {gumbel4.dat};
\end{axis}
\end{tikzpicture}
}}
{\includegraphics[width=\textwidth]{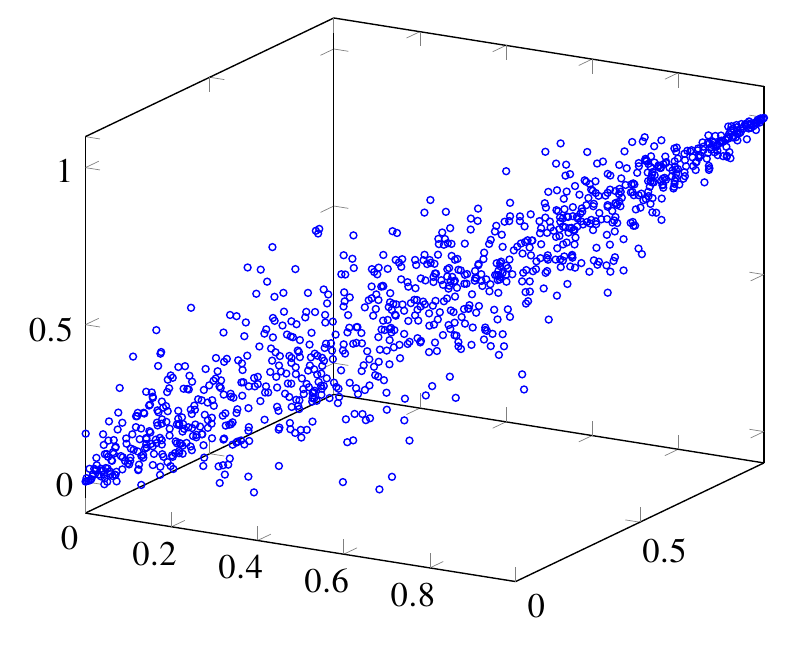}}
\caption{$\Gumbel\{4\}$.}
\label{sfig.gumbel4}
\end{subfigure}
\hfill
\begin{subfigure}[a]{.24\textwidth}
\FIG
{\resizebox{\linewidth}{!}{
\begin{tikzpicture}
\begin{axis}[view={30}{20}]
\addplot3+[draw=none, mark=o,color=blue,mark size=1] table {joe2.dat};
\end{axis}
\end{tikzpicture}
}}
{\includegraphics[width=\textwidth]{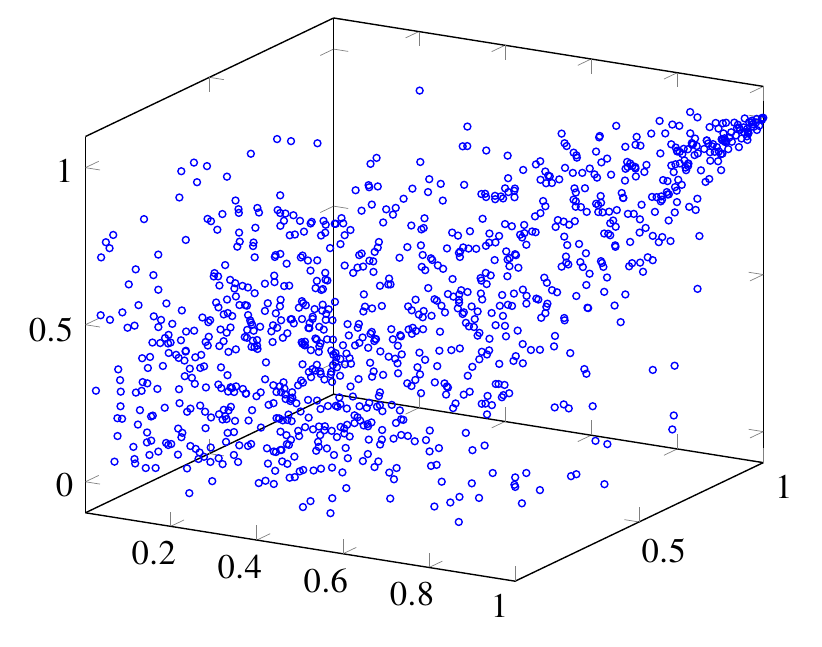}}
\caption{$\Joe\{2\}$.}
\label{sfig.joe2}
\end{subfigure}
\hfill
\begin{subfigure}[a]{.24\textwidth}
\FIG
{\resizebox{\linewidth}{!}{
\begin{tikzpicture}
\begin{axis}[view={30}{20}]
\addplot3+[draw=none, mark=o,color=blue,mark size=1] table {joe5.dat};
\end{axis}
\end{tikzpicture}
}}
{\includegraphics[width=\textwidth]{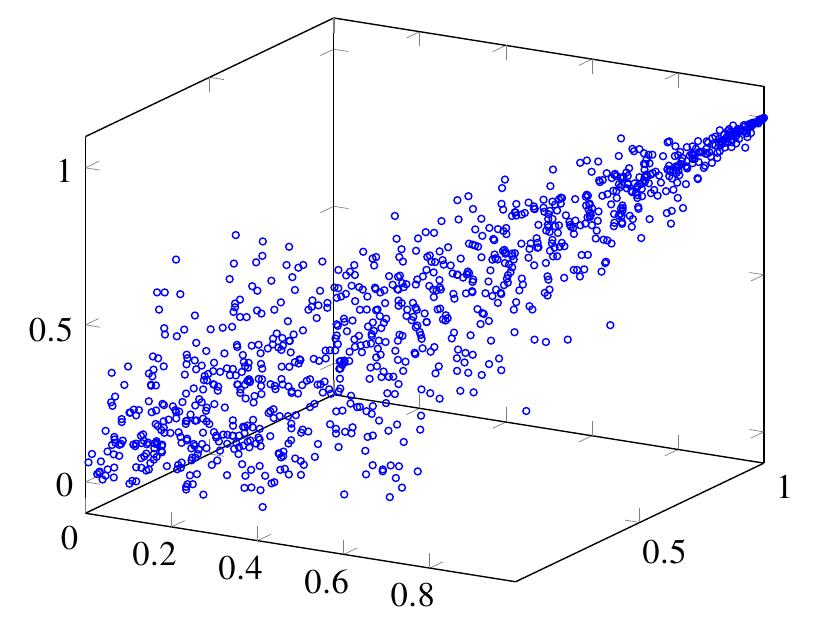}}
\caption{$\Joe\{5\}$.}
\label{sfig.joe5}
\end{subfigure}
\caption[Scatterplots: Random samples of size $n=1000$ are drawn from various trivariate copulas]{Scatterplots: Random samples of size $n=1000$ are drawn from Clayton, Frank, Gumbel and Joe trivariate copulas for various values of their parameters; the parameter value is shown in curly brackets.}
\label{fig.copulaScatterplots}
\end{figure}

\begin{table*}[htp]
 \centering{
 \caption[Monte Carlo averages (and associated mean square error) of $\hI_n$ and $\hI^*_n$; the number of Monte Carlo repetitions is $r=1000$. Random samples of size $n$ are drawn from different trivariate F-G-M copulas, for various values of $n$.]{Monte Carlo averages (and associated mean square error) of $I$ and $I^*$; the number of Monte Carlo repetitions is $r=1000$. Random samples of size $n$ are drawn from different F-G-M copulas given by \eqref{eq.F-G-M(a)} or \eqref{eq.F-G-M(b)}, for various values of $\theta$ and $n$. For each case, the true values of $I$ and $I^*$ are denoted in the caption.}
 \label{table.Sim-FGM}
 \footnotesize
\begin{subtable}{.47\textwidth}
 \caption{$C_{0.5}$; $I=0.052$, $I^*=0.093$.}
 \label{subtable.FGMa0.5}
 \scriptsize
 \begin{tabular*}{\textwidth}
 {@{\hspace{0ex}}@{\extracolsep{\fill}}l@{\hspace{0ex}}c@{\hspace{0ex}}c@{\hspace{0ex}}c@{\hspace{0ex}}c@{\hspace{0ex}}}
 \addlinespace
 \toprule
 & \multicolumn{4}{c}{$n$}\\
 \cmidrule{2-5}
 & 100 & $200$ & $500$ & $1000$ \\
 \midrule
 $\hI_n$   & .105(.003) & .075(.001) & .059(.000) & .055(.000) \\
 $\hI^*_n$ & .205(.013) & .141(.003) & .108(.001) & .100(.000) \\
 \bottomrule
 \end{tabular*}
\end{subtable}
\hfill
\begin{subtable}{.47\textwidth}
 \caption{$C_{0.7}$; $I=0.073$, $I^*=0.136$.}
 \label{subtable.FGMa0.7}
 \scriptsize
 \begin{tabular*}{\textwidth}
 {@{\hspace{0ex}}@{\extracolsep{\fill}}l@{\hspace{0ex}}c@{\hspace{0ex}}c@{\hspace{0ex}}c@{\hspace{0ex}}c@{\hspace{0ex}}}
 \addlinespace
 \toprule
 & \multicolumn{4}{c}{$n$}\\
 \cmidrule{2-5}
 & 100 & $200$ & $500$ & $1000$ \\
 \midrule
 $\hI_n$   & .110(.002) & .087(.001) & .076(.000) & .074(.000) \\
 $\hI^*_n$ & .220(.009) & .167(.002) & .144(.000) & .138(.000) \\
 \bottomrule
 \end{tabular*}
\end{subtable}
\bigskip\linebreak
\begin{subtable}{.47\textwidth}
 \caption{$C_{0.9}$; $I=0.094$, $I^*=0.180$.}
 \label{subtable..FGMa0.9}
 \scriptsize
 \begin{tabular*}{\textwidth}
 {@{\hspace{0ex}}@{\extracolsep{\fill}}l@{\hspace{0ex}}c@{\hspace{0ex}}c@{\hspace{0ex}}c@{\hspace{0ex}}c@{\hspace{0ex}}}
 \addlinespace
 \toprule
 & \multicolumn{4}{c}{$n$}\\
 \cmidrule{2-5}
 & 100 & $200$ & $500$ & $1000$ \\
 \midrule
 $\hI_n$   & .123(.001) & .102(.000) & .095(.000) & .094(.000) \\
 $\hI^*_n$ & .245(.006) & .198(.002) & .184(.000) & .181(.000) \\
 \bottomrule
 \end{tabular*}
\end{subtable}
\hfill
\begin{subtable}{.47\textwidth}
 \caption{$C_{1}$; $I=0.104$, $I^*=0.203$.}
 \label{subtable.FGMa1}
 \scriptsize
 \begin{tabular*}{\textwidth}
 {@{\hspace{0ex}}@{\extracolsep{\fill}}l@{\hspace{0ex}}c@{\hspace{0ex}}c@{\hspace{0ex}}c@{\hspace{0ex}}c@{\hspace{0ex}}}
 \addlinespace
 \toprule
 & \multicolumn{4}{c}{$n$}\\
 \cmidrule{2-5}
 & 100 & $200$ & $500$ & $1000$ \\
 \midrule
 $\hI_n$   & .129(.001) & .110(.000) & .105(.000) & .104(.000) \\
 $\hI^*_n$ & .257(.004) & .216(.001) & .204(.001) & .203(.000) \\
 \bottomrule
 \end{tabular*}
\end{subtable}
\medskip\linebreak
\begin{subtable}{.47\textwidth}
 \caption{$\wC_{0.5}$; $I=0.020$, $I^*=0.033$.}
 \label{subtable.FGMb0.5}
 \scriptsize
 \begin{tabular*}{\textwidth}
 {@{\hspace{0ex}}@{\extracolsep{\fill}}l@{\hspace{0ex}}c@{\hspace{0ex}}c@{\hspace{0ex}}c@{\hspace{0ex}}c@{\hspace{0ex}}}
 \addlinespace
 \toprule
 & \multicolumn{4}{c}{$n$}\\
 \cmidrule{2-5}
 & 100 & $200$ & $500$ & $1000$ \\
 \midrule
 $\hI_n$   & .098(.005) & .062(.001) & .037(.000) & .028(.000) \\
 $\hI^*_n$ & .189(.021) & .114(.005) & .065(.001) & .049(.000) \\
 \bottomrule
 \end{tabular*}
\end{subtable}
\hfill
\begin{subtable}{.47\textwidth}
 \caption{$\wC_{0.7}$; $I=0.028$, $I^*=0.048$.}
 \label{subtable.FGMb0.7}
 \scriptsize
 \begin{tabular*}{\textwidth}
 {@{\hspace{0ex}}@{\extracolsep{\fill}}l@{\hspace{0ex}}c@{\hspace{0ex}}c@{\hspace{0ex}}c@{\hspace{0ex}}c@{\hspace{0ex}}}
 \addlinespace
 \toprule
 & \multicolumn{4}{c}{$n$}\\
 \cmidrule{2-5}
 & 100 & $200$ & $500$ & $1000$ \\
 \midrule
 $\hI_n$   & .098(.004) & .063(.001) & .041(.000) & .034(.000) \\
 $\hI^*_n$ & .190(.016) & .117(.003) & .074(.001) & .059(.000) \\
 \bottomrule
 \end{tabular*}
\end{subtable}
\bigskip\linebreak
\begin{subtable}{.47\textwidth}
 \caption{$\wC_{0.9}$; $I=0.035$, $I^*=0.062$.}
 \label{subtable.FGMb0.9}
 \scriptsize
 \begin{tabular*}{\textwidth}
 {@{\hspace{0ex}}@{\extracolsep{\fill}}l@{\hspace{0ex}}c@{\hspace{0ex}}c@{\hspace{0ex}}c@{\hspace{0ex}}c@{\hspace{0ex}}}
 \addlinespace
 \toprule
 & \multicolumn{4}{c}{$n$}\\
 \cmidrule{2-5}
 & 100 & $200$ & $500$ & $1000$ \\
 \midrule
 $\hI_n$   & .099(.003) & .066(.001) & .046(.000) & .040(.000) \\
 $\hI^*_n$ & .192(.012) & .121(.002) & .082(.001) & .071(.000) \\
 \bottomrule
 \end{tabular*}
\end{subtable}
\hfill
\begin{subtable}{.47\textwidth}
 \caption{$\wC_{1}$; $I=0.039$, $I^*=0.070$.}
 \label{subtable.FGMb1}
 \scriptsize
 \begin{tabular*}{\textwidth}
 {@{\hspace{0ex}}@{\extracolsep{\fill}}l@{\hspace{0ex}}c@{\hspace{0ex}}c@{\hspace{0ex}}c@{\hspace{0ex}}c@{\hspace{0ex}}}
 \addlinespace
 \toprule
 & \multicolumn{4}{c}{$n$}\\
 \cmidrule{2-5}
 & 100 & $200$ & $500$ & $1000$ \\
 \midrule
 $\hI_n$   & .100(.002) & .067(.001) & .049(.000) & .043(.000) \\
 $\hI^*_n$ & .194(.010) & .125(.002) & .088(.001) & .077(.000) \\
 \bottomrule
 \end{tabular*}
\end{subtable}
}
 \end{table*}

\begin{figure}[htp]
\begin{subfigure}[a]{.24\textwidth}
\FIG
{\resizebox{\linewidth}{!}{
\begin{tikzpicture}
\begin{axis}[view={30}{20}]
\addplot3+[draw=none, mark=o,color=blue,mark size=1] table {FGMa5.dat};
\end{axis}
\end{tikzpicture}
}}
{\includegraphics[width=\textwidth]{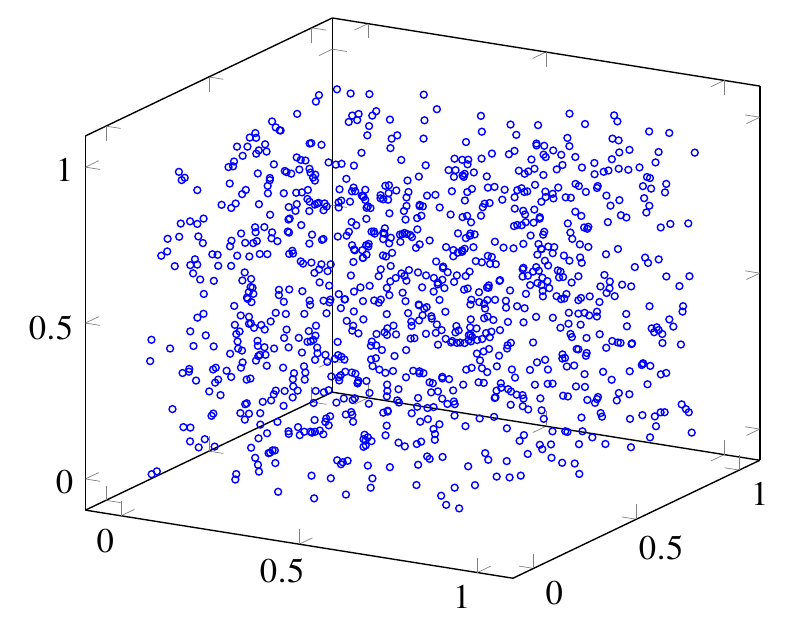}}
\caption{$C_{0.5}$.}
\label{sfig.FGMa0.5}
\end{subfigure}
\hfill
\begin{subfigure}[a]{.24\textwidth}
\FIG
{\resizebox{\linewidth}{!}{
\begin{tikzpicture}
\begin{axis}[view={30}{20}]
\addplot3+[draw=none, mark=o,color=blue,mark size=1] table {FGMa7.dat};
\end{axis}
\end{tikzpicture}
}}
{\includegraphics[width=\textwidth]{FGMa5}}
\caption{$C_{0.7}$.}
\label{sfig.FGMa0.7}
\end{subfigure}
\hfill
\begin{subfigure}[a]{.24\textwidth}
\FIG
{\resizebox{\linewidth}{!}{
\begin{tikzpicture}
\begin{axis}[view={30}{20}]
\addplot3+[draw=none, mark=o,color=blue,mark size=1] table {FGMa9.dat};
\end{axis}
\end{tikzpicture}
}}
{\includegraphics[width=\textwidth]{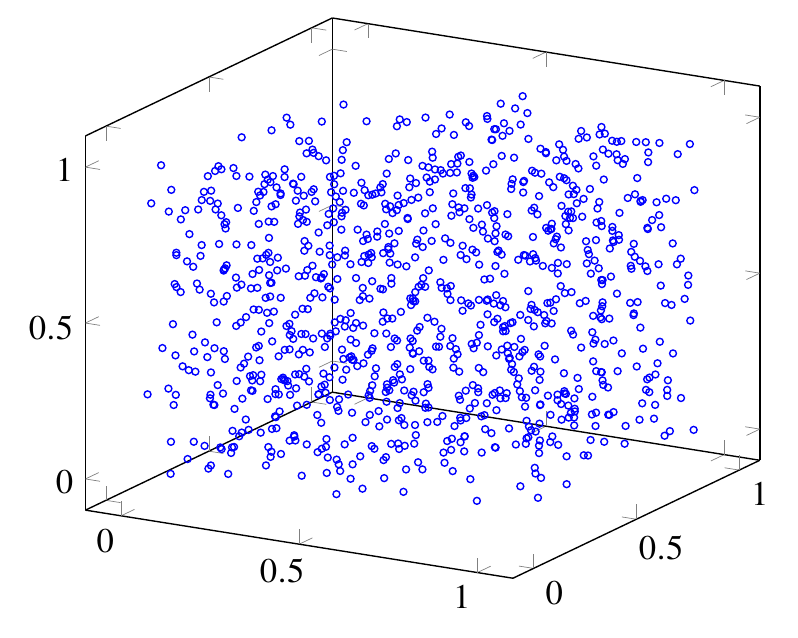}}
\caption{$C_{0.9}$.}
\label{sfig.FGMa0.9}
\end{subfigure}
\hfill
\begin{subfigure}[a]{.24\textwidth}
\FIG
{\resizebox{\linewidth}{!}{
\begin{tikzpicture}
\begin{axis}[view={30}{20}]
\addplot3+[draw=none, mark=o,color=blue,mark size=1] table {FGMa10.dat};
\end{axis}
\end{tikzpicture}
}}
{\includegraphics[width=\textwidth]{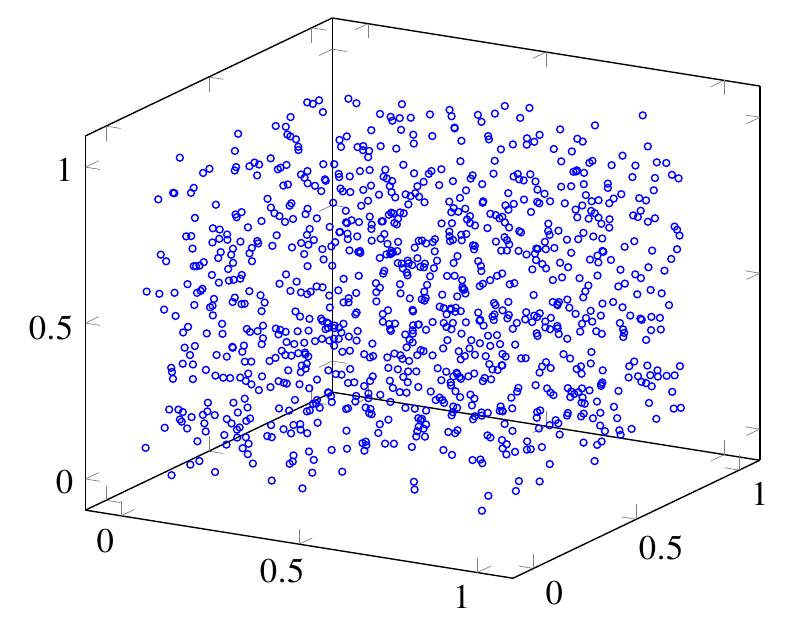}}
\caption{$C_{1}$.}
\label{sfig.FGMa1}
\end{subfigure}
\medskip\linebreak
\begin{subfigure}[a]{.24\textwidth}
\FIG
{\resizebox{\linewidth}{!}{
\begin{tikzpicture}
\begin{axis}[view={30}{20}]
\addplot3+[draw=none, mark=o,color=blue,mark size=1] table {FGMb5.dat};
\end{axis}
\end{tikzpicture}
}}
{\includegraphics[width=\textwidth]{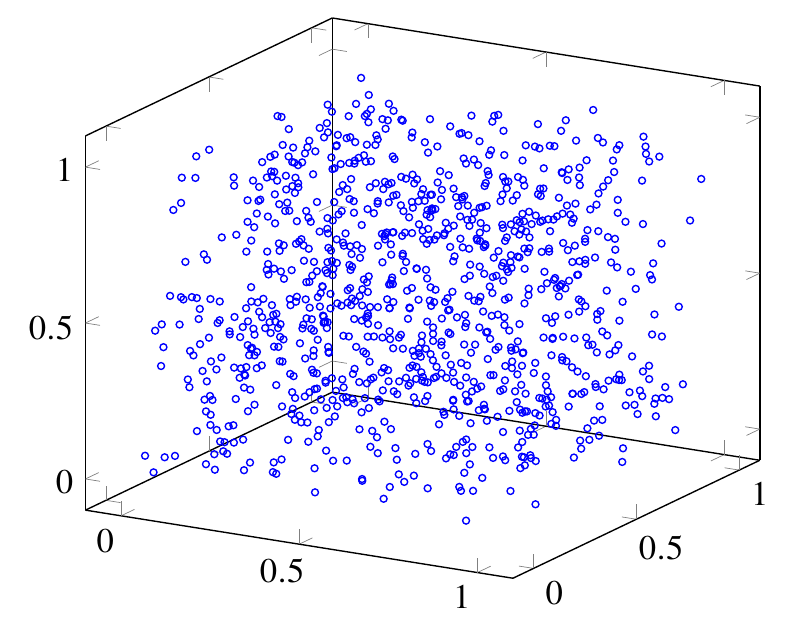}}
\caption{$\wC_{0.5}$.}
\label{sfig.FGMb0.5}
\end{subfigure}
\hfill
\begin{subfigure}[a]{.24\textwidth}
\FIG
{\resizebox{\linewidth}{!}{
\begin{tikzpicture}
\begin{axis}[view={30}{20}]
\addplot3+[draw=none, mark=o,color=blue,mark size=1] table {FGMb7.dat};
\end{axis}
\end{tikzpicture}
}}
{\includegraphics[width=\textwidth]{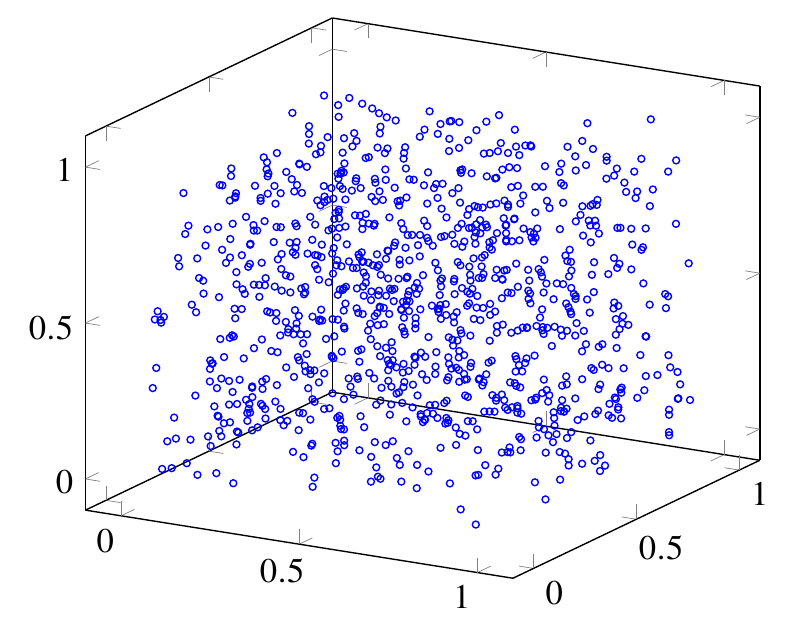}}
\caption{$\wC_{0.7}$.}
\label{sfig.FGMb0.7}
\end{subfigure}
\hfill
\begin{subfigure}[a]{.24\textwidth}
\FIG
{\resizebox{\linewidth}{!}{
\begin{tikzpicture}
\begin{axis}[view={30}{20}]
\addplot3+[draw=none, mark=o,color=blue,mark size=1] table {FGMb9.dat};
\end{axis}
\end{tikzpicture}
}}
{\includegraphics[width=\textwidth]{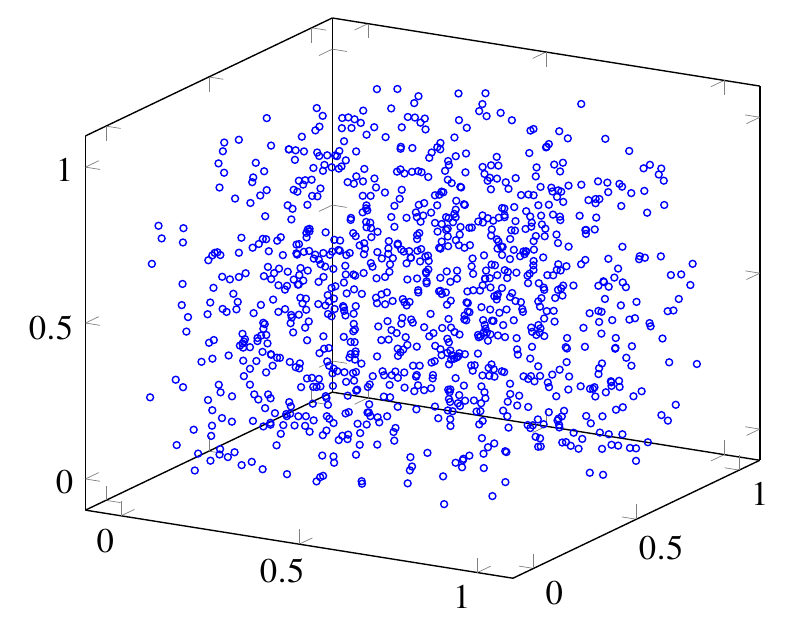}}
\caption{$\wC_{0.9}$.}
\label{sfig.FGMb0.9}
\end{subfigure}
\hfill
\begin{subfigure}[a]{.24\textwidth}
\FIG
{\resizebox{\linewidth}{!}{
\begin{tikzpicture}
\begin{axis}[view={30}{20}]
\addplot3+[draw=none, mark=o,color=blue,mark size=1] table {FGMb10.dat};
\end{axis}
\end{tikzpicture}
}}
{\includegraphics[width=\textwidth]{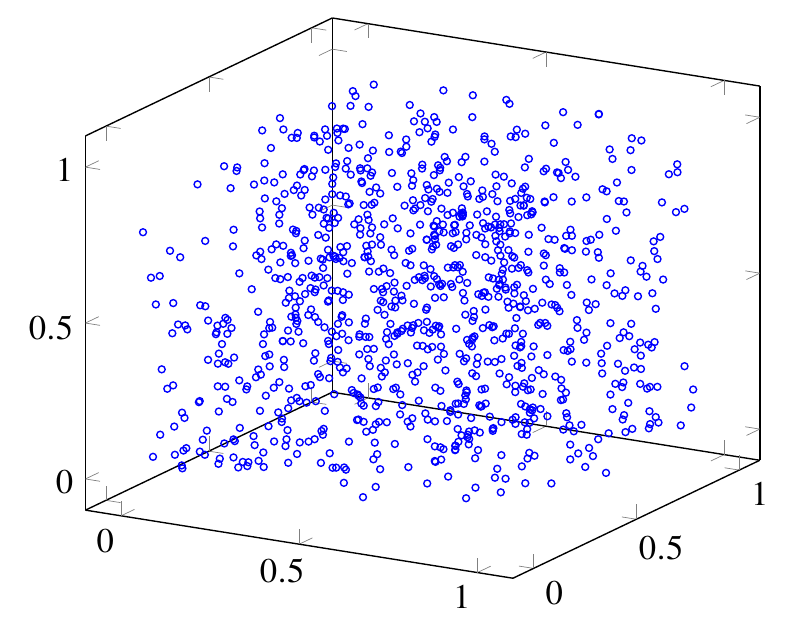}}
\caption{$\wC_{1}$.}
\label{sfig.FGMb1}
\end{subfigure}
\caption[Scatterplots: Random samples of size $n=1000$ are drawn from various trivariate copulas]{Scatterplots: Random samples of size $n=1000$ are drawn from different trivariate F-G-M copulas $C_{\theta}$ and $\wC_{\theta}$ for various values $\theta$, see \eqref{eq.F-G-M}.}
\label{fig.FGM.Scatterplots}
\end{figure}

{\bf General Observations:}
1) As the sample size increases the bias decreases;
2) The highest percentage of bias, across all sample sizes, is observed in the case of complete dependence $(\rho_{12},\rho_{13},\rho_{23})=(1,1,1)$. In this case $I=I^*=1$. The observed bias (defined as $\hI-I$) goes from $-26.1\%$ ($n=100$, $I$) to $-0.1\%$ ($n=1000$, $I^*$).

\subsection{Evaluating the proposed test of independence}
\label{ssec.sim.test}

We now investigate the performance of the test of independence, introduced in Section \ref{sec.test}. Performance will be investigated in terms of both, level and power. The dimensions of the generated data are $d=2,6$. We simulate normal and non-normal data distributions. We choose to work with the bivariate case because it allows comparisons between numerical results produced and the shape of the data distributions. Figure \ref{Fig.BVdistributions} illustrates graphically the shape of a single random sample of size 1000 drawn from bivariate non-normal distributions we work with, with various degrees of dependence.

The results shown in Table \ref{table.power.normal} verify the consistency of the proposed independence test, in the bivariate normal case. As the sample size increases to infinity the power of the test increases to 1. The rate of convergence of the power to 1 depends on the strength of the association between the components of the random vector. %For sufficiently large sample sizes the power of the test is either 1 or almost 1 independently of the rotation of the data distribution; while, for small sample sizes this power depends on the rotation of the data distribution. This is an expected observation because the $\AUK$ depends on the rotations of the data distribution. The correction of this abnormal behavior of the power of the test in small samples constitutes an interesting problem for future work.

Table \ref{table.power.normal} also presents a comparison between the proposed test and the test for independence based on the estimated \dHSIC statistic. We simulated this test using the R package {\tt`dHSIC'} by \citet{PP2019}. Our test exhibits higher power than the \dHSIC-based test for all smaller samples and smaller values of $\rho$, while it exhibits the same power with the \dHSIC-based test for larger samples and values of $\rho$.

\begin{table}[htp]
\caption[Empirical rejection rates (\%) at the $\alpha=5\%$ significance level of $\AUK$ and \dHSIC independence tests, calculated from $1000$ random samples of size $n$ simulated from different bivariate normal distributions.]{Empirical rejection rates (\%) at the $\alpha=5\%$ significance level of $\AUK$ and \dHSIC independence tests, calculated from $1000$ random samples of size $n$ simulated from $\CN_2\left(\bm{0},\Sig_2(\rho)\right)$ distribution for various values of $\rho$. The first two rows correspond to the empirical level of significance of the tests, while the other ones correspond to the empirical power of the tests.}
 \label{table.power.normal}
 \small
 \begin{tabular*}{\textwidth}
 {@{\hspace{0ex}}@{\extracolsep{\fill}}l@{}l@{}S@{}S@{}S@{}S@{}S@{}S@{}S@{}S@{}S@{\hspace{-1.3ex}}}
 \toprule
 \raisebox{-1ex}{$\rho$} & \raisebox{-1ex}{test} \quad \ \ $n$ & {50} & {100} & {200} & {300} & {500} & {750} & {1000} & {1500} & {2000} \\
 \midrule
 \multirow{2}{*}{$0$}
 & $\AUK$ & 5.2 & 4.9 & 4.8 & 4.9 & 5   & 5.1 & 4.9 & 5.1 & 5   \\
 [-.4ex]
 & \dHSIC & 5.5 & 5.9 & 5.1 & 4.7 & 4.3 & 4.7 & 5.5 & 5.4 & 5.2 \\
 [.6ex]
 \multirow{2}{*}{$0.1$}
 & $\AUK$ & 14.7 & 22   & 32.1 & 42.9 & 62.1 & 76.6 & 86.5 & 95.2 & 98.2 \\
 [-.4ex]
 & \dHSIC & 6.8  & 8.1  & 12.7 & 15   & 24.2 & 37.9 & 48.4 & 70.3 & 80.3 \\
 [.6ex]
 \multirow{2}{*}{$0.2$}
 & $\AUK$ & 35.2 & 54.5 & 80.4 & 93.3 & 99   & 99.9 & 100  & 100 & 100 \\
 [-.4ex]
 & \dHSIC & 13.2 & 23.3 & 40.1 & 59.6 & 84.4 & 95.8 & 99.3 & 100 & 100 \\
 [.6ex]
 \multirow{2}{*}{$0.3$}
 & $\AUK$ & 59.2 & 84.8 & 98.9 & 100  & 100  & 100 & 100 & 100 & 100 \\
 [-.4ex]
 & \dHSIC & 27.7 & 43.3 & 80.2 & 94.8 & 99.8 & 100 & 100 & 100 & 100 \\
 [.6ex]
 \multirow{2}{*}{$0.4$}
 & $\AUK$ & 84.6 & 97.6 & 100  & 100  & 100 & 100 & 100 & 100 & 100 \\
 [-.4ex]
 & \dHSIC & 46.3 & 79   & 98.1 & 99.9 & 100 & 100 & 100 & 100 & 100 \\
 [.6ex]
 \multirow{2}{*}{$0.5$}
 & $\AUK$ & 95.6 & 99.9 & 100  & 100 & 100 & 100 & 100 & 100 & 100 \\
 [-.4ex]
 & \dHSIC & 71   & 96.6 & 99.9 & 100 & 100 & 100 & 100 & 100 & 100 \\
 [.6ex]
 \multirow{2}{*}{$0.6$}
 & $\AUK$ & 99.4 & 100 & 100 & 100 & 100 & 100 & 100 & 100 & 100 \\
 [-.4ex]
 & \dHSIC & 92.6 & 100 & 100 & 100 & 100 & 100 & 100 & 100 & 100 \\
  \bottomrule
 \end{tabular*}
 \end{table}

 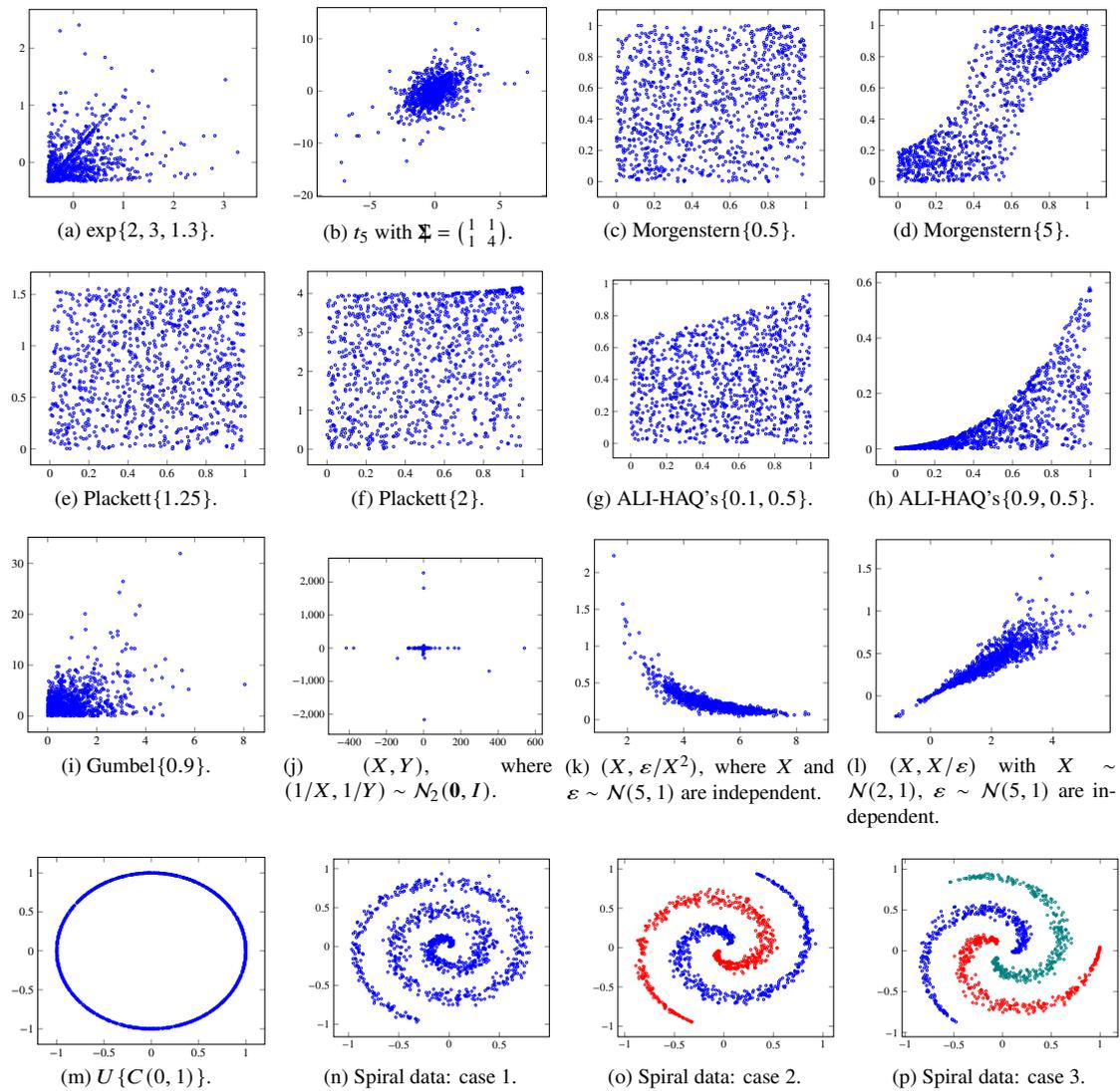
\begin{figure}[htp]
 \centering
 \begin{subfigure}[t]{.24\textwidth}
 \centering
\FIG
{\resizebox{\linewidth}{!}{
\begin{tikzpicture}[scale=.94\tikzscale]
 \begin{axis}
 \addplot[blue,only marks, mark=o, mark options={scale=.5}] table [x index=2,y index=3] {coordinates.dat};
 \end{axis}
\end{tikzpicture}
}}
{\includegraphics[width=.94\linewidth]{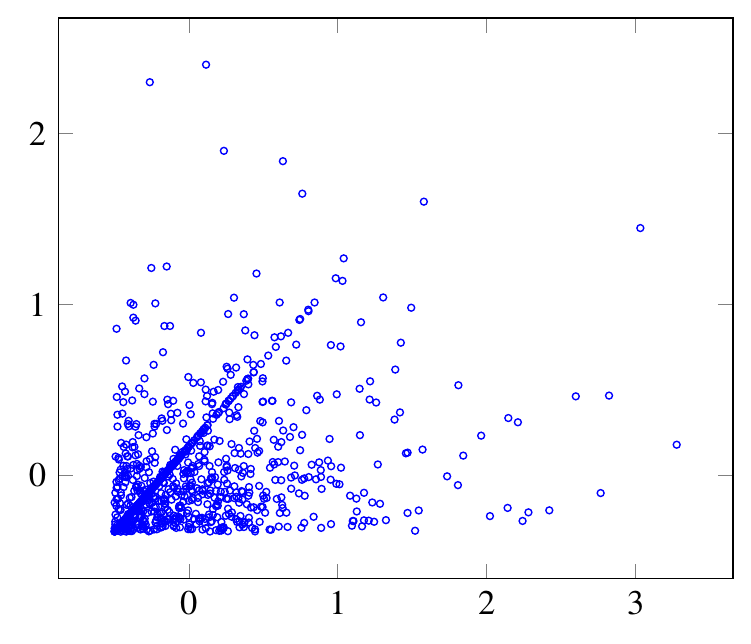}}
\caption{$\exp\{2,3,1.3\}$.}
 \label{subfig.BExp}
 \end{subfigure}
 \hfill
  \begin{subfigure}[t]{.24\textwidth}
 	\centering
 	\FIG
 	{\resizebox{\linewidth}{!}{
 			\begin{tikzpicture}[scale=\tikzscale]
 				\begin{axis}
 					\addplot[blue,only marks, mark=o, mark options={scale=.5}] table [x index=17,y index=18] {coordinates.dat};
 				\end{axis}
 			\end{tikzpicture}
 	}}
 	{\includegraphics[width=\linewidth]{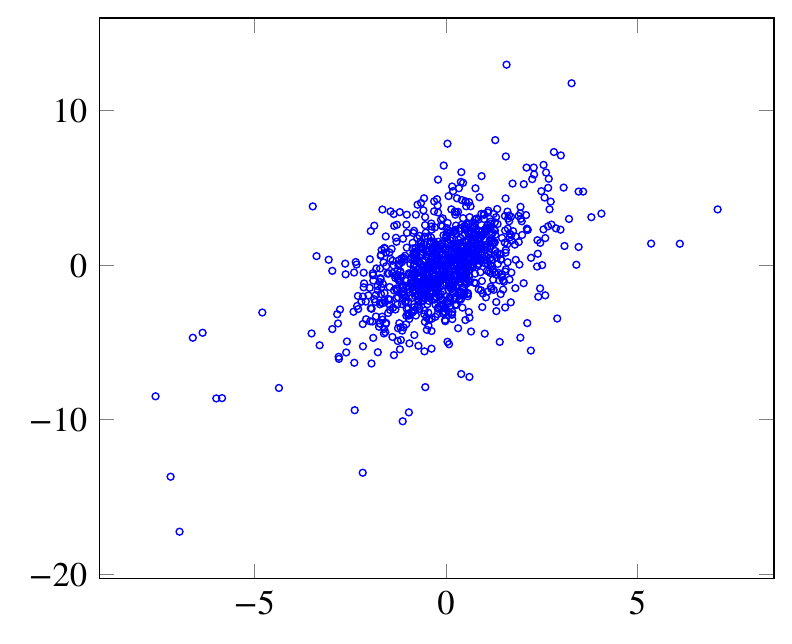}}
 	\caption{$t_5$ with $\Sig=\big({1\atop1}~{1\atop4}\big)$.}
 	\label{subfig.BVt5}
 \end{subfigure}
 \hfill
 \begin{subfigure}[t]{.24\textwidth}
 \centering
\FIG
{\resizebox{\linewidth}{!}{
\begin{tikzpicture}[scale=.97\tikzscale]
 \begin{axis}
 \addplot[blue,only marks, mark=o, mark options={scale=.5}] table [x index=6,y index=7] {coordinates.dat};
 \end{axis}
\end{tikzpicture}
}}
{\includegraphics[width=.97\linewidth]{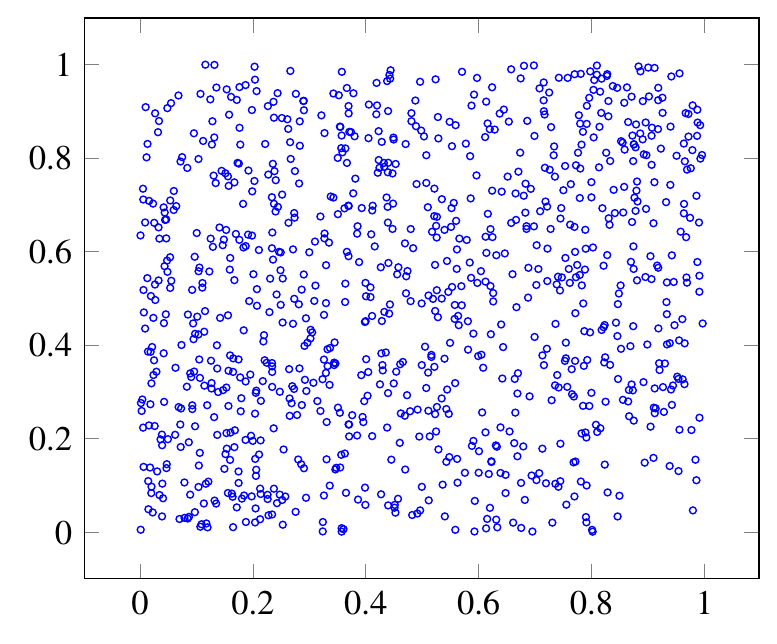}}
 \caption{$\Morgenstern\{0.5\}$.}
 \label{subfig.BVMorgenstern0.5}
 \end{subfigure}
\hfill
 \begin{subfigure}[t]{.24\textwidth}
 \centering
\FIG
{\resizebox{\linewidth}{!}{
\begin{tikzpicture}[scale=.97\tikzscale]
 \begin{axis}
 \addplot[blue,only marks, mark=o, mark options={scale=.5}] table [x index=8,y index=9] {coordinates.dat};
 \end{axis}
\end{tikzpicture}
}}
{\includegraphics[width=.97\linewidth]{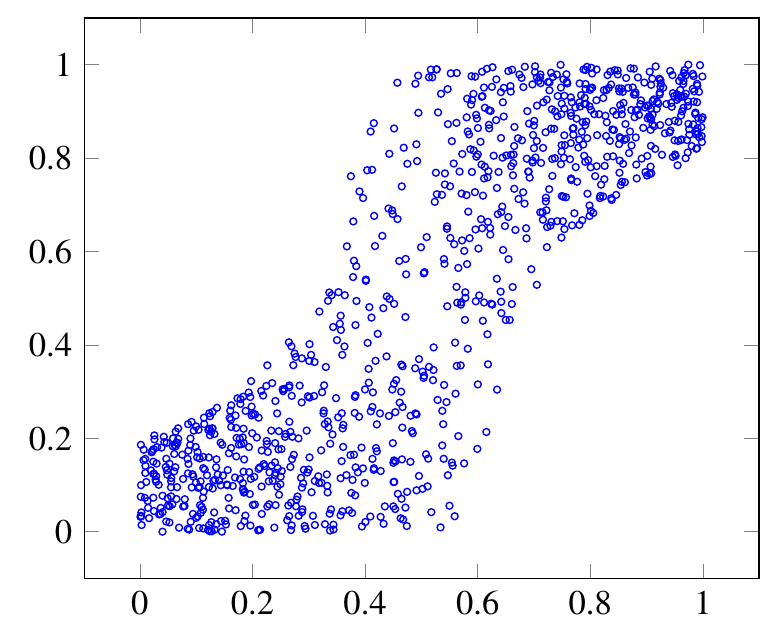}}
 \caption{$\Morgenstern\{5\}$.}
 \label{subfig.BVMorgenstern5}
 \end{subfigure}
 \medskip\linebreak
 \begin{subfigure}[t]{.24\textwidth}
 \centering
\FIG
{\resizebox{\linewidth}{!}{
\begin{tikzpicture}[scale=\tikzscale]
 \begin{axis}
 \addplot[blue,only marks, mark=o, mark options={scale=.5}] table [x index=10,y index=11] {coordinates.dat};
 \end{axis}
\end{tikzpicture}
}}
{\includegraphics[width=\linewidth]{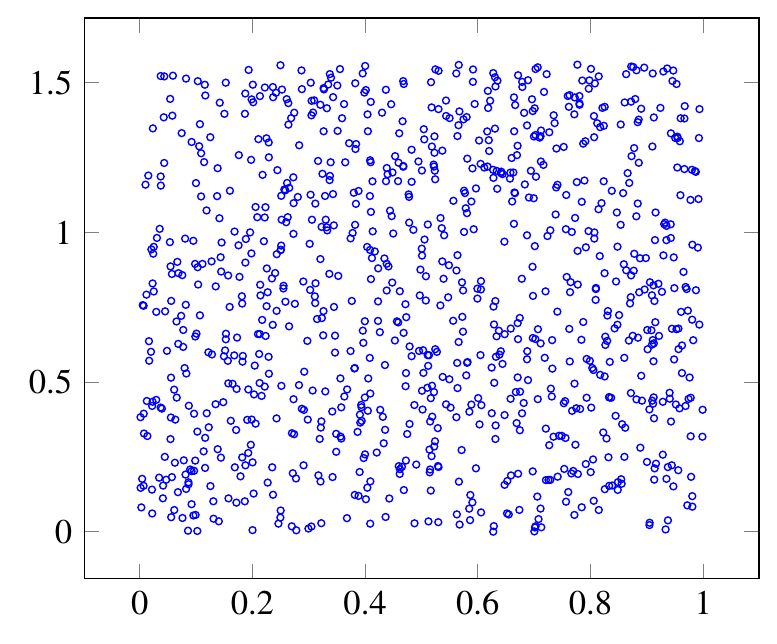}}
 \caption{$\Plackett\{1.25\}$.}
 \label{subfig.BVPlackett1.25}
 \end{subfigure}
 \hfill
 \begin{subfigure}[t]{.24\textwidth}
 \centering
\FIG
{\resizebox{\linewidth}{!}{
\begin{tikzpicture}[scale=.97\tikzscale]
 \begin{axis}
 \addplot[blue,only marks, mark=o, mark options={scale=.5}] table [x index=12,y index=13] {coordinates.dat};
 \end{axis}
\end{tikzpicture}
}}
{\includegraphics[width=.97\linewidth]{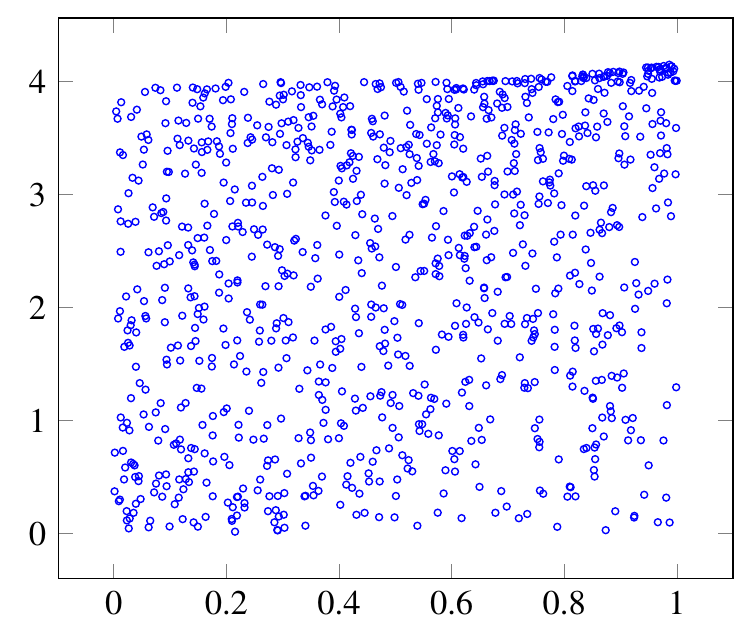}}
 \caption{$\Plackett\{2\}$.}
 \label{subfig.BVPlackett2}
 \end{subfigure}
 \hfill
 \begin{subfigure}[t]{.24\textwidth}
 \centering
\FIG
{\resizebox{\linewidth}{!}{
\begin{tikzpicture}[scale=\tikzscale]
 \begin{axis}
 \addplot[blue,only marks, mark=o, mark options={scale=.5}] table [x index=14,y index=6] {coordinates.dat};
 \end{axis}
\end{tikzpicture}
}}
{\includegraphics[width=\linewidth]{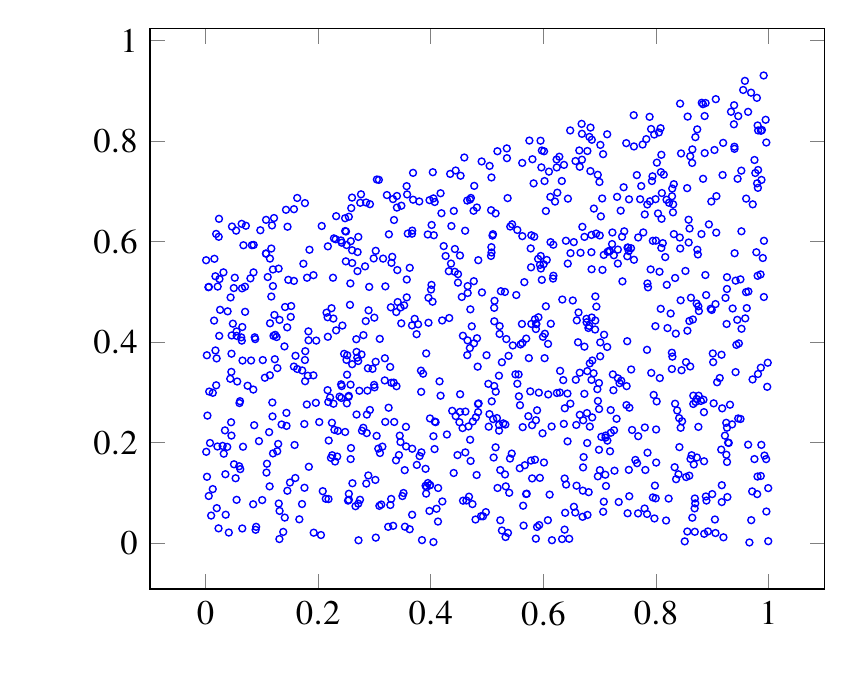}}
 \caption{$\ALIHAQ\{0.1,0.5\}$.}
 \label{subfig.BVALI-HAQ's1-5}
 \end{subfigure}
 \hfill
 \begin{subfigure}[t]{.24\textwidth}
	\centering
	\FIG
	{\resizebox{\linewidth}{!}{
			\begin{tikzpicture}[scale=\tikzscale]
				\begin{axis}
					\addplot[blue,only marks, mark=o, mark options={scale=.5}] table [x index=14,y index=6] {coordinates.dat};
				\end{axis}
			\end{tikzpicture}
	}}
	{\includegraphics[width=\linewidth]{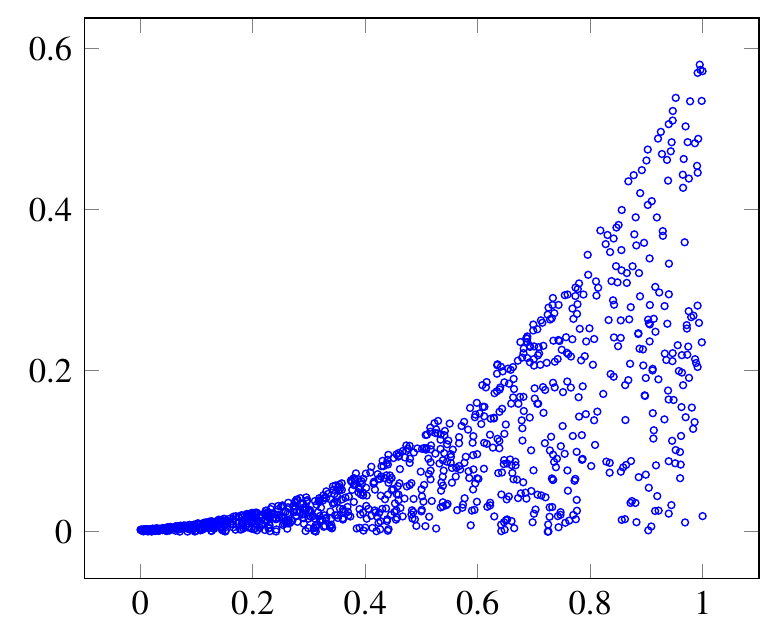}}
	\caption{$\ALIHAQ\{0.9,0.5\}$.}
	\label{subfig.BVALI-HAQ's9-5}
\end{subfigure}
 \medskip\linebreak
  \begin{subfigure}[t]{.24\textwidth}
 	\centering
 	\FIG
 	{\resizebox{\linewidth}{!}{
 			\begin{tikzpicture}[scale=\tikzscale]
 				\begin{axis}
 					\addplot[blue,only marks, mark=o, mark options={scale=.5}] table [x index=15,y index=16] {coordinates.dat};
 				\end{axis}
 			\end{tikzpicture}
 	}}
 	{\includegraphics[width=\linewidth]{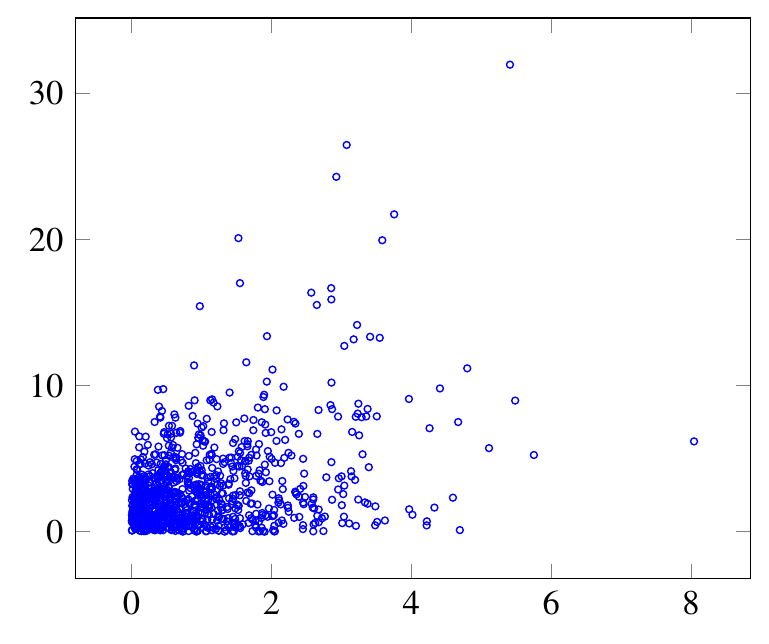}}
 	\caption{$\Gumbel\{0.9\}$.}
 	\label{subfig.BVGumbel}
 \end{subfigure}
 \hfill
 \begin{subfigure}[t]{.24\textwidth}
 \centering
\FIG
{\resizebox{\linewidth}{!}{
\begin{tikzpicture}[scale=\tikzscale]
 \begin{axis}
 \addplot[blue,only marks, mark=o, mark options={scale=.5}] table [x index=26,y index=27] {coordinates.dat};
 \end{axis}
\end{tikzpicture}
}}
{\includegraphics[width=\linewidth]{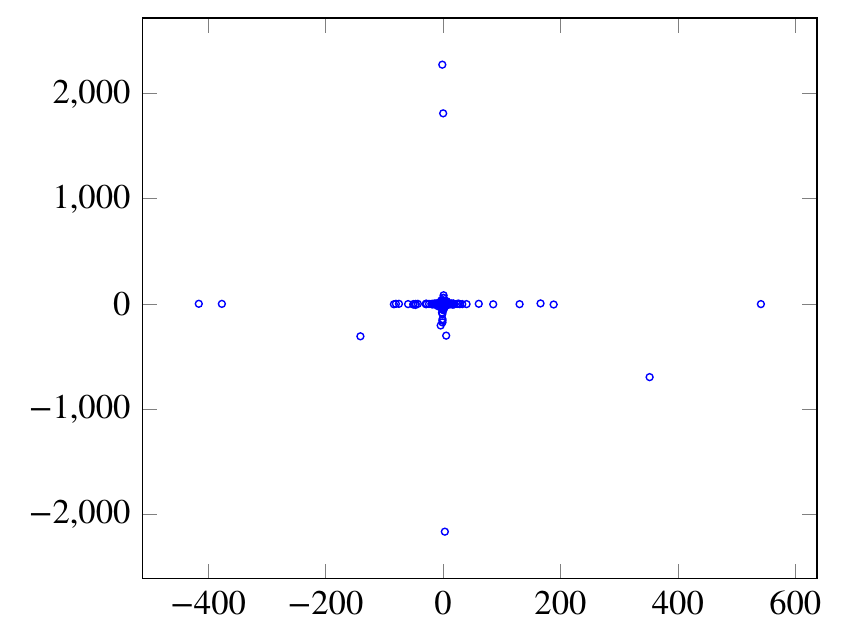}}
 \caption{$(X,Y)$, where $(1/X,1/Y)\sim \CN_2(\bm{0},I)$.}
 \label{subfig.BVinverseNormals}
 \end{subfigure}
 \hfill
 \begin{subfigure}[t]{.24\textwidth}
 \centering
\FIG
{\resizebox{\linewidth}{!}{
\begin{tikzpicture}[scale=\tikzscale]
 \begin{axis}
 \addplot[blue,only marks, mark=o, mark options={scale=.5}] table [x index=28,y index=29] {coordinates.dat};
 \end{axis}
\end{tikzpicture}
}}
{\includegraphics[width=\linewidth]{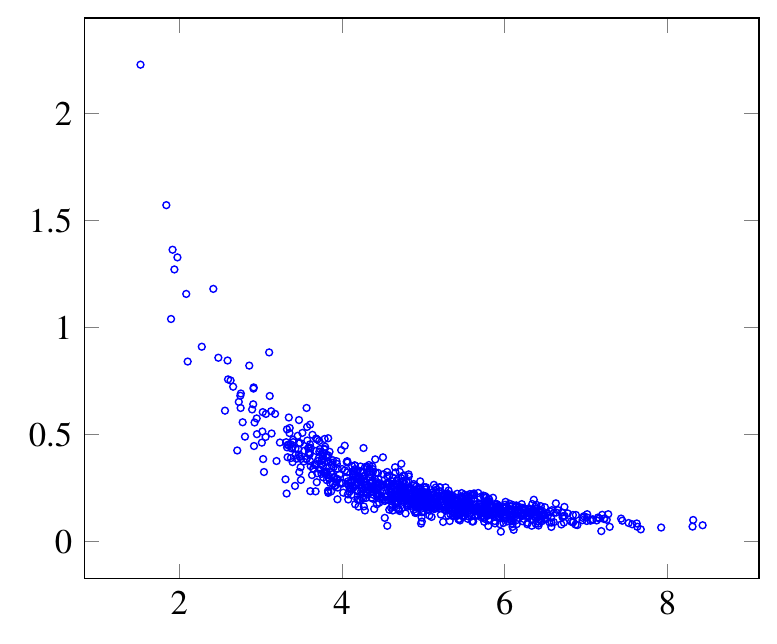}}
 \caption{$(X,\varepsilon/X^2)$, where $X$ and $\varepsilon\sim\CN(5,1)$ are independent.}
 \label{subfig.noise1}
 \end{subfigure}
 \hfill
 \begin{subfigure}[t]{.24\textwidth}
 \centering
\FIG
{\resizebox{\linewidth}{!}{
\begin{tikzpicture}[scale=\tikzscale]
 \begin{axis}
 \addplot[blue,only marks, mark=o, mark options={scale=.5}] table [x index=30,y index=31] {coordinates.dat};
 \end{axis}
\end{tikzpicture}
}}
{\includegraphics[width=\linewidth]{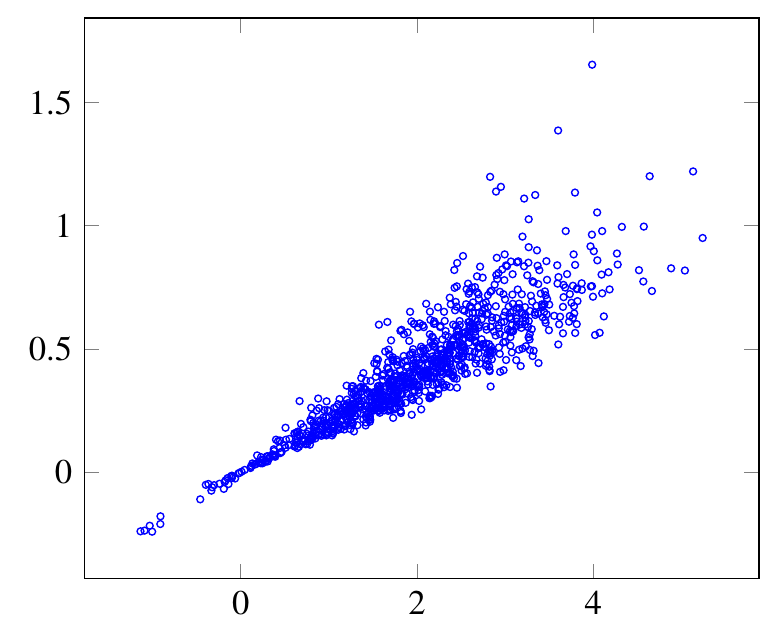}}
 \caption{$(X,X/\varepsilon)$ with $X\sim\CN(2,1)$, $\varepsilon\sim\CN(5,1)$ are independent.}
 \label{subfig.noise2}
 \end{subfigure}
  \medskip\linebreak

  \begin{subfigure}[t]{.24\textwidth}
 	\centering
 	\FIG
 	{\resizebox{\linewidth}{!}{
 			\begin{tikzpicture}[scale=\tikzscale]
 				\begin{axis}
 					\addplot[blue,only marks, mark=o, mark options={scale=.5}] table [x index=0,y index=1] {coordinates.dat};
 				\end{axis}
 			\end{tikzpicture}
 	}}
 	{\includegraphics[width=\linewidth]{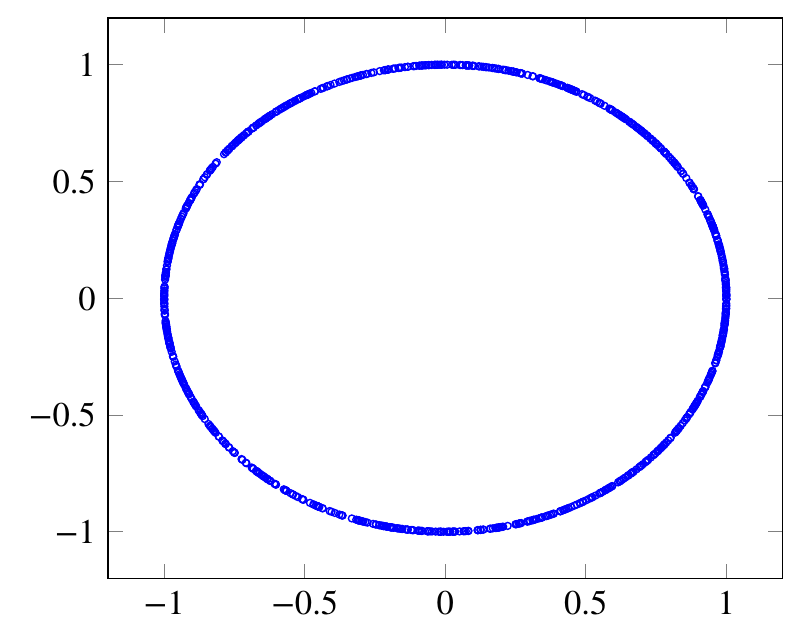}}
 	\caption{$U\{C(0,1)\}$.}
 	\label{subfig.BVU}
 \end{subfigure}
 \hfill
 \begin{subfigure}[t]{.24\textwidth}
 \centering
\FIG
{\resizebox{\linewidth}{!}{
\begin{tikzpicture}[scale=\tikzscale]
 \begin{axis}
 \addplot[blue,only marks, mark=o, mark options={scale=.5}] table [x index=17,y index=18] {coordinates.dat};
 \end{axis}
\end{tikzpicture}
}}
{\includegraphics[width=\linewidth]{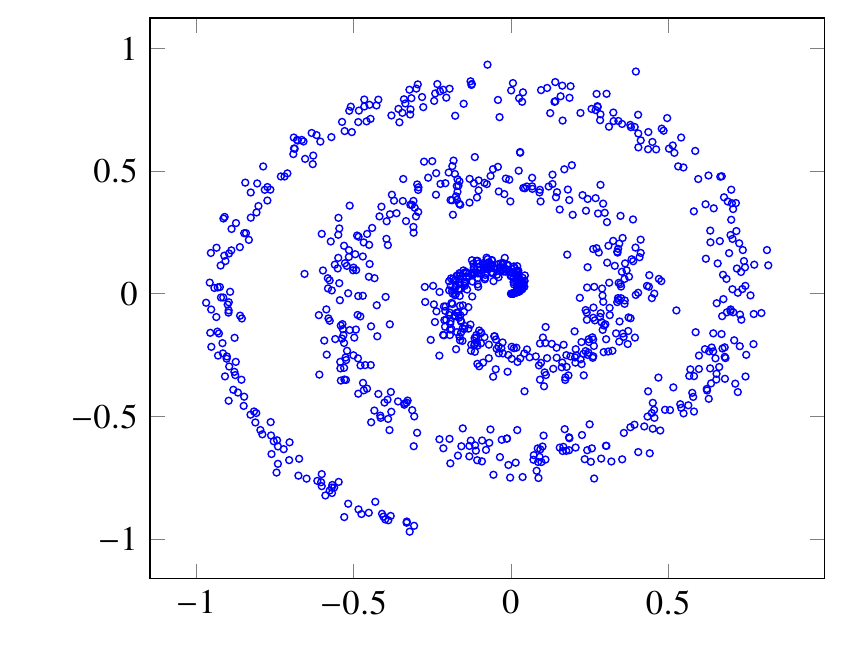}}
 \caption{Spiral data: case 1.}
 \label{subfig.Spiral1}
 \end{subfigure}
 \hfill
 \begin{subfigure}[t]{.24\textwidth}
 \centering
\FIG
{\resizebox{\linewidth}{!}{
\begin{tikzpicture}[scale=\tikzscale]
 \begin{axis}
 \addplot[blue,only marks, mark=o, mark options={scale=.5}] table [x index=26,y index=27] {coordinates.dat};
 \end{axis}
\end{tikzpicture}
}}
{\includegraphics[width=\linewidth]{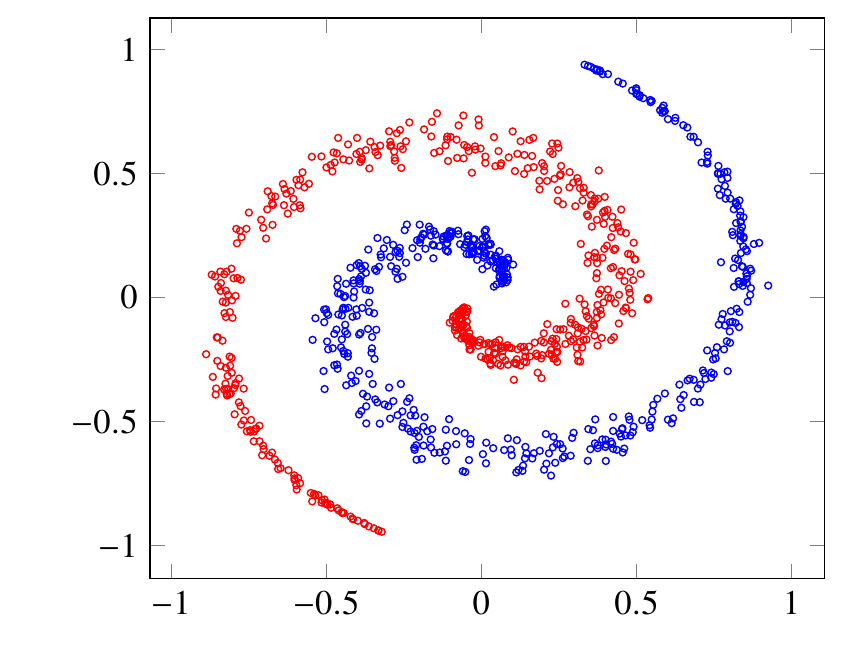}}
 \caption{Spiral data: case 2.}
 \label{subfig.Spiral2}
 \end{subfigure}
 \hfill
 \begin{subfigure}[t]{.24\textwidth}
 \centering
\FIG
{\resizebox{\linewidth}{!}{
\begin{tikzpicture}[scale=\tikzscale]
 \begin{axis}
 \addplot[blue,only marks, mark=o, mark options={scale=.5}] table [x index=28,y index=29] {coordinates.dat};
 \end{axis}
\end{tikzpicture}
}}
{\includegraphics[width=\linewidth]{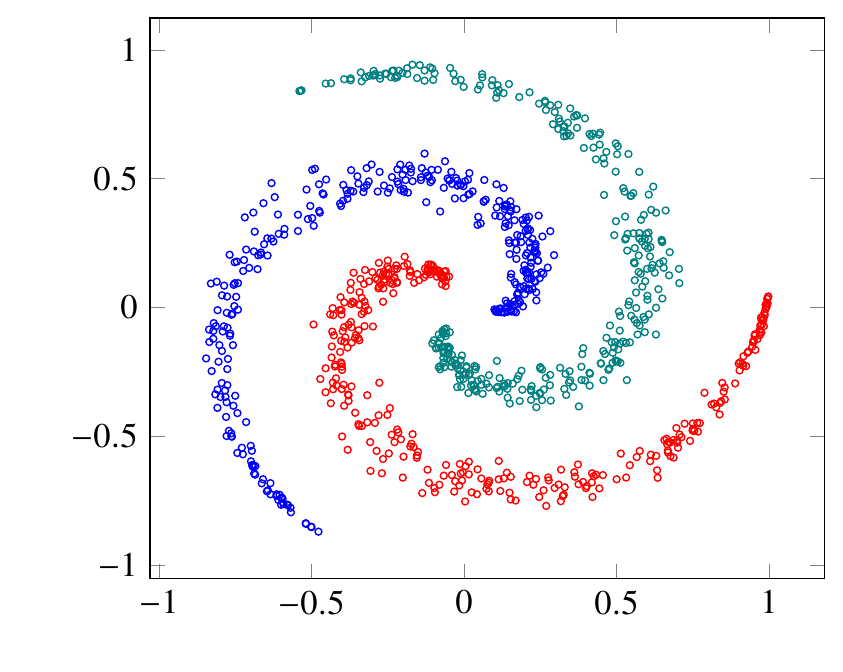}}
 \caption{Spiral data: case 3.}
 \label{subfig.Spiral3}
 \end{subfigure}
 \caption{Scatter plots: Random samples of size $n=1000$ are drawn from different bivariate distributions.}
 \label{Fig.BVdistributions}
 \end{figure}

In view of Table \ref{table.power.non-normal}, we see that the consistency of the proposed test of independence is also verified (as the sample size increases to infinity the power of the test tends to 1).

The case $U\{C(0,1)\}$ indicates an exact cyclical dependence (see Figure \ref{subfig.BVU}). In this case, the Kendall cdfs are $K_i(t)=\one{0\le t<1/2}(t+1/4)+\one{1/2\le t\le1}$ and so $\AUK_i=11/16-\ln(2)/4\approxeq0.514$, $i=1,\ldots,4$ --- see Example \ref{exm.C(0,1),normal}; hence, we compute $I=0.036$ and $I^*=0.074$. This kind of dependence is not recognized by both $\AUK$-indices. On the other hand, $\AUK$ is slightly different than $1/2$ and so we expect the power of the proposed test will tend to 1 as $n$ goes to infinity. Indeed, this happens even in this limited case, see the 25th row of the Table \ref{table.power.non-normal}; additionally, we have that the empirical power for $n=2500$ and $3000$ is $0.941$ and $0.986$ respectively.

In view of Figure \ref{subfig.BVinverseNormals}, one may incorrectly conclude that there is an association between the marginal distributions of the data distribution. Since $(1/X,1/Y)\sim \CN_2(\bm{0},I)$, $X$ and $Y$ are independent continuous rvs. Thus, the rate of rejection of the null hypothesis corresponds to the empirical size of the tests. For the same values of $n$ as in Table \ref{table.power.non-normal}, the empirical size of the $\AUK$ independence test calculated from 1000 random samples of size $n$ results in the following values 0.060, 0.053, 0.060, 0.052, 0.039, 0.052, 0.057, 0.054 and 0.046, that is, the test maintains the nominal level well for all sample sizes; notice that the \dHSIC independence test also maintains the nominal level well for all sample sizes.

Table \ref{table.power.non-normal} presents the empirical power of our proposed test, as well as that of the \dHSIC-based test for a variety of sample sizes and distributions. We observe that for distributions, such as $t_5$ with the reported covariance matrix, $\Morgenstern\{0.5\}$, $\Plackett\{1.25\}$ or $\Plackett\{2\}$ our proposed test outperforms the \dHSIC-based test for smaller samples and it has the same power with \dHSIC-based test for larger samples. The case where the \dHSIC-based test outperforms our test in terms of power is when the data come from $\ALIHAQ\{0.1,0.5\}$, $U\{C(0,1)\}$ and spiral data cases 1 and 3. Hence, overall the proposed test either performs competitively or outperforms the kernel-based tests in most of the cases met in practice.

\begin{table}[htp]
 \caption{Empirical power (\%) of $\AUK$ and \dHSIC independence tests, with a nominal significance level of $\alpha=0.05$, calculated from $1,000$ random samples of size $n$ simulated from various bivariate non-normal distributions.}
 \label{table.power.non-normal}
 \small
 \begin{tabular*}{\textwidth}
 {@{}@{\extracolsep{\fill}}l@{}l@{}S@{}S@{}S@{}S@{}S@{}S@{}S@{}S@{}S@{\hspace{-1.3ex}}}
 \toprule
 \raisebox{-1ex}{distribution} & \raisebox{-1ex}{test} \quad \ \ $n$ & {50} & {100} & {200} & {300} & {500} & {750} & {1000} & {1500} & {2000} \\
 \midrule
 \multirow{2}{*}{$\exp\{2,3,1.3\}$}
 & $\AUK$ & 97.3 & 100  & 100 & 100 & 100 & 100 & 100 & 100 & 100 \\
 [-.4ex]
 & \dHSIC & 90.5 & 99.9 & 100 & 100 & 100 & 100 & 100 & 100 & 100 \\
 [.6ex]
 \multirow{2}{*}{$t_5$ with $\Sig=\big({1\atop1}~{1\atop4}\big)$}
 & $\AUK$ & 93.4 & 99.9 & 100 & 100 & 100 & 100 & 100 & 100 & 100 \\
 [-.4ex]
 & \dHSIC & 76.9 & 97.3 & 100 & 100 & 100 & 100 & 100 & 100 & 100 \\
 [.6ex]
 \multirow{2}{*}{$\Morgenstern\{0.5\}$}
 & $\AUK$ & 27.4 & 42.7 & 66.1 & 82.1 & 95.8 & 99   & 99.9 & 100 & 100 \\
 [-.4ex]
 & \dHSIC & 16.6 & 28.1 & 50.4 & 69.7 & 87.9 & 97.9 & 99.6 & 100 & 100 \\
 [.6ex]
\multirow{2}{*}{$\Morgenstern\{5\}$}
 & $\AUK$ & 100 & 100 & 100 & 100 & 100 & 100 & 100 & 100 & 100 \\
 [-.4ex]
 & \dHSIC & 100 & 100 & 100 & 100 & 100 & 100 & 100 & 100 & 100 \\
 [.6ex]
 \multirow{2}{*}{$\Plackett\{1.25\}$}
 & $\AUK$ & 11.4 & 14.1 & 22.9 & 29.5 & 40.3 & 53.2 & 68.7 & 82   & 90.5 \\
 [-.4ex]
 & \dHSIC &  6.3 & 10   & 15.6 & 19.7 & 26.8 & 40.4 & 48.9 & 67.9 & 80.1 \\
 [.6ex]
 \multirow{2}{*}{$\Plackett\{2\}$}
 & $\AUK$ & 46.8 & 70.4 & 92.1 & 98.7 & 99.9 & 100 & 100 & 100 & 100 \\
 [-.4ex]
 & \dHSIC & 29.6 & 57.2 & 86.9 & 96.4 & 99.8 & 100 & 100 & 100 & 100 \\
 [.6ex]
 \multirow{2}{*}{$\ALIHAQ\{0.1,0.5\}$}
 & $\AUK$ & 17.2 & 26.3 & 39   & 51.7 & 69   & 84.6 & 92.4 & 99.1 & 100 \\
 [-.4ex]
 & \dHSIC & 16.1 & 30.7 & 66.8 & 89.2 & 99.4 & 100  & 100  & 100  & 100 \\
 [.6ex]
 \multirow{2}{*}{$\ALIHAQ\{0.9,0.5\}$}
 & $\AUK$ & 100 & 100 & 100 & 100 & 100 & 100 & 100 & 100 & 100 \\
 [-.4ex]
 & \dHSIC & 100 & 100 & 100 & 100 & 100 & 100 & 100 & 100 & 100 \\
 [.6ex]
 \multirow{2}{*}{$\Gumbel\{0.9\}$}
 & $\AUK$ & 22.1 & 35.3 & 52.6 & 68.1 & 85.9 & 95.4 & 98.7 & 100 & 100 \\
 [-.4ex]
 & \dHSIC & 24.9 & 43.7 & 76.1 & 90.4 & 99.2 & 100  & 100  & 100 & 100 \\
 [.6ex]
 $(X,\varepsilon/X^2)$
 & $\AUK$ & 100 & 100 & 100 & 100 & 100 & 100 & 100 & 100 & 100 \\
 [-.4ex]
 see Figure \ref{subfig.noise1}
 & \dHSIC & 100 & 100 & 100 & 100 & 100 & 100 & 100 & 100 & 100 \\
 [.6ex]
 $(X,X/\varepsilon)$
 & $\AUK$ & 100 & 100 & 100 & 100 & 100 & 100 & 100 & 100 & 100 \\
 [-.4ex]
 see Figure \ref{subfig.noise2}
 & \dHSIC & 100 & 100 & 100 & 100 & 100 & 100 & 100 & 100 & 100 \\
 [.6ex]
 \multirow{2}{*}{$U\{C(0,1)\}$}
 & $\AUK$ & 0   & 0.2 & 1.1 & 3.1 & 9.3 & 16. & 36.9 & 69 & 86.1 \\
 [-.4ex]
 & \dHSIC & 96.7 & 100 & 100 & 100 & 100 & 100 & 100 & 100 & 100 \\
 [.6ex]
 \multirow{2}{*}{Spiral data: case 1}
 & $\AUK$ & 14.5 & 22.9 & 35.6 & 44.9 & 60.5 & 76 & 86.7 & 95.4 & 98.8 \\
 [-.4ex]
 & \dHSIC & 28.6 & 58.5 & 93.7 & 99.2 & 100 & 100 & 100 & 100 & 100 \\
 [.6ex]
 \multirow{2}{*}{Spiral data: case 2}
 & $\AUK$ & 19.6 & 36.8 & 57.2 & 71.8 & 89.8 & 97.8 & 99.8 & 100 & 100 \\
 [-.4ex]
 & \dHSIC & 10.5 & 19   & 41.4 & 69.1 & 95.3 & 99.9 & 100  & 100 & 100 \\
 [.6ex]
 \multirow{2}{*}{Spiral data: case 3}
 & $\AUK$ & 8.2 &  7.5 & 11.2 & 18.5 & 23.7 & 33.9 & 44.5 & 58.9 & 70.4 \\
 [-.4ex]
 & \dHSIC & 6.8 & 14.5 & 35.7 & 62.1 & 94.2 & 100  & 100  & 100  & 100  \\
 \bottomrule
 \end{tabular*}
 \end{table}

Table \ref{table.AUK-dHSIC.speed} provides the time, in seconds, needed to compute our $\AUK$-based and the \dHSIC-based tests of independence. This table, clearly shows that the $\AUK$-based test is very fast to compute, requiring, for example, for sample sizes 2000 and dimension 2 only 0.048 sec.

\begin{table}[htp]
 \caption{Speed time (in seconds) of the test-procedures $\AUK$ and \dHSIC for testing the total independence of a two-variable data set of size $n$.}
 \label{table.AUK-dHSIC.speed}
 \footnotesize
 \begin{tabular*}{\textwidth}
 {@{}@{\extracolsep{\fill}}l@{}c@{}c@{}S@{}S@{}S@{}S@{}S@{}S@{}S@{}}
 \toprule
 $n$ & {200} & {300} & {500} & {750} & {1000} & {1250} & {1500} & {1750} & {2000} \\
 \midrule
 $\AUK$ test & 0.000 & 0.000 &   0.015 &   0.016 &    0.017 &    0.031 &   0.032 &    0.047 &    0.048 \\
 \dHSIC test & 0.609 & 1.500 &   4.633 &  10.017 &   19.814 &   33.800 &  56.782 &   72.949 &  113.030 \\
 \bottomrule
 \end{tabular*}
 \end{table}

\section{A real data example}
\label{sec.RealData}

Alanine and aspartate aminotransferase, ALT and AST respectively, are markers of liver inflammation. Further, alkaline phosphatase (AP), as well as direct bilirubin (DB) are biomarkers of hepatic injury, measuring inflammation from viral, metabolic or autoimmune causes. Elevation of ALT and AST is referred to as hepatocellular liver injury pattern. In contrast, alkaline phosphatase and bilirubin are increased under conditions that lead to damage of the bile ducts that drain the liver and is referred to as cholestatic injury pattern. These parameters can be elevated due to bile duct obstruction by a stone or tumor or in the case of autoimmune diseases, in which body's immune system can attack the cholangiocytes. There can also be a mixed hepatocellular and cholestatic pattern. In the evaluation of liver disease, the pattern of injury guides the clinician toward the potential etiology and assists in directing the etiologic evaluation. Table \ref{table.DataSet} of the Supplementary Material presents a data set that contains measurements of (DB, AST, ALT, AP) on 208 patients with liver disease.

We estimate the indices $I$ and $I^*$ of the biomarkers' vector as $\hI_{208}(\mathrm{DB},\mathrm{AST},\mathrm{ALT},\mathrm{AP})=0.2546$ and $\hI^*_{208}(\mathrm{DB},\mathrm{AST},\mathrm{ALT},\mathrm{AP})=\varphi_4(0.2546)\approxeq0.561$, where the approximate polynomial $\varphi_4$ has been find by applying Algorithm \ref{algorithm-phi}. According to Definition \ref{def.level},
we conclude that the random vector of the biomarkes DB, AST, ALT and AP has a strong mutual dependence at the level $56.1\%$.

We next investigate the dependence structure of the random vector of these biomarkes. To accomplish this goal, we compute the indices $\hI_n$ and $\hI^*_n$ for all 2- and 3-dimensional sub-vectors of the vector of these biomarkers; Tables \ref{subtable.2D} and \ref{subtable.3D} present these results. According to Definition \ref{def.level}, we extract the following conclusions. Both $(\mathrm{DB,AST,ALT})$ and $(\mathrm{AST,ALT,AP})$ have strong mutual dependence of level $65.1\%$ and $60.5\%$ respectively, $(\mathrm{DB,AST,AP})$ has a mild mutual dependence of level $29.4\%$ and $(\mathrm{DB,ALT,AP})$ has a weak mutual dependence of level $23.2\%$; for the bidimensional sub-vectors, we have that $(\mathrm{AST,ALT})$ has a very strong mutual dependence of level $79.2\%$, $(\mathrm{DB,AST})$ has a mild mutual dependence of level $35.4\%$, and $(\mathrm{DB,ALT})$, $(\mathrm{DB,AP})$, $(\mathrm{AST,AP})$ and $(\mathrm{ALT,AP})$ have weak mutual dependence of level $23\%$, $20.4\%$, $17.1\%$ and $14.2\%$ respectively. In addition to the indices $\hI$ and $\hI^*$, Table \ref{subtable.2D} presents the corresponding sample Kendall's $\tau$ correlation coefficients. We observe that $\hI^*$  and $\h\tau$ recognize an association between two biomarkers almost in a similar way but, in general, the index $\hI^*$ indicates a stronger dependence level. Notice that Kendall's $\tau$ correlation coefficient cannot be used beyond the bivariate case.

The strong dependence between AST and ALT can be explained from the physiological point of view, since they are both released when hepatocytes undergo cell death. On the other hand, the weak dependence among the pairs of biomarkers can be explained by the fact that there are released under different mechanisms. When evaluating three biomarkers simultaneously, we identify the strongest dependency among the triplet $(\mathrm{DB,AST,ALT})$. This observation can be explained from the physiologic point of view because processes that cause hepatocyte injury would also have a likely predilection to affect hepatocyte synthetic ability. This results in the inability of hepatocyte to conjugate bilirubin, as assessed by DB. The dependence among the different biomarkers we observe can, therefore, be explained in the context of liver function and injury.

In conclusion, we say that the high level of the mutual dependence of the random vector of the biomarkers is mainly due to the very strong dependence between biomarkers AST and ALT, and secondly, to a small extent, on the dependence between the biomarkers bilirubin and AST.

\begin{table}[htp]
 \caption{The estimators of the indices $\hI_n$ and $\hI^*_n$ for the 2- and 3-dimensional sub-vectors of the vector of the biomarkers DB, AST, ALT and AP.}
 \label{table.3D,2D-indices}
  \begin{subtable}{.47\textwidth}
 \caption{2-dimensional sub-vectors.}
 \label{subtable.2D}
 \footnotesize
 \begin{tabular*}{\linewidth}
 {@{}@{\extracolsep{\fill}}l@{}c@{}c@{}c@{}c@{}c@{}c@{}}
 \toprule
  & \tiny(DB, AST) & \tiny(DB, ALT) & \tiny(DB, AP) & \tiny(AST,ALT) & \tiny(AST,AP) & \tiny(ALT, AP) \\
 \cmidrule{2-7}
  $\hI_n$   & 0.176 & 0.112 & 0.099 & 0.468 & 0.083 & 0.069 \\
  $\hI^*_n$ & 0.354 & 0.230 & 0.204 & 0.792 & 0.171 & 0.142 \\
  $\h\tau$  & 0.215 & 0.092 & 0.061 & 0.619 & 0.097 & 0.077 \\
 \bottomrule
 \end{tabular*}
 \end{subtable}
 \hfill
 \begin{subtable}{.47\textwidth}
 \caption{3-dimensional sub-vectors.}
 \label{subtable.3D}
 \footnotesize
 \begin{tabular*}{\linewidth}
 {@{}@{\extracolsep{\fill}}l@{}c@{}c@{}c@{}c@{}}
 \toprule
  & \tiny(DB, AST, ALT) & \tiny(DB, AST, AP) & \tiny(DB, ALT, AP) & \tiny(AST, ALT, AP) \\
 \cmidrule{2-5}
  $\hI_n$   & 0.322 & 0.146 & 0.118 & 0.296 \\
  $\hI^*_n$ & 0.651 & 0.294 & 0.232 & 0.605 \\
 \bottomrule
 \\
 \end{tabular*}
 \end{subtable}
 \end{table}

In addition to the preceding numerical results, Figures \ref{fig.2Dbiomarkers} and \ref{fig.3Dbiomarkers} shows the scatterplots of the bivariate and trivariate sub-vectors of the vector of biomarkers $(\mathrm{DB,AST,ALT,AP})$ respectively, illustrating graphically the dependence structure of these sub-vectors. Notice that even in the simplest case of bivariate random vectors it seems difficult to evaluate the strength of the dependence structure via scatterplots.

\begin{figure}[htp]
\begin{subfigure}[a]{.32\textwidth}
\FIG
{\resizebox{\linewidth}{!}{
\begin{tikzpicture}
\begin{axis}
  \addplot[blue, only marks, mark=o, mark options={scale=.5}] table [x=Blrbn,y=AST] {biomarkers.dat};
\end{axis}
\end{tikzpicture}
}}
{\includegraphics[width=\textwidth]{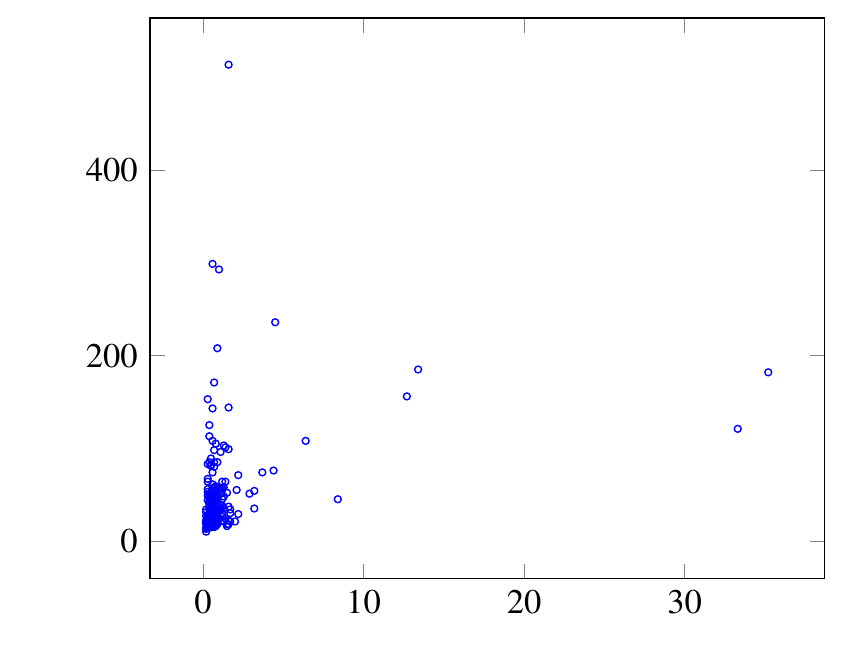}}
\caption{$(\mathrm{DB,AST})$.}
\label{sfig.B-AST}
\end{subfigure}
\hfill
\begin{subfigure}[a]{.32\textwidth}
\FIG
{\resizebox{\linewidth}{!}{
\begin{tikzpicture}
\begin{axis}
  \addplot[blue, only marks, mark=o, mark options={scale=.5}] table [x=Blrbn,y=ALT] {biomarkers.dat};
\end{axis}
\end{tikzpicture}
}}
{\includegraphics[width=\textwidth]{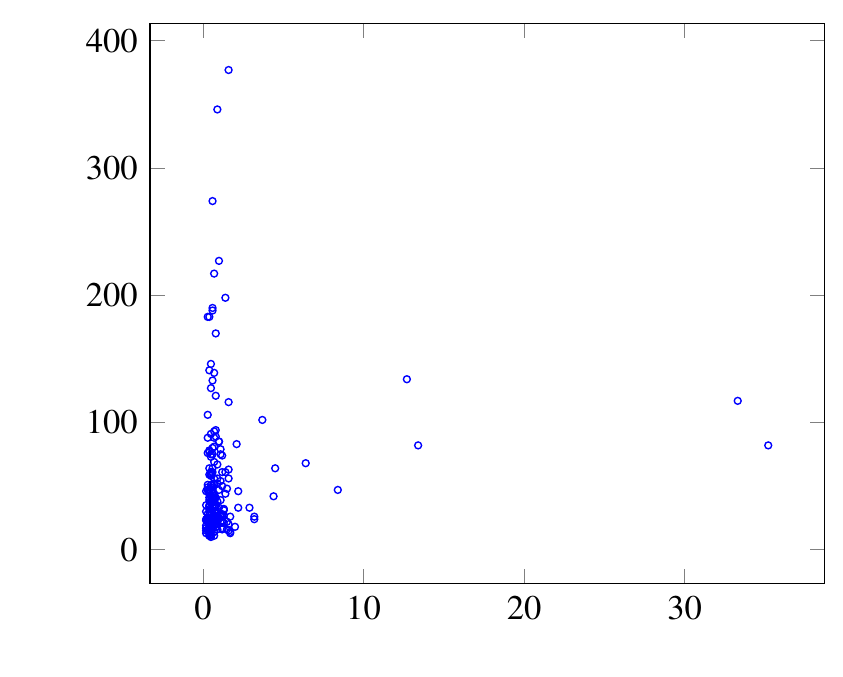}}
\caption{$(\mathrm{DB,ALT})$.}
\label{sfig.B-ALT}
\end{subfigure}
\hfill
\begin{subfigure}[a]{.32\textwidth}
\FIG
{\resizebox{\linewidth}{!}{
\begin{tikzpicture}
\begin{axis}
  \addplot[blue, only marks, mark=o, mark options={scale=.5}] table [x=Blrbn,y=AP] {biomarkers.dat};
\end{axis}
\end{tikzpicture}
}}
{\includegraphics[width=\textwidth]{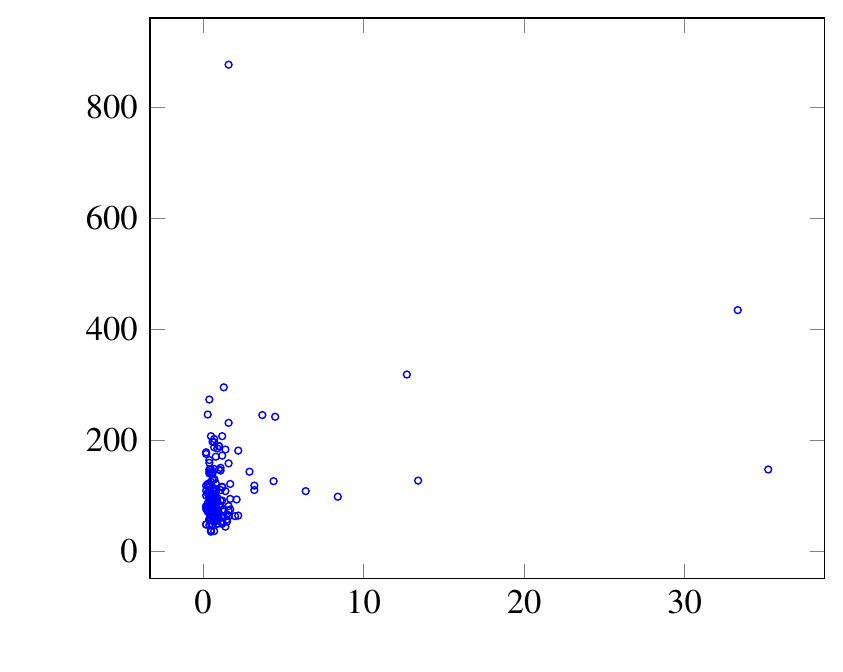}}
\caption{$(\mathrm{DB,AP})$.}
\label{sfig.B-AP}
\end{subfigure}
\medskip\linebreak
\begin{subfigure}[a]{.32\textwidth}
\FIG
{\resizebox{\linewidth}{!}{
\begin{tikzpicture}
\begin{axis}
  \addplot[blue, only marks, mark=o, mark options={scale=.5}] table [x=AST,y=ALT] {biomarkers.dat};
\end{axis}
\end{tikzpicture}
}}
{\includegraphics[width=\textwidth]{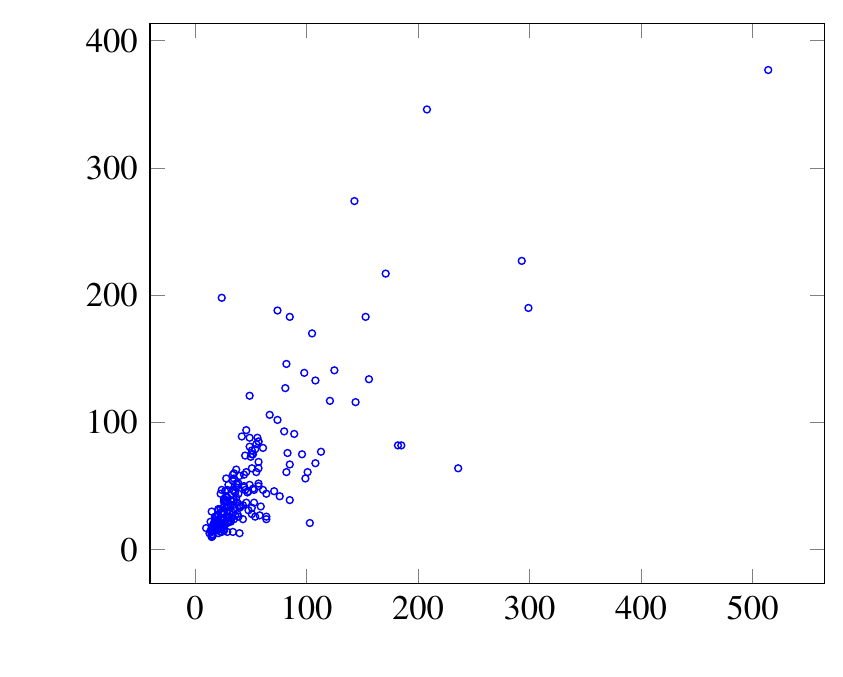}}
\caption{$(\mathrm{AST,ALT})$.}
\label{sfig.AST-ALT}
\end{subfigure}
\hfill
\begin{subfigure}[a]{.32\textwidth}
\FIG
{\resizebox{\linewidth}{!}{
\begin{tikzpicture}
\begin{axis}
  \addplot[blue, only marks, mark=o, mark options={scale=.5}] table [x=AST,y=AP] {biomarkers.dat};
\end{axis}
\end{tikzpicture}
}}
{\includegraphics[width=\textwidth]{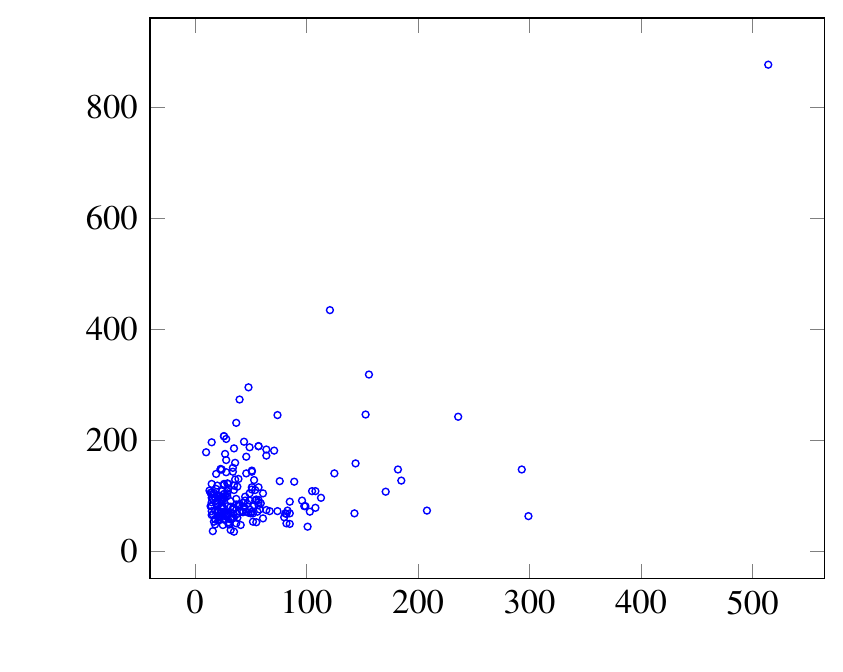}}
\caption{$(\mathrm{AST,AP})$.}
\label{sfig.AST-AP}
\end{subfigure}
\hfill
\begin{subfigure}[a]{.32\textwidth}
\FIG
{\resizebox{\linewidth}{!}{
\begin{tikzpicture}
\begin{axis}
  \addplot[blue, only marks, mark=o, mark options={scale=.5}] table [x=ALT,y=AP] {biomarkers.dat};
\end{axis}
\end{tikzpicture}
}}
{\includegraphics[width=\textwidth]{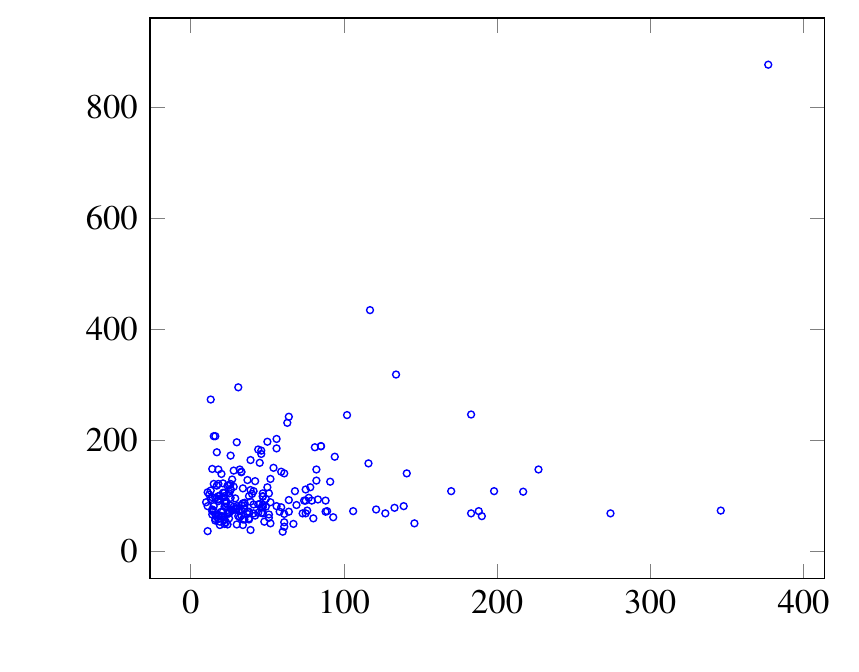}}
\caption{$(\mathrm{ALT,AP})$.}
\label{sfig.ALT-AP}
\end{subfigure}
\caption[The scatterplots of the bivariate sub-vectors of $(\mathrm{B,AST,ALT,AP})$.]{The scatterplots of the bivariate sub-vectors of $(\mathrm{DB,AST,ALT,AP})$.}
\label{fig.2Dbiomarkers}
\end{figure}
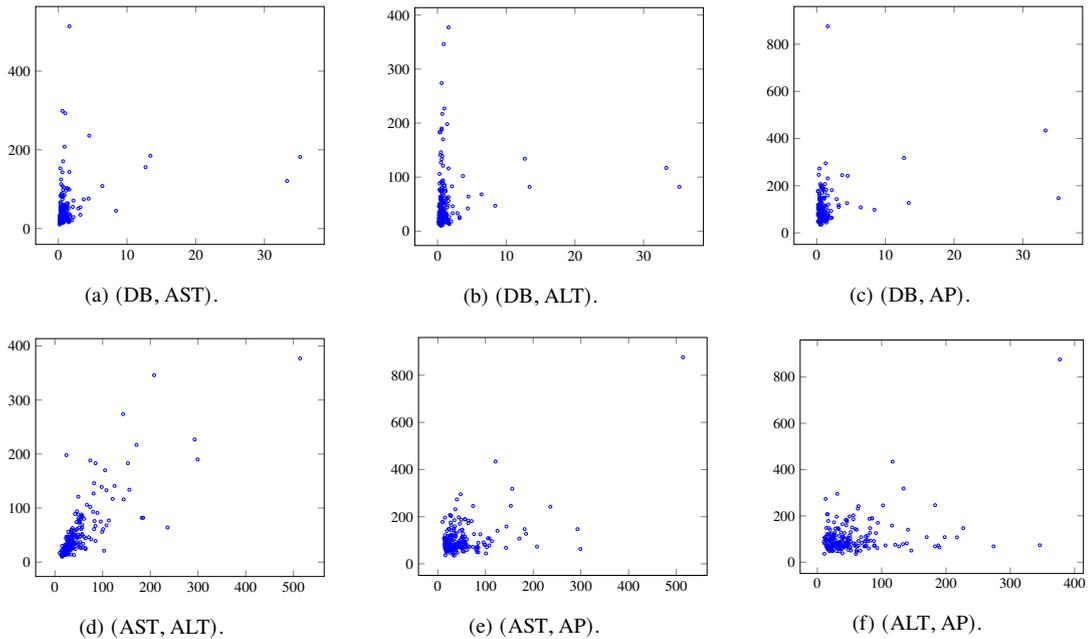

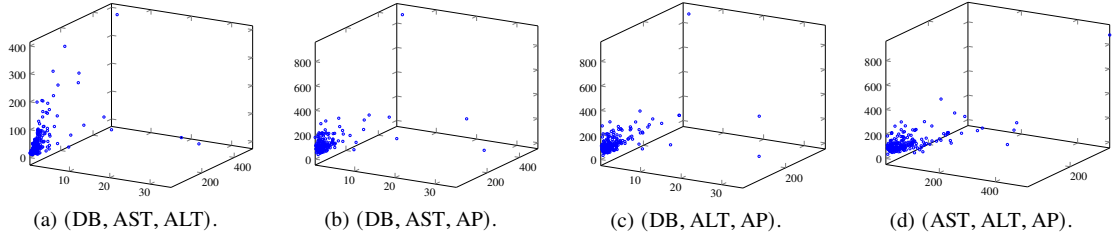
\begin{figure}[htp]
\begin{subfigure}[a]{.24\textwidth}
\FIG
{\resizebox{\linewidth}{!}{
\begin{tikzpicture}
\begin{axis}[view={30}{20}]
\addplot3+[draw=none, mark=o,color=blue,mark size=1] table {biomarkers123.dat};
\end{axis}
\end{tikzpicture}
}}
{\includegraphics[width=\textwidth]{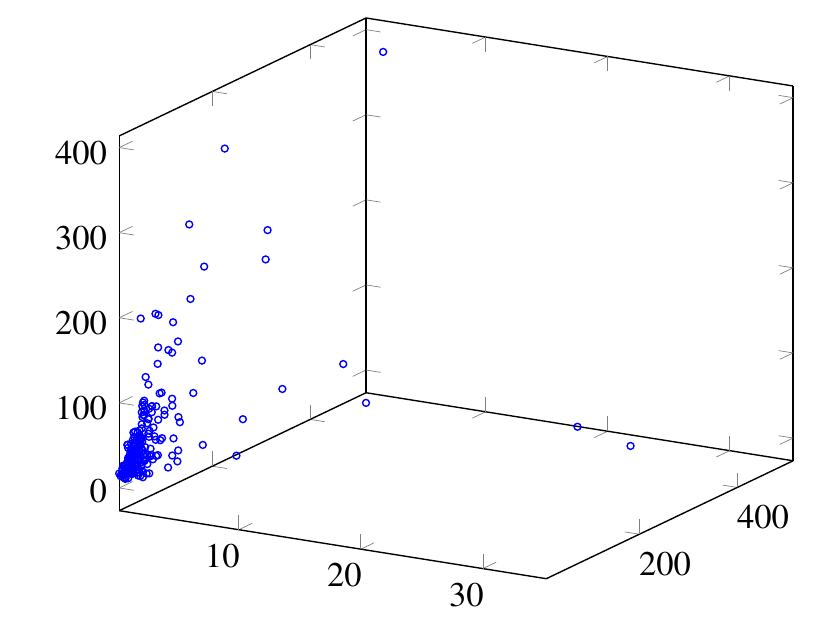}}
\caption{$(\mathrm{DB,AST,ALT})$.}
\label{sfig.B-AST-ALT}
\end{subfigure}
\hfill
\begin{subfigure}[a]{.24\textwidth}
\FIG
{\resizebox{\linewidth}{!}{
\begin{tikzpicture}
\begin{axis}[view={30}{20}]
\addplot3+[draw=none, mark=o,color=blue,mark size=1] table {biomarkers124.dat};
\end{axis}
\end{tikzpicture}
}}
{\includegraphics[width=\textwidth]{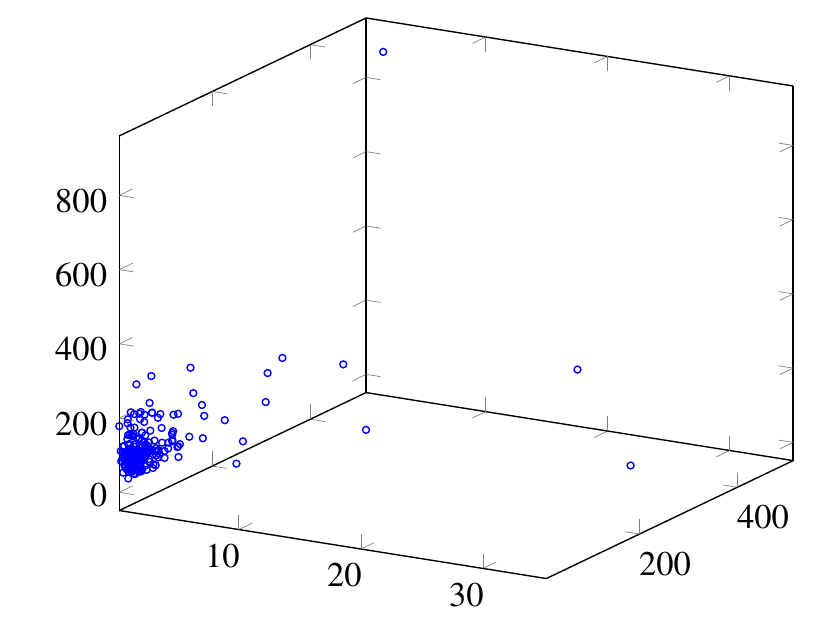}}
\caption{$(\mathrm{DB,AST,AP})$.}
\label{sfig.B-AST-AP}
\end{subfigure}
\hfill
\begin{subfigure}[a]{.24\textwidth}
\FIG
{\resizebox{\linewidth}{!}{
\begin{tikzpicture}
\begin{axis}[view={30}{20}]
\addplot3+[draw=none, mark=o,color=blue,mark size=1] table {biomarkers134.dat};
\end{axis}
\end{tikzpicture}
}}
{\includegraphics[width=\textwidth]{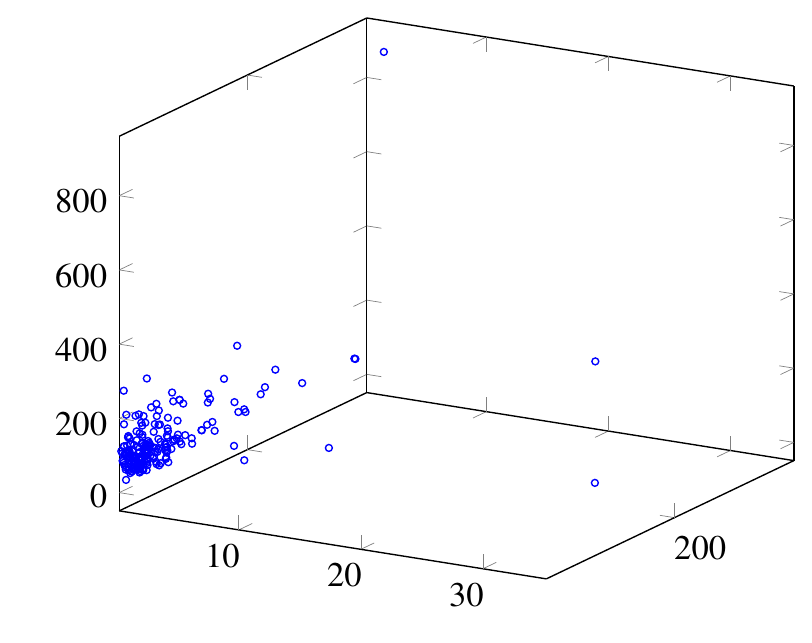}}
\caption{$(\mathrm{DB,ALT,AP})$.}
\label{sfig.B-ALT-AP}
\end{subfigure}
\hfill
\begin{subfigure}[a]{.24\textwidth}
\FIG
{\resizebox{\linewidth}{!}{
\begin{tikzpicture}
\begin{axis}[view={30}{20}]
\addplot3+[draw=none, mark=o,color=blue,mark size=1] table {biomarkers234.dat};
\end{axis}
\end{tikzpicture}
}}
{\includegraphics[width=\textwidth]{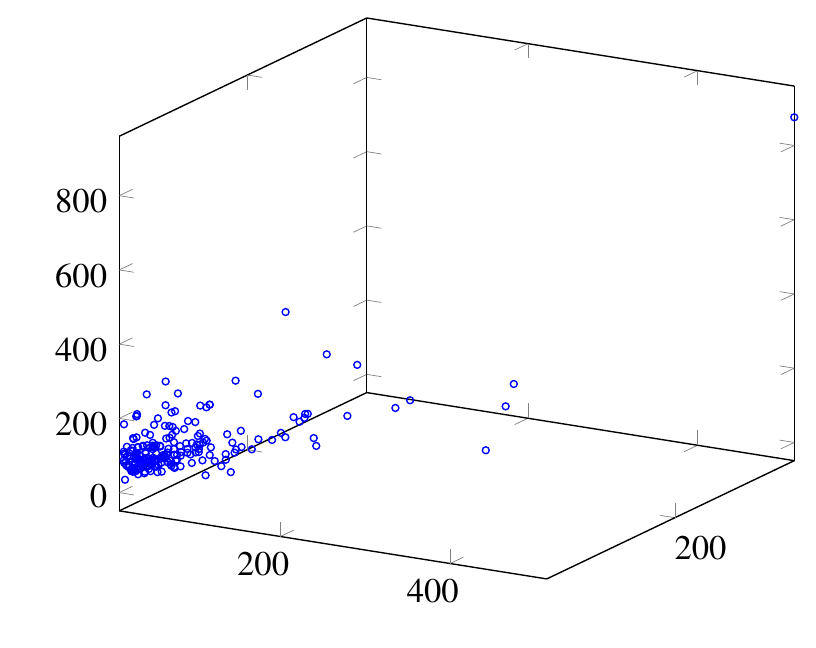}}
\caption{$(\mathrm{AST,ALT,AP})$.}
\label{sfig.AST-ALT-AP}
\end{subfigure}
\caption[The scatterplots of the trivariate sub-vectors of $(\mathrm{B,AST,ALT,AP})$.]{The scatterplots of the trivariate sub-vectors of $(\mathrm{DB,AST,ALT,AP})$.}
\label{fig.3Dbiomarkers}
\end{figure}

\section{Concluding Remarks and Discussion}
\label{sec.Discussion}

In this article we construct indices that measure joint or mutual dependence in a $d$-dimensional random vector. The indices depend on the Kendall process and satisfy the axioms set by \citeauthor{Renyi1959} that a function must satisfy in order to be called a measure of dependence. We evaluate the performance of the indices via simulation for a variety of distributions and sample sizes and present algorithms for their computation. Our proposed indices are copula-based and they are able to capture the degree of dependence exhibited in the data.

We also present a practical and fast, easy to compute, test for independence based on the estimated area under the Kendall curve, and propose an algorithm to compute quantities associated with its distribution.The performance of the proposed test statistic is evaluated via Monte Carlo simulation and is compared with the performance of the \dHSIC-based test in terms of power and level, in a wide range of distributions and sample sizes. Our proposed test either outperforms or is competitive with the \dHSIC-based test in most of the cases studied.

\begin{appendices}
\crefalias{section}{appsec}
{
\section{Lists of axioms for dependence measures}
\label{app:axioms}

According to \citet{SW1981}, we can present \citeauthor{Renyi1959}'s conditions regarding a measure of dependence $R(X,Y)$ for two continuously distributed
variables $X$ and $Y$ in the following form.

\begin{enumerate}[topsep=0ex, itemsep=0ex, labelindent=0pt, leftmargin=6.28ex, label=$A^{\mathrm{R}}_{\arabic*}\colon$, ref=\textrm{\textcolor{black}{$A^{\mathrm{R}}_{\arabic*}$}}]
 \item
 \label{R1}
 $R(X,Y)$ is defined for any $X$ and $Y$.

 \item
 \label{R2}
 $R(X,Y)=R(Y,X)$.

 \item
 \label{R3}
 $0\le R(X,Y)\le1$

 \item
 \label{R4}
 $R(X,Y)=0$ if and only if $X$ and $Y$ are independent.

 \item
 \label{R5}
 $R(X,Y)=1$ if and only if each of $X$, $Y$ is a.s.\ a strictly monotone function of the other.

 \item
 \label{R6}
 If $f$ and $g$ are strictly monotone a.s.\ on Range $X$ and Range $Y$, respectively, then $R\{f (X), g(Y)\} = R(X, Y)$.

 \item
 \label{R7}
 If the joint distribution of $X$ and $Y$ is bivariate normal, with correlation coefficient $\rho$, then $R(X, Y)$ is a strictly increasing function of $|\rho|$.

 \item
 \label{R8}
 If $(X, Y)$ and $(X_n, Y_n)$, $n=1,2,\ldots$, are pairs of random variables with joint distributions $H$ and $H_n$, respectively, and if the sequence $\{H_n\}$ converges weakly to $H$, then $\lim_{n\to\infty} R(X_n, Y_n) = R(X, Y)$.
 \end{enumerate}

\citet{MSz2019} propose four natural axioms for dependence measures as follows. Let $\Ss$ be a nonempty set of pairs of nondegenerate random variables $\bX$, $\bY$ taking values in Euclidean spaces or in real, separable Hilbert spaces $\Hh$. (Nondegenerate means that the random variable is not constant with probability 1.) Then $\varDelta \colon \Ss \to [0,1]$ is called a dependence measure on $\Ss$ if the following four axioms hold. In the axioms below we need similarity transformations of $\Hh$. Similarity of $\Hh$ is defined as a bijection (1--1 correspondence) from $\Hh$ onto itself that multiplies all distances by the same positive real number (scale). Similarities are known to be compositions of a translation, an orthogonal linear mapping, and a uniform scaling. They assume that if $(\bX,\bY)\in\Ss$ then $(L\bX,M\bY)\in\Ss$ for all similarity transformations $L$, $M$ of $\Hh$.
\begin{enumerate}[topsep=0ex, itemsep=0ex, labelindent=0pt, leftmargin=7.5ex, label=$A^{\mathrm{MS}}_{\arabic*}\colon$, ref=\textrm{\textcolor{black}{$A^{\mathrm{MS}}_{\arabic*}$}}]
 \item
 \label{MS1}
 $\varDelta(\bX,\bY)=0$ if and only if $\bX$ and $\bY$ are independent.

 \item
 \label{MS2}
 $\varDelta(\bX,\bY)$ is invariant with respect to all similarity transformations of $\Hh$; that is, $\varDelta(L\bX,M\bY)=\varDelta(\bX,\bY)$ where $L$, $M$ are similarity transformations of $\Hh$.

 \item
 \label{MS3}
 $\varDelta(\bX,\bY)=1$ if and only if $\bY=L\bX$ with probability 1, where $L$ is a similarity transformation of $\Hh$.

 \item
 \label{MS4}
 $\varDelta(\bX,\bY)$ is continuous; that is, if $(\bX_n,\bY_n)\in\Ss$, $n=1,2,\ldots$, such that for some positive constant $K$ we have $\E(|\bX_n|^2+|\bY_n|^2)\le K$, $n=1,2,\ldots$, and $(\bX_n,\bY_n)$ converges weakly (converges in distribution) to $(\bX,\bY)$ then $\varDelta(\bX_n,\bY_n)\to\varDelta(\bX,\bY)$. (The condition on the boundedness of second moments can be replaced by any other condition that guarantees the convergence of expectations: $\E(\bX_n)\to\E(\bX)$ and $\E(\bY_n)\to\E(\bY)$; such a condition is the uniform integrability of $\bX_n$, $\bY_n$ which follows from the boundedness of second moments.)
\end{enumerate}

\section{Kendall process}
\label{app:Gn}

First, we state a modified empirical Kendall cdf as it is introduced by \citet{GR1993}, \citep[see also][]{BGGR1996}. Let $\wT_{ni}=\{\# j\ne i\colon \bX_j\le\bX_i \textrm{ componentwise}\}/(n-1)$. Denote by $\wbbK_n$ the empirical distribution function of the $\wT_{ni}$'s that is an empirical Kendall cdf, i.e.\ an estimator of the Kendall cdf $K$.

We state now the following hypotheses.

\begin{enumerate}[topsep=0ex, itemsep=0ex, labelindent=0pt, leftmargin=6ex, label=$\mathrm{H}_{\arabic*}\colon$, ref=\textrm{\textcolor{black}{$\mathrm{H}_{\arabic*}$}}]
 \item
 \label{H1}
 The Kendall cdf $K(t)$ admits a continuous probability density function $k(t)$ on $(0, 1]$ that verifies $k(t)=o\{t^{-1/2}\ln^{-1/2-\varepsilon}(1/t)\}$ for some $\varepsilon>0$ as $t\to0$.

 \item
 \label{H2}
 There exists a version of the conditional distribution of the vector $(F_1(X_1),\ldots,F_d(X_d))$ given $T=t$ and a countable family $\CP$ of partitions $\CCC$ of $[0,1]^d$ into a finite number of Borel sets satisfying
\[
\inf_{\CCC\in\CP}\max_{C\in\CCC}\mathrm{diam}(C)=0,
\]
such that for all $C\in\CCC$, the mapping $t\mapsto\mu_t(C)=k(t)\Pr\{(F_1(X_1),\ldots,F_d(X_d))\in C|T=t\}$ is continuous on $(0,1]$ with $\mu_1(C)=k(1)\one{(1,\ldots,1)\in C}$.
\end{enumerate}

Given $0\le s,t\le1$, define
\[
Q(s,t)=\Pr\{F(\bX_1)\le s, \bX_1\le\bX_2|F(\bX_2)=t\}-tK(s)
\]
and
\[
R(s,t)=\Pr\{\bX_1\le\bX_2\wedge\bX_3|F(\bX_2)=s, H(\bX_3)=t\}-st,
\]
where $\bX_1,\ldots,\bX_3$ are independent and identically distributed random vectors from $F$, and $\bm{u}\wedge\bm{v}$ denotes the componentwise minimum between $\bm{u}$ and $\bm{v}$.

Consider now the Kendall process $\bbG_n(t)\doteq\surd{n}\{\wbbK_n(t)-K(t)\}$, $0\le t\le 1$. The following result is given by \citet[Theo.~1]{BGGR1996}, \citep[cf.][]{GR1993,Wellner2005,SDP2013}.

\begin{theorem}
\label{thm.Gn}
Under \ref{H1} and \ref{H2}, the Kendal process $\bbG_n(t)=\surd{n}\{\wbbK_n(t)-K(t)\}$, $0\le t\le 1$ converges in distribution to a Gaussian process $\bbG_K$ with zero mean and covariance function
\[
\varGamma(s,t)=K(s\wedge t)-K(s)K(t)+k(s)k(t)R(s,t)-k(t)Q(s,t)-k(s)Q(t,s).
\]
\end{theorem}

We restate here the hypothesis test $H_0\colon F=\prod_{j=1}^{d}F_j$ versus $H_1\colon F\ne\prod_{j=1}^{d}F_j$. Under $H_0$, $K_{H_0}=K_{\vP}$ and $k_{H_0}=k_{\vP}$; additionally, \citep[see][Example~1, p.~202]{BGGR1996},
\begin{equation}\label{eq.Kendall-process}
 \surd{n}\{\wbbK_n(t)-K_\vP(t)\}\stackrel{H_0}{\rightsquigarrow}\bbG_\vP,
\end{equation}
where $\bbG_\vP$ is a Gaussian process with zero mean and for $0\le s,t\le1$ covariance function as it given by \eqref{eq.cov-function}. In \eqref{eq.cov-function(b)} we have corrected a minor misprint \citep[see][p.~208, for the correct general formula]{BGGR1996}. Relations \eqref{eq.cov-function} lead to a cumbersome but explicit expression for the covariance function $\varGamma_\vP$, which reduces to $\varGamma_\vP(s,t)=s(t-1-\ln t)$, $0\le s \le t\le1$, when $d=2$.

A $\wbbK_n$-based estimator of $\AUK$ is $\wAUK_{n}=\E\left\{1-K_{\vP}(\wT_{n})\right\}=1-n^{-1}\sum_{i=1}^{n}K_{\vP}(\wT_{ni})$. Hence, using functional \emph{delta}-method, we get $\surd{n}(\wAUK_n-1/2)\stackrel{H_0}{\rightsquigarrow}\CN(0,\sigma_{\vP}^2)$, where $\sigma_{\vP}^2$ is defined by \eqref{eq.covergence(b)}. When $d=2$, after a straightforward calculation, we compute $\sigma_{\vP}^2=19/423$. For completing the proof of \eqref{eq.covergence} it is sufficient to prove that $\surd{n}(\hAUK_n-\wAUK_n)=o_p(1)$ as $n\to\infty$. For the proof of the asymptotic normality of both $\wAUK$ and $\hAUK$ see Appendix \ref{app:proofs}.

\section{Useful results}
\label{app:results}

The next result is the multivariate version of P\'olya's Theorem \citep[Prop.~1.16, p.~51]{Shao2003}. First, denote by $\|f\|_{\infty}\doteq\sup_{\bx\in\RR^d}|f(\bx)|$ the supremum norm of $f\colon\RR^d\to\RR$.

\begin{theorem}
\label{thm.Polya}
Let $F$ and $F_n$ are cdfs on $\RR^d$. If $F_n\Rightarrow F$ and $F$ is a continuous cdf, then $\|F_n-F\|_{\infty}\to0$.
\end{theorem}

\begin{theorem}
\label{thm.Tn->T}
Let $\bX\sim F$ and $\bX_n\sim F_n$ on $\RR^d$ such that $\bX_n\rightsquigarrow\bX$. Consider the random variables $T\doteq F(\bX)$ and $T_n\doteq F_n(\bX_n)$. If $F$ is continuous, $T_n\rightsquigarrow T$.
\end{theorem}

\begin{proof}
By Skorokhod's representation Theorem, there are $\wbX$ and $\wbX_n$ defined in the same probability space $(\varOmega,\CA,\Pr)$ such that $\cbX\sim F$, $\cbX_n\sim F_n$ and
\begin{equation}\label{eq1}
\cbX_n(\omega)\to\cbX(\omega)\quad\textrm{for each} \ \omega\in\varOmega,
\end{equation}
\citep[see, for example,][Theo.~25.6]{Billingsley2013}. Consider now the random variables $\cT\doteq F(\cbX)$ and $\cT_n\doteq F_n(\cbX_n)$; then, applying the triangle inequality, for each $\omega\in\varOmega$,
\[
 |\cT_n(\omega)-\cT(\omega)|\le|F_n\{\cbX_n(\omega)\}-F\{\cbX_n(\omega)\}|+|F\{\cbX_n(\omega)\}-F\{\cbX(\omega)\}|.
\]
Because of continuity of $F$, \eqref{eq1} implies that
\[
|F\{\cbX_n(\omega)\}-F\{\cbX(\omega)\}|\to0.
\]
Using Theorem \ref{thm.Polya}, for each $\omega\in\varOmega$,
\[
|F_n\{\cbX_n(\omega)\}-F\{\cbX_n(\omega)\}|\le\|F_n-F\|_{\infty}\to0.
\]
Consequently, $\cT_n(\omega)\to\cT(\omega)$ for each $\omega\in\varOmega$; this implies that $\cT_n\as\cT$ which gives $\cT_n\rightsquigarrow\cT$, and so $T_n\rightsquigarrow T$.
\end{proof}

From Theorem \ref{thm.Tn->T} and \citet[Theo.~3.1, p.~274]{Faugeras2013} the following theorem follows.

\begin{theorem}
\label{thm.hT->T}
Let $\bX\sim F$ and $T$ as in Theorem {\rm\ref{thm.Polya}}, and $\{\bX_n\}$ be an independent and identically distributed sequence from $F$. Consider the sequences of the ecdfs $\bbF_{n}$ as well as its empirical Kendall distribution functions $\bbK_{n}$. If $\hT_n\sim \bbK_{n}$, then $\hT_n\rightsquigarrow T$.
\end{theorem}

\section{Proofs of the main manuscript}
\label{app:proofs}

%\begin{proof}[Proof of \eqref{eq.vP-CDF,PDF}]
%One can easily to see that $-\ln U_j\sim\exp(1)$ for all $j=1,\ldots,n$. Consequently, due to independence of $U_j$s, it follows that %$E=-\ln\vP=\sum_{j=1}^{d}-\ln U_j\sim \Erlang(d,1)$; equivalently, $2E\sim \chi^2_{2d}$. Since $E\sim\Erlang(d,1)$, we have that
%\[
%F_E(t)=1-\sum_{k=0}^{d-1}\frac{t^k}{k!}e^{-t}
%\ \ \text{and} \ \
%f_E(t)=\frac{t^{d-1}}{(d-1)!}e^{-t},
%\quad t>0.
%\]
%By definition, $\vP=e^{-E}$. Because $E$ is supported on $(0,\infty)$, the support of $\vP$ is the interval $(0,1)$. Also, using the above relations, for %$0<t<1$,
%\[
%K_{\vP}(t)=\Pr(\vP\le t)=\Pr(E\ge-\ln t)=1-F_E(-\ln t)=t\sum_{j=0}^{d-1}\frac{(-1)^k}{k!}\ln^k t,
%\]
%\[
%k_{\vP}(t)=\frac{\ud{}}{\ud{t}}K_{\vP}(t)=\frac{\ud{}}{\ud{t}}\{1-F_E(-\ln t)\}=\frac{1}{t}f_E(-\ln t)=\frac{(-1)^{d-1}}{(d-1)!}\ln^{d-1}t, \ t>0.
%\]
%Alternatively,
%\[
%K_{\vP}(t)=1-\Pr(2E<-2\ln t)=1-F_{\chi^2_{2d}}(-2\ln t)
%\ \ \text{and} \ \
%k_{\vP}(t)=\frac{2}{t}f_{\chi^2_{2d}}(-2\ln t), \ t>0,
%\]
%completing the proof.
%\end{proof}

\begin{proof}[Proof of Equation \eqref{eq.c_d}]
Suppose $Y$ is a continuous rv and consider the $d$-dimensional rvs $\bY^{(1)}=(Y,\ldots,Y)$ and $\bY^{(2)}=(Y,\ldots,Y,-Y,\ldots,-Y)^T$ (the first $p\ge1$ components of $\bY^{(2)}$ are $Y$ and the last $d-p\ge1$ are $-Y$). For each $\by=(y_1,\dots,y_d)^T\in\RR^d$, one can easily see that $F_{\bY^{(1)}}(\by)=\Pr(Y\le\min\{y_1,\ldots,y_d\})$; consequently, $T_{\bY^{(1)}}=F_{\bY^{(1)}}(\bY^{(1)})=F_Y(Y)\sim U(0,1)$, and so $K_{\bY^{(1)}}(t)=t$, $t\in[0,1]$. The $\AUK$ is
\[
\AUK_{\bY^{(1)}}=\int_0^1 t\frac{(-\ln t)^{d-1}}{(d-1)!}\ud{t}=2^{-d}\int_0^{\infty} \frac{2^d}{(d-1)!}x^{d-1}e^{-2x}\ud{x}=2^{-d}.
\]
In the case of $\bY^{(2)}$, we can easily to verify $F_{\bY^{(2)}}(\by)=\Pr(\max\{-y_{p+1},\ldots,-y_{d}\}\le Y\le\min\{y_1,\ldots,y_p\})=\max\{0,F_Y(\min\{y_1,\ldots,y_p\})-F_Y(\max\{-y_{p+1},\ldots,-y_{d}\})\}$; therefore, $T_{\bY^{(2)}}=\max\{0,F_Y(Y)-F_Y(Y)\}=0$ with probability one, and so $K_{\bY^{(2)}}(t)=1$, $t\in[0,1]$. Thus,
\[
\AUK_{\bY^{(2)}}=\Pr(T_{\bY^{(2)}}\le \vP)=\Pr(0\le \vP)=1.
\]

For a continuous rv $X$, consider the $d$-dimensional $\wbX=(X,\ldots,X)^T$. The set $\CR(\wbX)$ contains 2 rvs of the kind of $\bY^{(1)}$, specifically $(X,\ldots,X)^T$ and $(-X,\ldots,-X)^T$, and $2^d-2$ of the kind of $\bY^{(2)}$ (note that the coordinates of minus do not play any role). Hence, the corresponding vector $\w\bD$ has 2 elements equals $2^{-d}$ and $2^d-2$ elements equals 1; so, $\|\w\bD-\bDelta\|=\left(2^{d-2}-2^{1-d}+2^{1-2d}\right)^{1/2}$, completing the proof.
\end{proof}

\begin{proof}[Proof of Proposition \ref{prop.A1-A8}]
\eqref{prop.A1-A8(a)}
By definition of $I$, \ref{A1} holds in $\CX^d_0$; while \ref{A2} is obvious. Regarding \ref{A6}, for any such $f_i$s define $s^{f_i}=1$ if $f_i$ is increasing function and $-1$ when $f_i$ is decreasing function, and set $\bs^{f}=(s^{f_1},\ldots,s^{f_d})^T\in\SSS^d$, $\wbX=\diag(\bs^{f})\bX$ and $\w{f}_i(x_i)=f_i(s^{f_i}x_i)$, $i=1,\ldots,d$ (strictly increasing functions). Then, obviously, the random vectors $f(\bX)=(f_1(X_1),\ldots,f_d(X_d))^T$ and $\w{f}(\bX)=(\w{f}_1(X_1),\ldots,\w{f}_d(X_d))^T$ belong to $\CX^d_0$ and $\CR\{f(\bX)\}=\CR\{\w{f}(\bX)\}$. Observe now that for each $\bs_j=(s_{1j},\ldots,s_{dj})^T\in\SSS^d$, $\diag(\bs_j)f(\bX)=\diag(\bs_j)\w{f}(\wbX)=\w{f}^{(j)}\{\diag(\bs_j)\wbX\}$, where $\w{f}^{(j)}=\big(\w{f}^{(j)}_1,\ldots,\w{f}^{(j)}_d\big)^T$ with $\w{f}^{(j)}_i(x_i)=s_{ij}\w{f}_i(s_{ij}x_i)$ to be an strictly increasing function for each $i=1,\ldots,d$. Hence, it becomes clear that it is enough to prove the desired result for strictly increasing functions; to prove this, one can use the similar arguments as in \citet[proof of Proposition~4.4, p.~439]{VAM2019}. With regard to \ref{A7}, only for $d=2$, it has been studied by \citet{VAM2019}. Finally, suppose $\bX_n\rightsquigarrow\bX\in\CX^d_0$. Since $K_{\vP}(t)$, $0\le t\le1$, is a continuous and bounded function, from Theorem \ref{thm.Tn->T} (see Appendix \ref{app:results}) and Portmanteau theorem we have that, for each $j=1,\ldots,2^d$, $\AUK_{nj}\to\AUK_j$, and so, $I(\bX_n)\to I(\bX)$; that, is \ref{A8} holds.

\eqref{prop.A1-A8(b)}
If $X_1,\ldots,X_d$ are independent, one can easily see that $\AUK_j=0$ for all $j=1,\ldots,2^d$, and so $I=0$. Conversely, by definition, $I=0$ implies that $\AUK_j=0$ for all $j=1,\ldots,2^d$. Since $\bX\in\CX^d_1$, there is at least an subscript $j'\in\{1,\ldots,2^d\}$ for which $K_{j'}(t)=K_{\vP}(t)$ for $0\le t\le1$. Hence, from Sklar's Theorem \citep[see][]{Sklar1959,DF-SS2012,SDP2013} the corresponding (unique) copula to $F_{\bX_{j'}}$ is the copula $\bC_{0}(\bu)=u_1\cdots u_d$, $\bu\in[0,1]^d$, that corresponds to the independence case. Consequently, the components of $\bX_{j'}$ are totally independent which is equivalent to the components of $\bX$ are totally independent, i.e.\ $F_{\bX}(\bx)=\prod_{i=1}^{d}F_{X_i}(x_i)$ for all $\bx\in\RR^d$.

\eqref{prop.A1-A8(c)}
It is clear that $I\ge0$ for all $\bX\in\CX^d_0$. Suppose that $\bX\in\CX^d_2$. Based on the Fr\'{e}chet-Hoeffding upper-bound \citep[see, for example,][]{Ruschendorf2017}, we have $K(t)=\Pr\{F(\bX)\le t\}\ge\Pr\left[\min_{j=1,\ldots,d}\{F_j(X_j)\}\le t\right]\ge\Pr\{F_1(X_1)\le t\}=t$ for all $0\le t\le1$; also, obviously, $K(t)\le1$ for all $0\le t\le1$. Since $\bX\in\CX^d_2$, there are at least two $\bX_{j'},\bX_{j''}\in\CR(\bX)$ for which \ref{C2} holds; so, see in the proof of \eqref{eq.c_d} above, $\sum_{j=j',j''}(\AUK_j-1/2)^2\le2(1-2^{-d})^2$. On the other hand, it is obvious that $\sum_{j\ne j',j''}(\AUK_j-1/2)^2\le(2^d-2)/4$. Hence, $\sum_{j}(\AUK_j-1/2)^2\le2^{d-2}-2^{1-d}+2^{1-2d}$, which implies that $I\le1$. Suppose now $\bX\in\CX^d_2$ with $I=1$. Then, there is a subscript $j'$ such that $K_{j'}(t)=t$, $0\le t\le1$. As in \eqref{prop.A1-A8(b)}, from Sklar's Theorem the corresponding (unique) copula to $F_{\bX_{j'}}$ is the copula $\bC_{1}(\bu)=\min\{u_1,\ldots,u_d\}$, $\bu\in[0,1]^d$, that corresponds to the total dependence case.
\end{proof}

\begin{proof}[Proof of Proposition \ref{prop.B1-B4}]
\eqref{prop.B1-B4(a)}
See Proposition \ref{prop.A1-A8}\ref{prop.A1-A8(b)}.
\eqref{prop.B1-B4(b)}
See Proposition \ref{prop.A1-A8}\ref{prop.A1-A8(a)}.
\eqref{prop.B1-B4(c)}
Similar to Proposition \ref{prop.A1-A8}\ref{prop.A1-A8(c)}.
\end{proof}

\begin{proof}[Proof of Proposition \eqref{eq.as->AUK}]
Suffice it to prove the result for the original (non-rotated) data. Set $D_n\doteq\|\bbF_n-F\|_{\infty}$. The multivariate version of Dvoretzky-Kiefer-Wolfowitz inequality says that \citep[see][Eq.~(1.2), p.~649]{Kiefer1961}, there are positive constants $c'_d$ and $c_d$ (do not depend on $F$) such that, $\Pr(D_n>z)\le c'_d\exp(-c_dnz^2)$ for all $n>0$ and $z\ge0$. For $\epsilon>0$, set $\theta_\epsilon=\exp(-c_d\epsilon^2)\in(0,1)$; then,
\[
\Pr\left(\sup_{m\ge n}D_m>\epsilon\right)\le\sum_{m=n}^{\infty}\Pr(D_m>\epsilon)\le c'_d\sum_{m=n}^{\infty}\theta_\epsilon^m=\frac{c'_d}{1-\theta_\epsilon}\theta_\epsilon^n\to0,
\textrm{ as } n\to\infty.
\]
Therefore, the multivariate version of Glivenko-Cantelli Theorem follows, that is,
\[
\|\bbF_n-F\|_{\infty}\as0.
\]
Consider $\xi(\bx)\doteq K_{\vP}\{F(\bx)\}$ and $\hxi_n(\bx)\doteq K_{\vP}\{\bbF_n(\bx)\}$. Since $K_{\vP}$ is a continuous and bounded function, from the preceding result we have that $\hxi_n$ convergence almost surely to $\xi$ (uniformly). Let $P$ and $\bbP_n$ denote the probability measures that correspond to $F$ and $\bbF_{n}$ respectively. Because of bounding of $\xi$, $\hAUK=\int \hxi_n\ud{\bbP_n}\as\int \xi\ud{P}=\AUK$ \citep[see][Sec.~19.4]{vdV1998}.
\end{proof}

\begin{proof}[Proof of the asymptotic normality of $\wAUK$ and $\hAUK$]
Consider the map $K\mapsto\psi(K)$ defined by
\[
\psi(K)\doteq\int_0^1\{1-K_{\vP}(t)\}{\ud{K(t)}}.
\]
Observe that $\AUK=\psi(K)$ and $\wAUK=\psi(\wbbK_n)$. Since $\AUK_{H_0}=\psi(K_\vP)=1/2$, using \eqref{eq.Kendall-process}, an application of functional \emph{delta}-method \citep[see][Theo.~20.8, p.~297]{vdV1998} implies that
\[
\surd{n}(\wAUK_n-1/2)\stackrel{H_0}{\rightsquigarrow}\psi'_{\vP}(\bbG_\vP),
\]
where $\psi'_\vP(H)=\int_0^1k_{\vP}(t)H(t)\ud{t}$ is the Hadamard derivative of $\psi$ at $K_\vP$. Due to the linearity of $\psi'_\vP(\cdot)$, it follows that $\psi'_{\vP}(\bbG_\vP)\eqd\CN(0,\sigma_{\vP}^2)$.

We now need to prove that $\surd{n}(\hAUK_n-\wAUK_n)=o_p(1)$ as $n\to\infty$. The Kendall ecdf $\bbK_n$ is the ecdf of the values $\hT_{ni}=\{\# j\colon \bX_j\le\bX_i \textrm{ componentwise}\}/n$; so $\hAUK_n=1-n^{-1}\sum_{i=1}^{n}K_{\vP}(\hT_{ni})$ while $\wAUK_n=1-n^{-1}\sum_{i=1}^{n}K_{\vP}(\wT_{ni})$. Therefore, $n|\hAUK_n-\wAUK_n|=|\sum_{i=1}^{n}K_{\vP}(\hT_{ni})-K_{\vP}(\wT_{ni})|$. By definition, one can see that $\hT_{ni}-\wT_{ni}=(1-\wT_{ni})/n\in[0,1/n]$; hence, $\wT_{ni}\le\hT_{ni}\le\wT_{ni}+1/n$. Suppose $g(x)=x-x\ln x$, $x\ge0$, with $g(0)=0$ and for each fixed $\varepsilon>0$ consider the function $h_{\varepsilon}(x)=g(x+\varepsilon)-g(x+\varepsilon)$, $x\ge0$. Then, $h_{\varepsilon}$ is a continuous function and $h'_{\varepsilon}(x)=\ln\{x/(x+\epsilon)\}<0$ for all $x>0$. Hence, $\max_{x\in[0,1]}h_{\varepsilon}(x)=h_{\varepsilon}(0)=\varepsilon-\varepsilon\ln\varepsilon$. In view of the preceding results,
\[
\sum_{i=1}^{n}K_{\vP}(\wT_{ni})\le\sum_{i=1}^{n}K_{\vP}(\hT_{ni})\le\sum_{i=1}^{n}K_{\vP}(\wT_{ni}+1/n)\le\sum_{i=1}^{n}K_{\vP}(\wT_{ni})+\{1+\ln n\}.
\]
From this immediately follows that $\wAUK_n\le\hAUK_n$ and $\hAUK_n\le\wAUK_n+(1+\ln n)/n$. Thus,
\[
0\le\surd{n}(\wAUK_n-\hAUK_n)\le(1+\ln n)/\surd{n};
\]
and so, $\surd{n}(\hAUK_n-\wAUK_n)=o_p(1)$ as $n\to\infty$ completing the proof of \eqref{eq.covergence}.

Suppose $d=2$. It remains to compute the exact variance $\sigma_{\vP}^2$ to this case. First we present some preliminary results. Let $k,n$ be nonnegative integers and $t\in[0,1]$. Define the integral
\[
I_{k,n}(t)\doteq\int_{0}^{t}s^k{\ln^n(1/s)\ud{s}}.
\]
Then, one can easily verify that
\[
I_{k,n}(t)=t^{k+1}\sum_{j=0}^{n}\frac{(n)_j}{(k+1)^j}\ln^{n-j}(1/t),
\quad\textrm{with}\quad
I_{k,n}\equiv I_{k,n}(1)=\frac{n!}{(k+1)^n},
\]
where $(x)_j=x\cdots(x-j+1)$ with $(x)_0=1$ denotes the $j$th order descending factorial function at $x$. Finally, we compute
\[
\sigma_{\vP}^2=2\int_0^1\left\{t\ln(1/t)-\ln(1/t)+\ln^2(1/t)\right\}{\left\{\int_{0}^{t}s\ln(1/s)\ud{s}\right\}\ud{t}},
\]
that is
\begin{align}
\sigma_{\vP}^2&=2\int_0^1\left\{t\ln(1/t)-\ln(1/t)+\ln^2(1/t)\right\}I_{1,1}(t)\ud{t}
               =\frac{I_{3,1}}{2}+I_{3,2}-\frac{I_{2,1}}{2}-\frac{I_{2,2}}{2}+I_{2,3}=\frac{19}{432}.\qedhere
\end{align}
\end{proof}

\section*{Supplementary Material}

\begin{description}

\item[\emph{Title:}]
Supplementary Material: An AUK-based index for measuring and testing the joint dependence of a random vector

\item[I. F-G-M copulas:]
Description of the Farlie–Gumbel–Morgenstern copulas that used in Subsubsection \ref{sssec.sim.non-normal}.

\item[II. Bivariate distributions that are used in Subsection \ref{ssec.sim.test}:]
Description of the bivariate distributions that are used in Subsection \ref{ssec.sim.test}

\item[III. Numerical results:]
The true values of the index $I$ in the normal distribution case $\CN_3(\bzero,\Sig_3(\rho))$ for various values of $\rho$.

\item[IV. A data set of biomarkers of hepatic injury:]
The data set of biomarkers of hepatic injury that used in Section \ref{sec.RealData}.

\item[V. \emph{R} Codes:]
\emph{R} codes to implement the developed method in this article.
\end{description}

\section*{Acknowledgements}

The second author acknowledges financial support in the form of a grant award (award number 82114), from the Troup Fund, KALEIDA Health Foundation. The authors thank Dr.\ Andrew H.\ Talal for providing the data set and for helpful discussions aided the interpretation of the results.

\newpage
\phantom{.}
\thispagestyle{empty}
%\vspace{1cm}
%\phantom{.}
%
%\phantom{.}
%\thispagestyle{empty}
\newpage

%%%%%%%%%%%%%%%%%%%%%%%%%%%%%%%%% Supplementary Material %%%%%%%%%%%%%%%%%%%%%%%%%%%%%%%%%
%\renewcommand{\thesection}{SM}
\renewcommand{\thesection}{\Roman{section}}
\fancypagestyle{firstpage}{%
  \lhead{} \chead{} \rhead{}
  \lfoot{} \cfoot{\thepage} \rfoot{}
}
\thispagestyle{firstpage}
\setcounter{page}{1}
\setcounter{section}{0}
\title{\Large\bf Supplementary Material: \sc An AUK-based index for measuring and testing the joint dependence of a random vector}
\maketitle{}
\renewcommand{\runtitle}{Supplementary Material: An AUK-based index for measuring  the joint dependence of a random vector}

\section{F-G-M copulas}
\label{sm:FGM}
We consider next two special cases of trivariate $\FGM$ copulas. For $\bu=(u_1,u_2,u_3)^T\in[0,1]^3$, consider the copulas
\begin{subequations}\label{eq.F-G-M}
\begin{equation}\label{eq.F-G-M(a)}
C_{\theta}(\bu)=u_1u_2u_3\{1+\theta(1-u_1)(1-u_2)\};
\end{equation}
\begin{equation}\label{eq.F-G-M(b)}
\wC_{\theta}(\bu)=u_1u_2u_3\{1+\theta(1-u_1)(1-u_2)(1-u_3)\},
\end{equation}
\end{subequations}
where $|\theta|\le1$ (see \citealp{JK1975}; \citealp{GLNQR2001}, Remark~4, p.~9). Obviously, both $C_{\theta=0}$ and $\wC_{\theta=0}$ correspond to the independent copula case $u_1u_2u_3$.

Suppose $\bU=(U_1,U_2,U_3)^T\sim C_{\theta}$, see \eqref{eq.F-G-M(a)}. Obviously, $(U_1,U_2)^T$ and $U_3$ are independent; also, one can easily verify that the cdf of $U_1|U_2=u_2$ is $F_{U_1|U_2}(u_1|u_2)=\alpha u_1^2+(1-\alpha)u_1$, $0\le u_1\le1$ (for each $0\le u_2\le1$), where $\alpha\equiv\alpha(u_2;\theta)=\theta(2u_2-1)$. Thus,
\[
F^{-1}_{U_1|U_2}(u_1|u_2)
=\left\{
  \begin{array}{c@{\quad\textrm{if} \ \ }l}
    u_1, & \alpha=0, \\
    u_{1,1}\one{0\le u_{1,1}\le1}+u_{1,2}\one{0\le u_{1,2}\le1}, & \alpha\ne0, \\
  \end{array}
 \right.
\]
where $u_{1,1}=[\alpha-1-\{(1-\alpha)^2+4\alpha u_1\}^{1/2}]/(2\alpha)$ and $u_{1,2}=[\alpha-1+\{(1-\alpha)^2+4\alpha u_1\}^{1/2}]/(2\alpha)$.

Let now $\w\bU=(\w{U}_1,\w{U}_2,\w{U}_3)^T\sim\wC_{\theta}$, see \eqref{eq.F-G-M(b)}. The cdf of $(\w{U}_2,\w{U}_3)^T$ is $F_{\w{U}_2,\w{U}_3}(u_2,u_3)=u_2u_3$, $(u_2,u_3)^T\in[0,1]^2$, that is $\w{U}_2$ and $\w{U}_3$ are independent. Moreover, the cdf of $\w{U}_1|(\w{U}_2,\w{U}_3)^T=(u_2,u_3)^T$ is $F_{\w{U}_1|\w{U}_2,\w{U}_3}(u_1|u_2,u_3)=\w{\alpha}u_1^2+(1-\w{\alpha})u_1$, $0\le u_1\le1$ (for each $0\le u_2,u_3\le1$), where $\w{\alpha}\equiv\w{\alpha}(u_2,u_3;\theta)=-\theta(2u_2-1)(2u_3-1)$. Thus,
\[
\w{F}^{-1}_{\w{U}_1|\w{U}_2,\w{U}_3}(u_1|u_2,u_3)
=\left\{
  \begin{array}{c@{\quad\textrm{if} \ \ }l}
    u_1, & \w{\alpha}=0, \\
    \w{u}_{1,1}\one{0\le \w{u}_{1,1}\le1}+\w{u}_{1,2}\one{0\le \w{u}_{1,2}\le1}, & \w{\alpha}\ne0, \\
  \end{array}
 \right.
\]
where $\w{u}_{1,1}=[\w{\alpha}-1-\{(1-\w{\alpha})^2+4\w{\alpha} u_1\}^{1/2}]/(2\w{\alpha})$ and $\w{u}_{1,2}=[\w{\alpha}-1+\{(1-\w{\alpha})^2+4\w{\alpha} u_1\}^{1/2}]/(2\w{\alpha})$.

From the preceding analysis, if $X_1$, $X_2$ and $X_3$ are independent $U(0,1)$ distributions, then $\left(F^{-1}_{U_1|U_2}(X_1|X_2),X_2,X_3\right)^T\sim C_{\theta}$ and $\left(\w{F}^{-1}_{\w{U}_1|\w{U}_2,\w{U}_3}(X_1|X_2,X_3),X_2,X_3\right)^T\sim \wC_{\theta}$.

\section{Bivariate distributions that are used in Subsection \ref{ssec.sim.test}}
\label{sm:BVdistr}

\begin{description}
\item[Bivariate uniform distribution on the circumference of a circle, $\bm{U\{C(0,1)\}}$]
Let $Z_1$, $Z_2$ be iid standard normal distributions. Consider the rvs $U_i=Z_i\big/\big(Z_1^2+Z_2^2\big)^{1/2}$, $i=1,2$. Then, $(U_1,U_2)$ is uniformly distributed on the circumference of the circle $x^2+y^2=1$ .

\item[Bivariate $\bm{\exp\{\lambda_1,\lambda_2,\lambda_{12}\}}$ distribution]
Let $\lambda_1, \lambda_2, \lambda_{12}>0$ with $\lambda_{12}<\min\{\lambda_1, \lambda_2\}$. Suppose that $E_1\sim \exp(\lambda_1-\lambda_{12})$, $E_2\sim \exp(\lambda_2-\lambda_{12})$ and $E_3\sim \exp(\lambda_{12})$, and consider the rvs $X_i=\min\{E_i,E_3\}-1/\lambda_i$, $i=1,2$. Then, $(X_1,X_2)$ follows the bivariate exponential distribution.

\item[Bivariate $\Morgenstern\bm{\{\alpha\}}$ distribution]
Let $\alpha>0$ and $X$, $U$ be independent $U(0,1)$ distributed. Consider the rvs $Z=\alpha(2X-1)-1$, $W=1-2\alpha(2X-1)+\alpha^2(2X-1)^2+4\alpha U(2X-1)$ and $Y=2U/(W^{1/2}-Z)$. Then, $(X,Y)$ follows the bivariate Morgenstern distribution.

\item[Bivariate $\Plackett\bm{\{s\}}$ distribution]
Let $s>1$ and $X$, $U$ be independent $U(0,1)$ distributed. Consider the rvs $W_1=U(1-U)$, $W_2=s+W_1(s-1)^2$, $W_3=2W_1(s^2X+1-X)+s(1-2W_1)$, $W_4=s\{s+4(1-s)^2X(1-X)W_1\}$ and $Y=W_2\{W_3-(1-2U)W_4^{1/2}\}/2$. Then, $(X,Y)$ follows the bivariate Plackett distribution.

\item[Bivariate $\ALIHAQ\bm{\{a,p\}}$ distribution]
Let $a>0$, $0<p<1$ and $X$, $U$ be independent $U(0,1)$ distributed. Consider the rvs $V=1-U$, $W_1=a(2VU+1)+2a^2V^2U+1$, $W_2=a^2(4V^2U-4VU+1)+a(4VU-4p+2)+1$ and $Y=2U(aV-1)^2/(W_1+W_2^{1/2})$. Then, $(X,Y)$ follows the bivariate ALI-HAQ's distribution.

\item[Bivariate $\Gumbel\bm{\{e\}}$ distribution]
Let $e>0$ and $U_1\ldots,U_4$ be independent $U(0,1)$ distributed. Consider the rvs $X=-\ln U_1$, $W_1=1+eX$, $W_2=(W_1-e)/W_1$, $W_3=-\ln U_2$ and $Y=\one{U_3<W_2}W_1W_3+\one{U_3\ge W_2}W_1(W_3-\ln U_4)$. Then, $(X,Y)$ follows the bivariate Gumbel distribution.

\item[Bivariate $\bm{t_5}$ distribution]
Using the package ``mvtnorm'' in R, we simulate data from bivariate student distribution with $5$ degree of freedom and variance-covariance matrix $\Sig=\big({1\atop1}~{1\atop4}\big)$.
\end{description}
\pagebreak

\section{Numerical results}
\label{sm:numerical}

This sections presents some useful numerical results. Table \ref{table.(r,I(r))} contains the true values of the index $I$ in the normal distribution case $\CN_3(\bzero,\Sig_3(\rho))$, where the variance/covariance matrix $\Sig_3(\rho)$ is defined by \eqref{eq.Sig(b)}, for various values of $\rho$. Additionally, Figure \ref{fig.N3Scatterplots0} presents scatterplots based on random samples from $\CN_3(\bzero,\Sig_3(\rho))$ distributions for various values of $\rho$, giving an illustration of the shapes of these  distributions.

\begin{table*}[htp]
 \centering{
 \caption[The index $I\equiv I(\rho)$ for various values of $\rho$ under the trivariate normal distribution.]{The index $I\equiv I(\rho)$ for various values of $\rho$ when the underline distribution is $\CN_3\left(\bzero,\Sig_3(\rho)\right)$.}
 \label{table.(r,I(r))}
 \scriptsize
 \begin{tabular*}{\textwidth}
 {@{\hspace{0ex}}l@{\extracolsep{\fill}}l@{\hspace{0ex}}l@{\hspace{0ex}}l@{\hspace{0ex}}l@{\hspace{0ex}}l@{\hspace{0ex}}l@{\hspace{0ex}}l@{\hspace{0ex}}l@{\hspace{0ex}}l@{\hspace{0ex}}l@{\hspace{0ex}}l@{\hspace{0ex}}l@{\hspace{0ex}}r@{\hspace{0ex}}}
 \addlinespace
 \toprule
 $\rho\colon$ & 0 & 0.05 & 0.10 & 0.15 & 0.20 & 0.25 & 0.30 & 0.35 & 0.40 & 0.45 & 0.50 & 0.55  &   \\
 $I\colon$    & 0 & 0.027 & 0.052 & 0.077 & 0.101 & 0.124 & 0.148 & 0.171 & 0.194 & 0.218 & 0.243 & 0.267 \\
 \midrule
 $\rho\colon$ & 0.60 & 0.65 & 0.70 & 0.75 & 0.80 & 0.85 & 0.90 & 0.95 & 0.97 & 0.98 & 0.99 & 0.995 & 1 \\
 $I\colon$    & 0.294 & 0.321 & 0.352 & 0.385 & 0.423 & 0.467 & 0.524 & 0.609 & 0.662 & 0.701 & 0.758 & 0.841 & 1 \\
 \bottomrule
 \end{tabular*}}
 \end{table*}

\begin{figure}[htp]
\begin{subfigure}[a]{.24\textwidth}
\centering
\FIG
{\resizebox{\linewidth}{!}{
\begin{tikzpicture}
\begin{axis}[view={30}{20}]
\addplot3+[draw=none, mark=o,color=blue,mark size=1] table {normal0.dat};
\end{axis}
\end{tikzpicture}
}}
{\includegraphics[width=\textwidth]{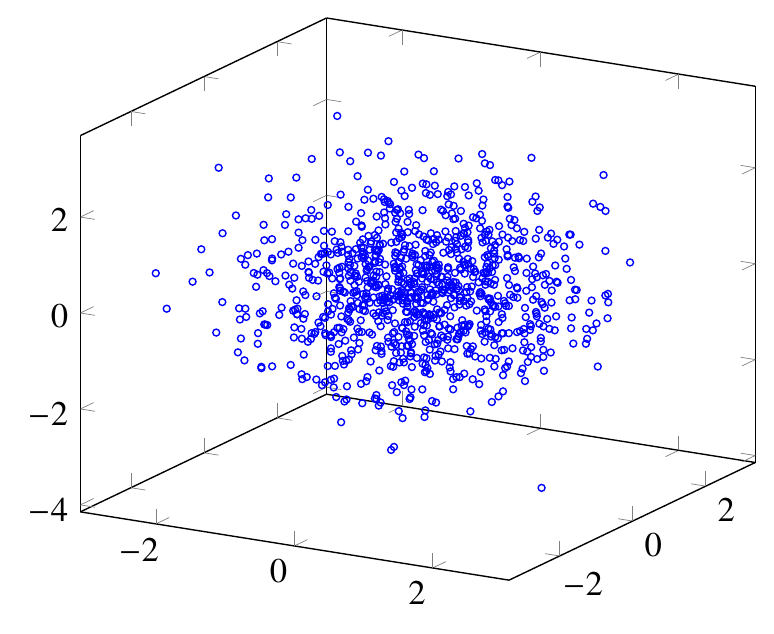}}
\caption{$\rho=0$.}
\label{sfig.normal0}
\end{subfigure}
\hfill
%\begin{subfigure}[a]{.24\textwidth}
%\centering
%\FIG
%{\resizebox{\linewidth}{!}{
%\begin{tikzpicture}
%\begin{axis}[view={30}{20}]
%\addplot3+[draw=none, mark=o,color=blue,mark size=1] table {normal2.dat};
%\end{axis}
%\end{tikzpicture}
%}}
%{\includegraphics[width=\textwidth]{normal2}}
%\caption{$\rho=0.2$.}
%\label{sfig.normal0.2}
%\end{subfigure}
%\hfill
\begin{subfigure}[a]{.24\textwidth}
\centering
\FIG
{\resizebox{\linewidth}{!}{
\begin{tikzpicture}
\begin{axis}[view={30}{20}]
\addplot3+[draw=none, mark=o,color=blue,mark size=1] table {normal25.dat};
\end{axis}
\end{tikzpicture}
}}
{\includegraphics[width=\textwidth]{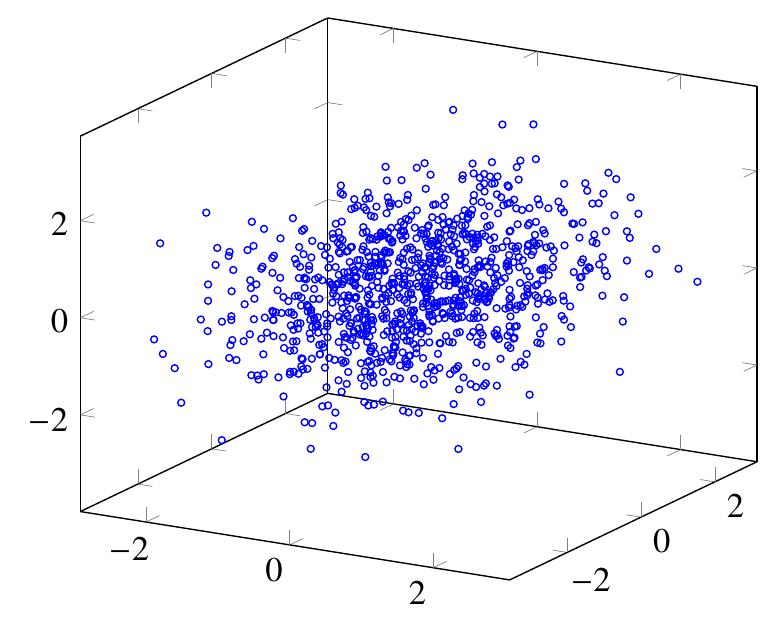}}
\caption{$\rho=0.25$.}
\label{sfig.normal0.25}
\end{subfigure}
\hfill
\begin{subfigure}[a]{.24\textwidth}
\FIG
{\resizebox{\linewidth}{!}{
\begin{tikzpicture}
\begin{axis}[view={30}{20}]
\addplot3+[draw=none, mark=o,color=blue,mark size=1] table {normal4.dat};
\end{axis}
\end{tikzpicture}
}}
{\includegraphics[width=\textwidth]{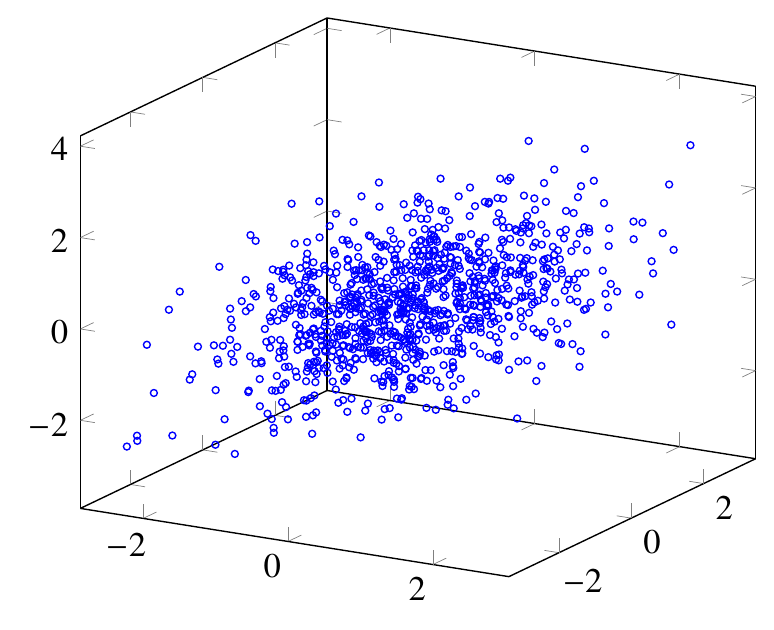}}
\caption{$\rho=0.4$.}
\label{sfig.normal0.4}
\end{subfigure}
\hfill
\begin{subfigure}[a]{.24\textwidth}
\centering
\FIG
{\resizebox{\linewidth}{!}{
\begin{tikzpicture}
\begin{axis}[view={30}{20}]
\addplot3+[draw=none, mark=o,color=blue,mark size=1] table {normal5.dat};
\end{axis}
\end{tikzpicture}
}}
{\includegraphics[width=\textwidth]{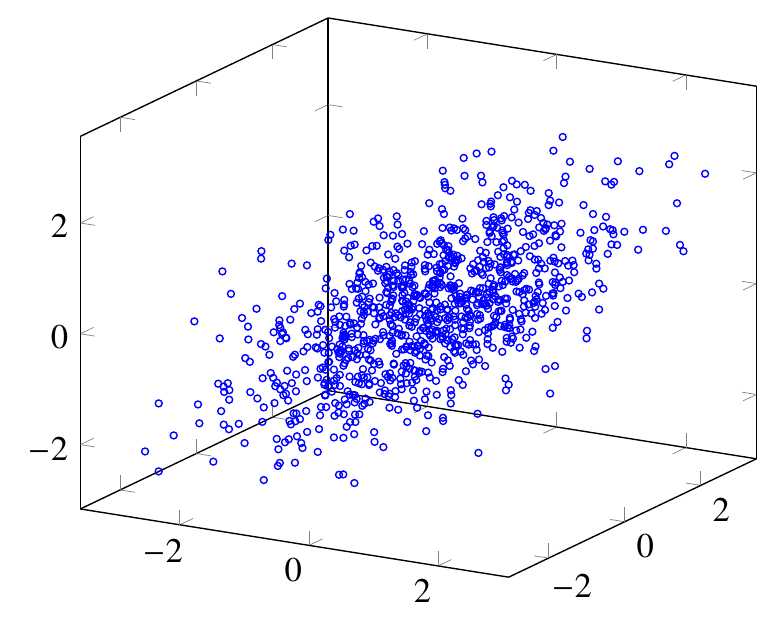}}
\caption{$\rho=0.5$.}
\label{sfig.normal0.5}
\end{subfigure}
\medskip\linebreak
\begin{subfigure}[a]{.24\textwidth}
\centering
\FIG
{\resizebox{\linewidth}{!}{
\begin{tikzpicture}
\begin{axis}[view={30}{20}]
\addplot3+[draw=none, mark=o,color=blue,mark size=1] table {normal6.dat};
\end{axis}
\end{tikzpicture}
}}
{\includegraphics[width=\textwidth]{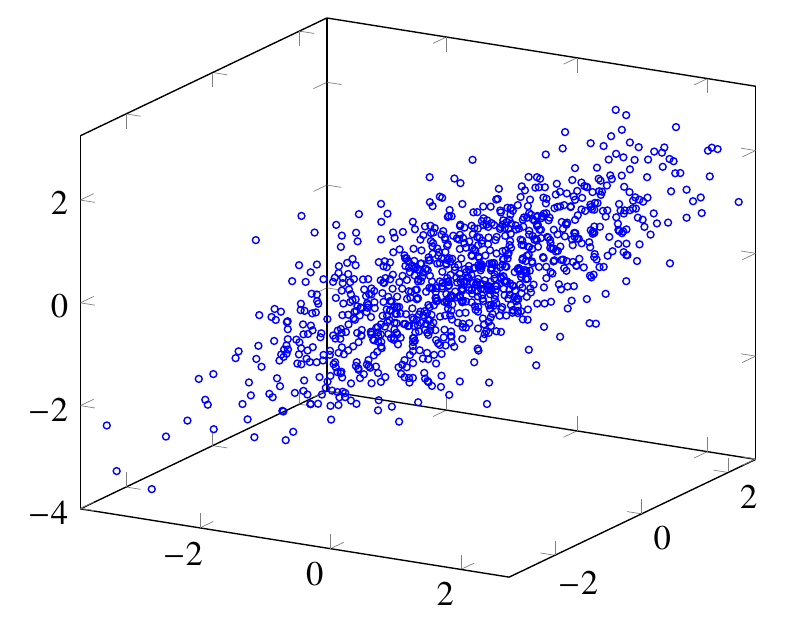}}
\caption{$\rho=0.6$.}
\label{sfig.normal0.6}
\end{subfigure}
\hfill
\begin{subfigure}[a]{.24\textwidth}
\centering
\FIG
{\resizebox{\linewidth}{!}{
\begin{tikzpicture}
\begin{axis}[view={30}{20}]
\addplot3+[draw=none, mark=o,color=blue,mark size=1] table {normal75.dat};
\end{axis}
\end{tikzpicture}
}}
{\includegraphics[width=\textwidth]{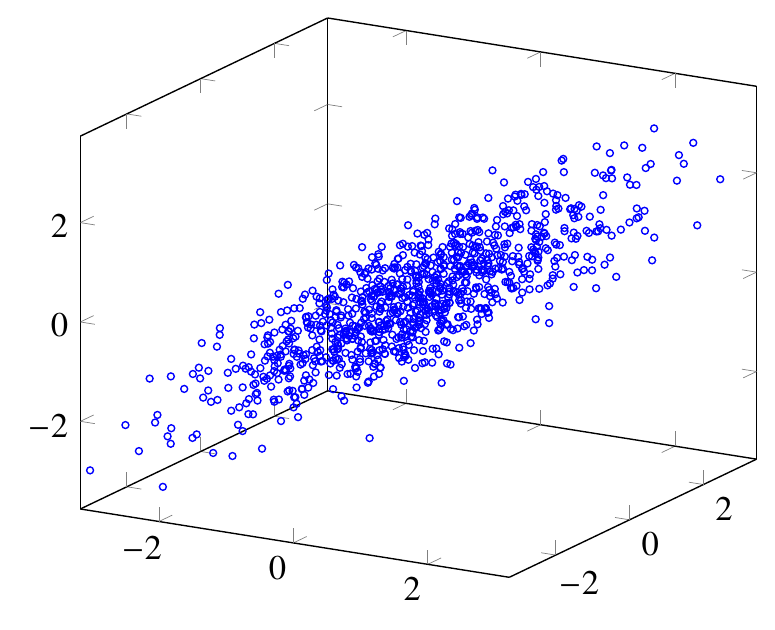}}
\caption{$\rho=0.75$.}
\label{sfig.normal75}
\end{subfigure}
\hfill
\begin{subfigure}[a]{.24\textwidth}
\centering
\FIG
{\resizebox{\linewidth}{!}{
\begin{tikzpicture}
\begin{axis}[view={30}{20}]
\addplot3+[draw=none, mark=o,color=blue,mark size=1] table {normal8.dat};
\end{axis}
\end{tikzpicture}
}}
{\includegraphics[width=\textwidth]{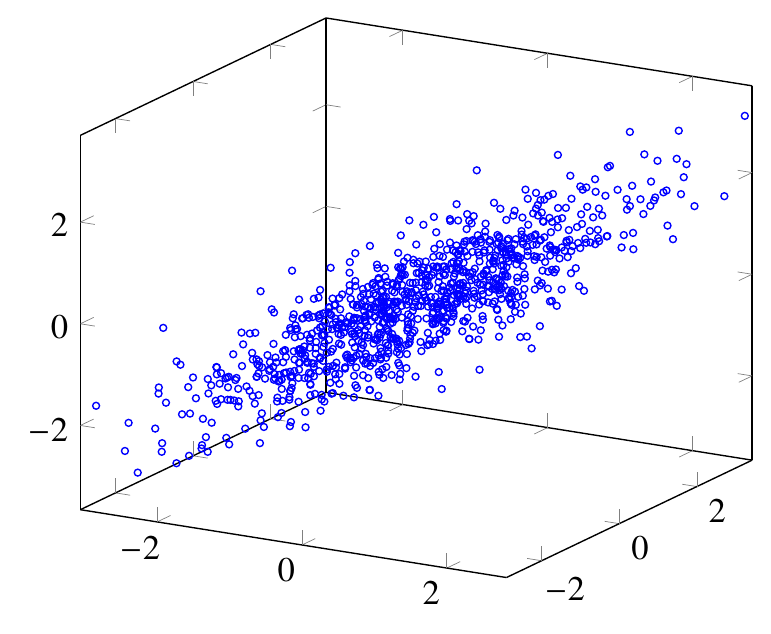}}
\caption{$\rho=0.8$.}
\label{sfig.normal0.8}
\end{subfigure}
\hfill
\begin{subfigure}[a]{.24\textwidth}
\centering
\FIG
{\resizebox{\linewidth}{!}{
\begin{tikzpicture}
\begin{axis}[view={30}{20}]
\addplot3+[draw=none, mark=o,color=blue,mark size=1] table {normal10.dat};
\end{axis}
\end{tikzpicture}
}}
{\includegraphics[width=\textwidth]{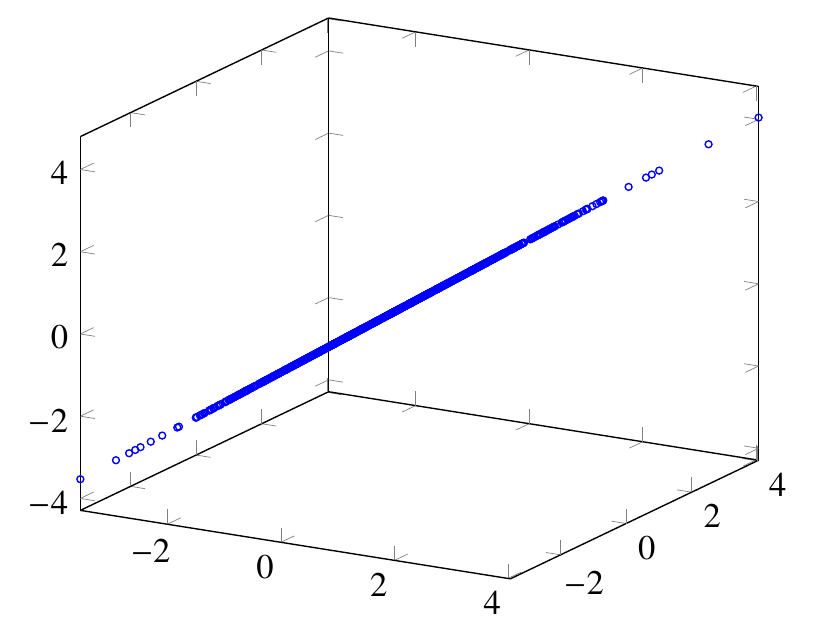}}
\caption{$\rho=1$.}
\label{sfig.normal1}
\end{subfigure}
\caption[Scatterplots: Random samples of size $n=1000$ are drawn from the trivariate normal distributions]{Scatterplots: Random samples of size $n=1000$ are drawn from the trivariate normal distributions $\CN_3(\bzero,\Sig_3(\rho))$ for various values of $\rho$.}
\label{fig.N3Scatterplots0}
\end{figure}
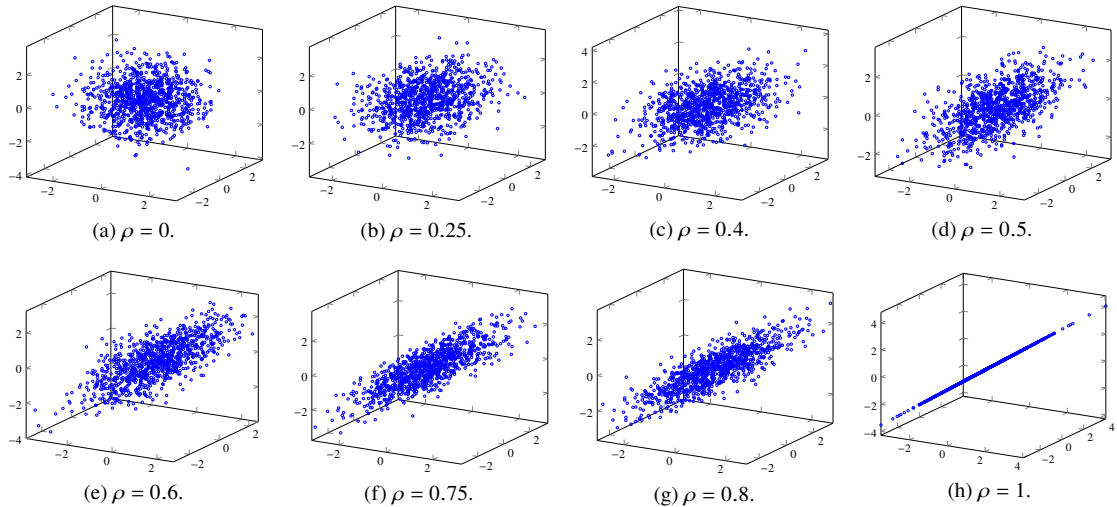
\newpage

\section{A data set of biomarkers of hepatic injury}
\label{sm:data(B,AST,ALT,AP)}
 \begin{table}[htp]
 \caption{A data set of size $n=208$ of the biomarkers Bilirubin (B), Alanine and Aspartate aminotransferase (ALT and AST) and Alkaline phosphatase (AP).}
 \label{table.DataSet}
 \scriptsize
 \begin{tabular*}{\textwidth}
 {@{}@{\extracolsep{\fill}}r@{}S@{}r@{~~}r@{~~}r@{}c@{\quad}r@{}S@{}r@{~~}r@{~~}r@{}c@{\quad}r@{}S@{}r@{~~}r@{~~}r@{}c@{\quad}r@{}S@{}r@{~~}r@{~~}r@{}}
 \toprule
 PID\!\! & B & AST & ALT & AP~\! & & PID\!\! & B & AST & ALT & AP~\! & & PID & B & AST & ALT & AP~\! & & PID & B & AST & ALT & AP~\! \\
 \cmidrule{1-5} \cmidrule{7-11} \cmidrule{13-17} \cmidrule{19-23}
  1 &  1.1 &  34 &  54 & 150 &  &  53 &  0.8 &  29 &  21 & 122 &  & 105 &  0.8 &  27 &  20 &  63 &  & 157 &  0.4 &  15 &  16 &  96 \\ [-.55ex]
  2 &  0.8 &  32 &  23 &  90 &  &  54 &  0.3 &  26 &  25 & 119 &  & 106 &  0.4 &  34 &  59 & 143 &  & 158 &  1.1 &  96 &  75 &  91 \\ [-.55ex]
  3 &  0.4 &  17 &  21 & 105 &  &  55 &  0.8 &  59 &  34 &  86 &  & 107 &  4.4 &  76 &  42 & 126 &  & 159 &  1.2 &  38 &  28 & 116 \\ [-.55ex]
  4 &  0.2 &  10 &  17 & 178 &  &  56 &  0.7 & 171 & 217 & 107 &  & 108 &  3.2 &  35 &  24 & 118 &  & 160 &  0.5 &  15 &  18 &  65 \\ %[1ex]
  5 &  0.5 &  26 &  38 &  57 &  &  57 &  0.5 &  15 &  11 &  81 &  & 109 &  0.4 &  41 &  34 &  47 &  & 161 &  0.2 &  22 &  19 & 100 \\ [-.55ex]
  6 &  1.6 & 144 & 116 & 158 &  &  58 &  0.2 &  13 &  13 & 109 &  & 110 &  1.2 &  57 &  50 & 115 &  & 162 &  0.3 &  53 &  47 &  83 \\ [-.55ex]
  7 &  1.1 &  51 &  28 & 145 &  &  59 &  0.4 &  51 &  64 &  71 &  & 111 &  1.6 &  22 &  20 &  63 &  & 163 &  1.6 &  99 &  56 &  81 \\ [-.55ex]
  8 &  0.5 &  26 &  15 & 207 &  &  60 &  0.5 &  15 &  10 &  88 &  & 112 &  2   &  21 &  18 &  63 &  & 164 &  1.6 &  37 &  63 & 231 \\ %[1ex]
  9 &  1.3 &  48 &  31 & 295 &  &  61 &  0.3 &  15 &  15 & 121 &  & 113 &  0.5 &  26 &  38 &  71 &  & 165 &  0.4 &  40 &  13 & 273 \\ [-.55ex]
 10 &  0.9 &  25 &  30 &  76 &  &  62 &  1.2 &  22 &  17 &  61 &  & 114 &  0.2 &  14 &  15 &  81 &  & 166 &  1   &  57 &  85 & 189 \\ [-.55ex]
 11 &  0.7 &  57 &  52 &  88 &  &  63 &  0.8 &  16 &  18 &  94 &  & 115 &  1.3 &  35 &  32 &  61 &  & 167 &  6.4 & 108 &  68 & 108 \\ [-.55ex]
 12 &  0.2 &  34 &  35 &  76 &  &  64 &  0.2 &  31 &  30 &  48 &  & 116 &  0.4 &  15 &  14 &  72 &  & 168 &  1.1 &  54 &  79 &  91 \\ %[1ex]
 13 &  0.5 &  29 &  14 &  66 &  &  65 &  0.5 &  29 &  25 & 104 &  & 117 &  0.5 &  30 &  51 &  66 &  & 169 &  0.7 &  28 &  56 & 202 \\ [-.55ex]
 14 &  1.7 &  30 &  26 & 121 &  &  66 &  0.6 &  24 &  18 &  90 &  & 118 &  0.8 &  40 &  35 &  82 &  & 170 &  0.7 &  38 &  51 &  60 \\ [-.55ex]
 15 &  1.7 &  21 &  13 &  94 &  &  67 &  0.3 & 153 & 183 & 246 &  & 119 &  0.8 &  46 &  94 & 170 &  & 171 &  0.4 &  15 &  11 & 106 \\ [-.55ex]
 16 &  1.2 &  26 &  16 & 207 &  &  68 &  0.8 &  49 & 121 &  75 &  & 120 &  1   &  57 &  85 & 189 &  & 172 &  0.4 &  36 &  45 & 159 \\ %[1ex]
 17 &  0.9 &  34 &  38 &  59 &  &  69 & 33.3 & 121 & 117 & 434 &  & 121 &  1.3 & 103 &  21 &  71 &  & 173 &  0.6 &  61 &  47 & 104 \\ [-.55ex]
 18 &  0.8 &  42 &  89 &  72 &  &  70 &  1   &  34 &  47 &  68 &  & 122 &  1.5 &  18 &  22 &  57 &  & 174 &  0.9 &  18 &  26 &  96 \\ [-.55ex]
 19 &  0.5 &  50 &  73 &  68 &  &  71 &  0.2 &  18 &  24 &  48 &  & 123 &  0.6 &  46 &  61 & 140 &  & 175 &  0.7 &  49 &  81 & 187 \\ [-.55ex]
 20 &  0.4 &  43 &  35 &  87 &  &  72 &  0.7 &  24 &  14 &  97 &  & 124 &  0.7 &  17 &  19 &  53 &  & 176 &  0.4 &  22 &  16 &  58 \\ %[1ex]
 21 &  0.8 &  19 &  26 & 112 &  &  73 &  4.5 & 236 &  64 & 242 &  & 125 &  0.8 &  29 &  30 &  81 &  & 177 &  2.9 &  51 &  33 & 143 \\ [-.55ex]
 22 &  0.7 &  16 &  11 &  36 &  &  74 &  0.6 &  74 & 188 &  72 &  & 126 &  0.7 &  80 &  93 &  61 &  & 178 &  1.4 &  24 & 198 & 108 \\ [-.55ex]
 23 &  0.4 &  20 &  22 &  99 &  &  75 &  0.2 &  20 &  23 &  80 &  & 127 &  0.7 &  32 &  22 &  80 &  & 179 &  0.4 &  29 &  41 & 108 \\ [-.55ex]
 24 &  0.5 &  40 &  58 &  71 &  &  76 &  0.5 &  18 &  18 &  58 &  & 128 &  0.8 & 105 & 170 & 108 &  & 180 &  0.4 &  26 &  18 & 121 \\ %[1ex]
 25 &  0.5 &  44 &  59 &  79 &  &  77 &  0.7 &  85 &  39 &  89 &  & 129 &  0.6 &  44 &  50 & 197 &  & 181 &  0.5 &  24 &  18 & 147 \\ [-.55ex]
 26 &  1.2 &  30 &  22 &  49 &  &  78 &  1.4 &  64 &  44 & 183 &  & 130 &  0.3 &  64 &  24 &  74 &  & 182 &  0.5 &  81 & 127 &  68 \\ [-.55ex]
 27 &  0.2 &  20 &  17 & 118 &  &  79 &  1.6 &  18 &  15 &  73 &  & 131 &  0.6 &  47 &  45 &  85 &  & 183 &  0.6 &  26 &  24 &  67 \\ [-.55ex]
 28 &  0.7 &  36 &  27 & 129 &  &  80 &  0.2 &  27 &  46 & 175 &  & 132 &  0.4 & 113 &  77 &  96 &  & 184 &  2.1 &  55 &  83 &  93 \\ %[1ex]
 29 &  1.3 &  58 &  27 &  75 &  &  81 &  1.4 & 101 &  61 &  44 &  & 133 &  0.9 &  21 &  16 &  55 &  & 185 &  0.7 &  26 &  40 & 103 \\ [-.55ex]
 30 & 35.2 & 182 &  82 & 147 &  &  82 &  0.5 &  89 &  91 & 125 &  & 134 &  1.1 &  39 &  26 &  82 &  & 186 &  0.7 &  23 &  14 & 148 \\ [-.55ex]
 31 &  0.3 &  25 &  28 &  80 &  &  83 &  1.2 &  45 &  74 &  91 &  & 135 &  0.7 &  57 &  69 &  83 &  & 187 &  1.3 &  21 &  32 &  72 \\ [-.55ex]
 32 &  0.9 &  20 &  22 &  89 &  &  84 &  1.6 & 514 & 377 & 876 &  & 136 &  0.9 &  35 &  56 & 185 &  & 188 &  0.8 &  34 &  29 &  76 \\ %[1ex]
 33 &  1.2 &  55 &  61 &  52 &  &  85 &  0.3 &  49 &  51 & 104 &  & 137 &  0.7 &  15 &  30 & 196 &  & 189 &  0.3 &  56 &  88 &  71 \\ [-.55ex]
 34 &  1.7 &  34 &  14 &  75 &  &  86 &  0.9 &  43 &  24 &  70 &  & 138 &  0.3 &  44 &  49 &  80 &  & 190 &  0.4 &  23 &  18 &  99 \\ [-.55ex]
 35 &  1.1 &  35 &  39 & 110 &  &  87 &  0.6 &  61 &  80 &  59 &  & 139 &  0.6 & 108 & 133 &  78 &  & 191 &  0.7 &  39 &  52 & 130 \\ [-.55ex]
 36 &  0.4 &  21 &  27 &  73 &  &  88 &  0.7 &  30 &  25 &  68 &  & 140 &  0.6 &  25 &  19 &  47 &  & 192 &  0.6 &  51 &  75 & 111 \\ %[1ex]
 37 &  0.4 &  24 &  30 &  77 &  &  89 &  0.8 &  20 &  20 &  71 &  & 141 &  1.2 &  30 &  23 &  52 &  & 193 &  0.6 &  21 &  31 &  63 \\ [-.55ex]
 38 &  0.5 &  38 &  37 &  67 &  &  90 &  0.4 &  26 &  34 &  57 &  & 142 &  0.6 &  37 &  41 &  84 &  & 194 &  0.6 &  25 &  18 &  93 \\ [-.55ex]
 39 &  0.7 &  22 &  22 &  56 &  &  91 &  0.4 &  23 &  32 & 147 &  & 143 &  0.3 &  83 &  76 &  73 &  & 195 &  0.5 &  82 & 146 &  50 \\ [-.55ex]
 40 &  0.6 &  53 &  37 & 128 &  &  92 &  0.5 &  82 &  61 &  67 &  & 144 &  0.6 & 143 & 274 &  68 &  & 196 &  0.3 &  67 & 106 &  72 \\ %[1ex]
 41 &  0.4 &  25 &  21 &  97 &  &  93 &  0.7 &  28 &  42 &  64 &  & 145 & 13.4 & 185 &  82 & 127 &  & 197 &  0.3 &  19 &  23 &  86 \\ [-.55ex]
 42 &  0.5 &  52 &  75 &  68 &  &  94 &  2.2 &  71 &  46 & 181 &  & 146 &  0.4 &  85 & 183 &  68 &  & 198 &  0.3 &  24 &  47 &  79 \\ [-.55ex]
 43 &  0.6 &  28 &  33 & 142 &  &  95 &  0.8 &  26 &  29 &  95 &  & 147 &  0.6 & 299 & 190 &  63 &  & 199 &  0.6 &  46 &  37 &  71 \\ [-.55ex]
 44 &  1   &  33 &  25 &  59 &  &  96 &  0.5 &  32 &  39 &  38 &  & 148 &  0.5 &  35 &  46 &  78 &  & 200 &  0.5 &  37 &  49 &  94 \\ %[1ex]
 45 &  0.4 &  28 &  39 & 164 &  &  97 &  2.2 &  29 &  33 &  64 &  & 149 &  0.3 &  14 &  22 & 105 &  & 201 &  0.4 &  23 &  44 &  69 \\ [-.55ex]
 46 &  0.7 &  39 &  44 &  84 &  &  98 &  1   &  40 &  33 &  83 &  & 150 &  0.4 &  48 &  46 &  69 &  & 202 &  0.5 &  19 &  20 & 139 \\ [-.55ex]
 47 &  0.8 &  28 &  41 &  68 &  &  99 &  1.5 &  52 &  48 &  53 &  & 151 &  0.7 &  27 &  38 &  99 &  & 203 &  0.9 & 208 & 346 &  73 \\ [-.55ex]
 48 &  1.2 &  64 &  26 & 172 &  & 100 &  0.9 &  85 &  67 &  49 &  & 152 &  0.9 &  37 &  52 &  50 &  & 204 &  0.6 &  57 &  64 &  92 \\ %[1ex]
 49 &  8.4 &  45 &  47 &  98 &  & 101 &  0.7 &  30 &  34 & 113 &  & 153 &  0.4 &  23 &  25 &  70 &  & 205 &  0.5 &  35 &  60 &  35 \\ [-.55ex]
 50 &  3.7 &  74 & 102 & 245 &  & 102 &  1.5 &  16 &  16 &  66 &  & 154 &  0.4 &  29 &  47 &  99 &  & 206 & 12.7 & 156 & 134 & 318 \\ [-.55ex]
 51 &  0.8 &  32 &  35 &  57 &  & 103 &  0.4 &  51 &  78 & 115 &  & 155 &  3.2 &  54 &  26 & 110 &  & 207 &  0.7 &  98 & 139 &  81 \\ [-.55ex]
 52 &  1   & 293 & 227 & 147 &  & 104 &  0.4 & 125 & 141 & 140 &  & 156 &  0.7 &  49 &  88 &  91 &  & 208 &  0.5 &  15 &  12 & 102 \\
 \bottomrule
 \end{tabular*}
 \end{table}
\newpage

\section{\emph{R} codes}
\label{sm:R}

\begin{lstlisting}[title={\normalsize\bfseries\scshape The required packages that used.}, basicstyle=\small, language=R]
rm(list=ls())
library(MASS); require(MASS); sfLibrary(MASS)
library(energy); require(energy)
library(Emcdf); sfLibrary(Emcdf)
library(mvtnorm)
library(gtools)
library(copula)
library("snowfall");library("rlecuyer")
sfInit(parallel=TRUE, cpus=8, type="SOCK")
library(dHSIC)
library(MonoPoly)
\end{lstlisting}

\begin{lstlisting}[title={\normalsize\bfseries\scshape The specific formula of $\bm{1-K_{\vP}(t)}$ when $\bm{d=3}$.}, basicstyle=\small, language=R]
g <- function(t){ifelse(t==0,1,1-t+t*log(t)-0.5*t*(log(t))^2)}
g <- Vectorize(g)
\end{lstlisting}

\begin{lstlisting}[title={\normalsize\bfseries\scshape The pdf $\bm{f}$ of $\bm{\CN_3(\bzero,\Sig_3(\brho))}$ distribution.}, basicstyle=\small, language=R]
f <- function(x,y,z,r12,r13,r23){
S <- matrix(c(1,r12,r13,r12,1,r23,r13,r23,1),3,3)
density <- dmvnorm(c(x,y,z), mean=c(0,0,0), sigma=S)
return(density)}
f <- Vectorize(f)
\end{lstlisting}

\begin{lstlisting}[title={\normalsize\bfseries\scshape The cdf $\bm{F}$ of $\bm{\CN_3(\bzero,\Sig_3(\brho))}$ distribution.}, basicstyle=\small, language=R]
F <- function(x,y,z,r12,r13,r23){
S <- matrix(c(1,r12,r13,r12,1,r23,r13,r23,1),3,3)
Pr <- pmvnorm(lower=rep(-Inf, 3), upper=c(x,y,z),
              mean=c(0,0,0), corr=S)
return(Pr)}
F <- Vectorize(F)
\end{lstlisting}

\begin{lstlisting}[title={\normalsize\bfseries\scshape Code for computing AUK when $\bm{\bX\sim \CN_3(\bzero,\Sig_3(\brho))}$.}, basicstyle=\small, language=R]
AUK <- function(r12,r13,r23){
auk <- integrate(Vectorize(function(x) {
        integrate(Vectorize(function(y) {
         integrate(function(z)
          {g(F(x,y,z,r12,r13,r23))*f(x,y,z,r12,r13,r23)},
          -Inf, Inf)$value }), -Inf,Inf)$value }), -Inf,Inf)$value
return(auk)}
\end{lstlisting}

\begin{lstlisting}[title={\normalsize\bfseries\scshape Code for computing $\bm{I}$ when $\bm{\bX\sim \CN_3\left(\bzero,\Sig_3(\rho)\right)}$.}, basicstyle=\small, language=R]
I <- function(rho){
Sign <- permutations(n=2,r=3,v=c(1,-1), repeats.allowed=TRUE)
auk <- rep(0,8)
for (j in 1:8){
  r12 = Sign[j,1]*Sign[j,2]*rho
  r13 = Sign[j,1]*Sign[j,3]*rho
  r23 = Sign[j,2]*Sign[j,3]*rho
  auk[j] <- AUK(r12, r13, r23)}
ind=sqrt(32/57)*sqrt(sum((auk-1/2)^2))
return(ind)}
\end{lstlisting}

\begin{lstlisting}[title={\normalsize\bfseries\scshape The ecdf $\bm{\pmb{\bbF}_n}$ of a data set $\bm{\BX_n=(\bX_1,\ldots,\bX_n)^T}$.}, basicstyle=\small, language=R]
Fn <- function(Xn){
 Fn1 <- function(i){emcdf(Xn, Xn[i,])}
 Fn1 <- Vectorize(Fn1)
 FFn <- Fn1(c(1:nrow(Xn)))
return(FFn)}
\end{lstlisting}

\begin{lstlisting}[title={\normalsize\bfseries\scshape The $\bm{\hAUK_n}$ estimator based on the data set $\bm{\BX_n=(\bX_1,\ldots,\bX_n)^T}$.}, basicstyle=\small, language=R]
AUK <- function(Xn){mean(g(Fn(Xn)))}
\end{lstlisting}

\begin{lstlisting}[title={\normalsize\bfseries\scshape Code for computing $\bm{I}$, via Monte Carlo, when $\bm{\bX\sim \CN_3\left(\bzero,\Sig_3(\rho)\right)}$.}, basicstyle=\small, language=R]
I <- function(N,n,rho){
 mu <- c(0,0,0)
 Sigma <- rho*matrix(rep(1,9),3,3)+(1-rho)*diag(3)
 I1 <- function(i){
  if (i==1){
  X <- mvrnorm(n, mu, Sigma)
  Sign <- permutations(n=2, r=3, v=c(1,-1), repeats.allowed=TRUE)
  auk1 <- function(j){AUK(X%*%diag(Sign[j,]))}
  auk1 <- Vectorize(auk1)
  auk <- auk1(1:8)
  A=sqrt(32/57)*sqrt(sum((auk-1/2)^2))
  return(A)}}
 I1 <- Vectorize(I1)
ind <- mean(I1(rep(1,N)))
return(ind)}
I <- Vectorize(I)
\end{lstlisting}

\begin{lstlisting}[title={\normalsize\bfseries\scshape The function $\bm{\varphi_3}$.}, basicstyle=\small, language=R]
phi <- function(t){1.61*t+4.513*t^2-13.607*t^3+12.235*t^4-3.751*t^5}
phi <- Vectorize(phi)
\end{lstlisting}

\begin{lstlisting}[title={\normalsize\bfseries\scshape The $\bm{\hI_n}$ and $\bm{\hI^*_n}$ estimators based on the data set $\bm{\BX_n=(\bX_1,\ldots,\bX_n)^T}$.}, basicstyle=\small, language=R]
Indices <- function(X){
 Sign <- permutations(n=2, r=3, v=c(1,-1), repeats.allowed=TRUE)
 auk1 <- function(j){AUK(X%*%diag(Sign[j,]))}
 auk1<- Vectorize(auk1)
 auk <- auk1(1:8)
 ind=sqrt(32/57)*sqrt(sum((auk-1/2)^2))
 ind.st <- phi(ind)
return(c(ind,ind.st))}
\end{lstlisting}

\begin{lstlisting}[title={\normalsize\bfseries\scshape Estimating the function $\bm{\varphi_d}$ using {\tt MonoPoly} R package.}, basicstyle=\small, language=R]
#let varrho and I be as in Step 1 of the algorithm
C <- SOSpol.fit(I, rho, deg.is.odd=TRUE, K=3, a=0, b=1,
                monotone = c("increasing", "decreasing"),
                trace = FALSE, plot.it = FALSE, type=2,
                control = monpol.control())$beta.raw
C <- C-C[1]
C <- C/sum(C)
phi.d <- function(t){evalPol(t,C)}
phi.d <- Vectorize(phi.d)
}
\end{lstlisting}

\begin{lstlisting}[title={\normalsize\bfseries\scshape Estimating the standard deviation $\bm{\sigma_{\vP}}$.}, basicstyle=\small, language=R]
Wrapper <- function(d){
  l <- length(d)
  Fn<- function(Un){
    Fn1<- function(i){
      emcdf(Un, Un[i,])
    }
    Fn1<- Vectorize(Fn1)
    FFn<- Fn1(c(1:nrow(Un)))
    return(FFn)
  }
  sigma_Pi<- function(r,n,d){
    AUK1<- function(i){
      if(i==1){
        X<- matrix(runif(2*n, min=0, max=1), n, 2)
        auk<- mean(pchisq(-2*log(Fn(X)), df=2*d, ncp=0, lower.tail=TRUE))
        return(auk)
      }
    }
    AUK1<- Vectorize(AUK1)
    AUK<- AUK1(rep(1,r))
    Zn<- sqrt(n)*(AUK-1/2)
    sigmaPi<- sqrt(var(Zn))
    sigma_Pi<- Vectorize(sigma_Pi)
    return(sigmaPi)
  }
  sPi <- sigma_Pi(1000, 50000, d)
}
#
start_time <- Sys.time()
Output<-sfLapply(Value.list, Wrapper)
R <- matrix(unlist(Output), nrow=length(d), ncol=1)
rownames(R) <-d
colnames(R) <-c("sigma_Pi")
end_time <- Sys.time()
running_time<-end_time - start_time
running_time
sfStop()
\end{lstlisting}

\begin{lstlisting}[title={\normalsize\bfseries\scshape A code to measure the speed-time of AUK and {\bf{d}}HSIC tests}, basicstyle=\small, language=R]
speed.times <- function(n){
x<- runif(n)
y<- runif(n)
X <- cbind(x,y)
StartTime.AUK <- Sys.time()
AUK <- mean(g(Fn(X)))
Z <- sqrt(n)*abs(AUK-1/2)/(sqrt(19/432))
EndTime.AUK <- Sys.time()
TimeTaken.AUK <- EndTime.AUK - StartTime.AUK
StartTime.dHSIC <- Sys.time()
dHSIC <- dhsic.test(list(x,y), alpha = 0.05, method = "permutation",
                    kernel = "gaussian", B = 1000, pairwise = FALSE,
                    bandwidth = 1, matrix.input = FALSE)
EndTime.dHSIC <- Sys.time()
TimeTaken.dHSIC <- EndTime.dHSIC - StartTime.dHSIC
Res=c(n,c(TimeTaken.AUK,TimeTaken.dHSIC))
R <- matrix(Res, nrow=1, ncol=3)
colnames(R) <- c("sample size n","AUK-time (in sec)","dHSIC-time (in sec)")
rownames(R) <- c("")
return(t(R))}
\end{lstlisting}
}
\end{appendices}

\end{document}